\documentclass[12pt]{amsart}
\usepackage{amssymb}
\usepackage{url}
\usepackage{graphicx}
\usepackage{tikz-cd}
\usetikzlibrary{babel}
\usetikzlibrary{matrix,arrows,decorations.pathmorphing}
\usepackage{mathtools}
\usepackage{mathrsfs}
\usepackage{comment}
\usepackage{stmaryrd}
\usepackage[colorlinks = true,
linkcolor = blue,
urlcolor  = blue,
citecolor = blue,
anchorcolor = blue]{hyperref}
\usepackage[all]{xy}

\allowdisplaybreaks

\newtheorem{theorem}[equation]{Theorem}

\newtheorem{lemma}[equation]{Lemma}
\newtheorem{proposition}[equation]{Proposition}
\newtheorem{corollary}[equation]{Corollary}

\newtheorem*{theorem:derhamisomorphism}{Theorem~\ref{T:derhamisomorphism}}
\newtheorem*{theorem:characterization}{Theorem~\ref{T:characterization}}

\theoremstyle{definition}
\newtheorem{definition}[equation]{Definition}
\newtheorem{example}[equation]{Example}

\newtheorem{remark}[equation]{Remark}

\numberwithin{equation}{section}

\newcommand{\FF}{\mathbb{F}}

\newcommand{\CC}{\mathbb{C}}

\newcommand{\zz}{\bold{z}}

\newcommand{\Cl}{\mathcal{C}\ell}

\DeclareMathOperator{\Aut}{Aut}

\DeclareMathOperator{\DR}{DR}

\DeclareMathOperator{\Ker}{Ker}
\DeclareMathOperator{\SL}{SL}
\DeclareMathOperator{\GL}{GL}
\DeclareMathOperator{\D}{D}
\DeclareMathOperator{\M}{M}
\DeclareMathOperator{\N}{N}
\DeclareMathOperator{\h}{H}
\DeclareMathOperator{\Sp}{Sp}

\DeclareMathOperator{\End}{End}

\DeclareMathOperator{\Frac}{Frac}
\DeclareMathOperator{\Gal}{Gal}

\DeclareMathOperator{\Coh}{Coh}
\DeclareMathOperator{\Spec}{Spec}
\DeclareMathOperator{\Spf}{Spf}
\DeclareMathOperator{\Tr}{Tr}
\DeclareMathOperator{\Nr}{Nr}
\DeclareMathOperator{\der}{d}
\DeclareMathOperator{\Cusps}{Cusps}
\DeclareMathOperator{\AlgCusps}{AlgCusps}

\DeclareMathOperator{\Lie}{Lie}

\DeclareMathOperator{\Ima}{Im}

\DeclareMathOperator{\Sym}{Sym}

\DeclareMathOperator{\sgn}{sgn}

\DeclareMathOperator{\Id}{Id}
\DeclareMathOperator{\TD}{TD}

\DeclareMathOperator{\im}{Im}

\newcommand{\Knr}{\widehat{K_{\infty}^{\text{nr}}}}

\newcommand{\inorm}[1]{{\lvert #1 \rvert}}

\begin{document}

\title[On nearly holomorphic Drinfeld modular forms]{On nearly holomorphic Drinfeld modular forms for admissible coefficient rings}

	\author{O\u{g}uz Gezm\.{i}\c{s}}
\address{Department of Mathematics, National Tsing Hua University, Hsinchu City 30042, Taiwan R.O.C.}
\email{gezmis@math.nthu.edu.tw}

\author{Sriram Chinthalagiri Venkata}
\address{University of Heidelberg, IWR, Im Neuenheimer Feld 205, 69120, Heidelberg, Germany}
\email{sriram.chinthalagiri@iwr.uni-heidelberg.de}

\date{\today}



\maketitle

\begin{abstract} Let $X$ be a smooth projective and geometrically irreducible curve over the finite field $\mathbb{F}_q$ with $q$ elements and $K$ be its function field. Let $\infty$ be a fixed closed point of $X$ and $A$ be the ring of functions regular away from $\infty$. In the present paper, by generalizing the previous work of Chen and the first author, we introduce the notion of nearly holomorphic Drinfeld modular forms for congruence subgroups of $\GL_2(K)$ as continuous but non-holomorphic functions on a certain subdomain of the Drinfeld upper half plane. By extending the de Rham sheaf to a compactification $\overline{M_I^2}$ of the Drinfeld moduli space $M_I^2$ parametrizing rank 2 Drinfeld $A$-modules with level $I$-structure over $K$-schemes, we also describe such forms algebraically as global sections of an explicitly described sheaf on $\overline{M_I^2}$ as well as construct a comparison isomorphism between analytic and algebraic description of them. Furthermore, we show the transcendence of special values of nearly holomorphic Drinfeld modular forms at CM points and relate them to the periods of CM Drinfeld $A$-modules. 
    
\end{abstract}
\section{Introduction}
\subsection{Background and motivation} Let $\mathbb{H}$ be the upper half plane and $\Gamma$ be a congruence subgroup of $\SL_2(\mathbb{Z})$. In a series of papers \cite{Shi75a,Shi75b, Shi77}, Shimura studied \textit{nearly holomorphic modular forms of weight $k$ and depth $r$ for $\Gamma$} which are smooth but non-holomorphic functions $f:\mathbb{H}\to \mathbb{C}$ described uniquely as
\[
f(z)=\sum_{i=0}^r\frac{f_i(z)}{\Ima(z)^i}
\]
for some holomorphic functions $f_0,\dots,f_r$ with $f_r\neq 0$, having certain growth conditions and satisfying
\[
f\Big(\frac{az+b}{cz+d}\Big)=(cz+d)^kf(z)
\]
for any $\gamma=\begin{pmatrix}
    a&b\\c&d
\end{pmatrix}\in \Gamma$ and $z\in \mathbb{H}$. For instance, one can consider the non-holomorphic Eisenstein series $\mathcal{G}_2$ of weight 2 given by the following Fourier expansion
\[
\mathcal{G}_2(z):=-\frac{1}{8\pi \Ima(z)\sqrt{-1} }-\frac{1}{24}+\sum_{\ell=1}^{\infty}\Bigg(\sum_{d>0, d|\ell}d\Bigg)e^{2\pi \ell z \sqrt{-1}}.
\]
Indeed, $\mathcal{G}_2$ is a nearly holomorphic modular form of weight 2 and depth one for $\SL_2(\mathbb{Z})$. Furthermore, it also serves as one of the generators of the $\mathbb{C}$-algebra of nearly holomorphic modular forms. In another direction, Shimura, generalizing the work of Maass, introduced a differential operator $\delta_k^r$ so that the image via $\delta_k^r$ of nearly holomorphic modular forms with certain rationality properties when evaluated at CM points, produces CM periods. For these aforementioned results and more details, we refer the reader to \cite[\S8, 12]{Shi07}. Using a geometric point of view, to study overconvergent elliptic modular forms, Urban revisited nearly holomorphic modular forms in his work \cite{Urb14} and constructed an explicit sheaf on the compactification of the moduli space of elliptic curves so that its global sections give rise to nearly holomorphic modular forms. We note that his method mainly relies on the Hodge decomposition of the sheaf of relative degree one de Rham cohomology of the universal elliptic curve.

In the present paper, we focus on the function field analogue of nearly holomorphic modular forms, namely nearly holomorphic Drinfeld modular forms. Our goal is to introduce a general picture for these objects by extending the work in \cite{CG23} which carried out the construction in \cite{Fra11} to the function field setting and study various aspects of them. More precisely, after describing them analytically as continuous but non-rigid analytic $\mathbb{C}_{\infty}$-valued functions on a subdomain of the Drinfeld upper half plane, using Hayes's theory of rank one Drinfeld $A$-modules and Drinfeld's observation on modular functions \cite[Prop.~9.3]{Dri74}, we obtain transcendence properties of their special values at CM points (Theorem \ref{T:MR2}). Furthermore, inspired by the work of Urban, we also describe them algebraically as global sections of a certain sheaf on the compactification of the Drinfeld moduli space and establish a comparison isomorphism between these two constructions (Theorem \ref{T:MR1}). As a by-product of our setting, we also obtain an explicit construction for the extension of the de Rham sheaf introduced  originally by Gekeler in \cite{Gek90} to the compactification of Drinfeld moduli spaces  so that the aforementioned algebraic description holds true. 
\subsection{Nearly holomorphic Drinfeld modular forms} Let $\mathbb{F}_q$ be the finite field with $q$ elements where $q$ is a positive power of a prime $p$. Let $X$ be a smooth projective, geometrically irreducible curve over $\mathbb{F}_q$ and denote by $K$ the function field of $X$. Let $\infty$ be a fixed closed point of $X$ and set $A$ to be the ring of functions regular away from $\infty$. Throughout this paper, we call such a ring \textit{an admissible coefficient ring}. As an illustration, if one considers $X=\mathbb{P}^1_{\mathbb{F}_q}$ and chooses $\infty$ to be the point at infinity, then $A$ becomes the polynomial ring $\mathbb{F}_q[\theta]$ for some variable $\theta$ over $\mathbb{F}_q$.  Note that any admissible coefficient ring is indeed a Dedekind domain. 

Let $\mathfrak{d}$ be the degree of $\infty$ over $\mathbb{F}_q$. For each $a\in A$, we define the $\infty$-adic norm $|a|:=q^{\deg(a)}$ where 
\[
\deg(a):=\mathfrak{d}\cdot (\text{the order of the pole of $a$ at $\infty$})
\]
and extend it canonically to $K$. Note that, for $a\in A$, $|a|$ is the cardinality of $A/aA$. We let $K_{\infty}$ be the completion of $K$ with respect to $|\cdot|$, which can be described as the formal Laurent series ring $\mathbb{F}_{q^\mathfrak{d}}((\pi_{\infty}))$ for a fixed uniformizer $\pi_{\infty}$ at $\infty$. We further set $\mathbb{C}_{\infty}$ to be the completion of an algebraic closure of $K_{\infty}$. 

Let $\overline{\mathbb{F}}_q$ be a fixed algebraic closure of $\mathbb{F}_q$ in $\mathbb{C}_{\infty}$. We consider $\widehat{K_{\infty}^{\text{nr}}}:=\overline{\mathbb{F}}_q((\pi_{\infty}))\subset \mathbb{C}_{\infty}$, which is the maximal unramified extension of $K_{\infty}$. We further define the Frobenius map $\sigma:\widehat{K_{\infty}^{\text{nr}}}\to \widehat{K_{\infty}^{\text{nr}}}$ given by 
\[
\sigma\Bigg(\sum_{i\geq i_0}a_i\pi_{\infty}^i\Bigg):=\sum_{i\geq i_0}a_i^{q^{\mathfrak{d}}}\pi_{\infty}^i, \ \ a_i\in \overline{\mathbb{F}}_q.\]
Observe that $\sigma$ is a continuous field automorphism of $\Knr$ stabilizing elements of $K_{\infty}$. Let $M$ be a field extension of $\Knr$ and $\varphi$ be a continuous automorphism of $\mathbb{C}_{\infty}$ that fixes $K_{\infty}$. We say that $\varphi$ is \textit{an extension of $\sigma$} if $\varphi|_{\Knr}=\sigma$. Let $M^{\varphi}:=\{z\in M \ \ | \varphi(z)=z\}$ and consider $\Omega^{\varphi}(M):=M\setminus M^{\varphi}$.  We note that, for each extension $M$ of $\Knr$, $\Omega^{\varphi}(M)$ lies in \textit{the Drinfeld upper half plane} $\Omega:=\mathbb{P}^1(\mathbb{C}_{\infty})\setminus \mathbb{P}^1(K_{\infty})$ which may be identified by the set $\mathbb{C}_{\infty}\setminus K_{\infty}$ (see \cite{Bos14, FvdP04} for more details on $\Omega$).

Throughout the present paper, we let $Y=\mathfrak{g}+\mathfrak{h}$
be a projective $A$-module of rank two embedded in $K^2$ by 
\begin{equation}\label{E:defofY}
Y=\mathfrak{g}(1,0)+\mathfrak{h}(0,1)\subset K^2
\end{equation}
for some fractional ideals $\mathfrak{g}$ and $\mathfrak{h}$ of $A$. We consider
\[
\Gamma_Y:=\GL(Y)=\Bigg\{\begin{pmatrix}
    a&b\\c&d
\end{pmatrix}\in \GL_2(K)| \ \ a,d\in A, \ \ ad-bc\in \mathbb{F}_q^{\times}, \ \ b\in \mathfrak{g}^{-1}\mathfrak{h}, \ \ c\in \mathfrak{g}\mathfrak{h}^{-1}\Bigg\}.
\]
Let $\Gamma$ be a congruence subgroup of $\Gamma_Y$ and let $M$ ($\varphi$ respectively) be an extension of $\Knr$ ($\sigma$ respectively). Generalizing the construction given in \cite[\S3.2]{CG23} to arbitrary admissible coefficient rings, we define \textit{a nearly holomorphic Drinfeld modular form $F$ of weight $k\in \mathbb{Z}$, type $m\in \mathbb{Z}/(q-1)\mathbb{Z}$ and depth $r\geq 0$ for $\Gamma$} to be a continuous function $F:\Omega^{\varphi}(M)\to \mathbb{C}_{\infty}$ that can be uniquely written as 
\begin{equation}\label{E:nearlholdef}
F(z)=\sum_{i=0}^{r}\frac{f_i(z)}{(z-\varphi(z))^i}
\end{equation}
for some rigid analytic functions $f_0,\dots,f_r$ with $f_r\neq 0$, having a certain growth condition and that also satisfies
\begin{equation}\label{E:funcequation}
F\Big(\frac{az+b}{cz+d}\Big)=(cz+d)^{k}\det(\gamma)^{-m}F(z)
\end{equation}
for each $\gamma=\begin{pmatrix}
    a&b\\c&d
\end{pmatrix}\in \Gamma$ and $z\in \Omega^{\varphi}(M) $. Following the analysis in \cite[\S3]{CG23}, we note that $F$ is not a rigid analytic function whereas it is equipped with the modularity condition as described in \eqref{E:funcequation}. This indeed motivates the notion of \textit{nearly holomorphic}. We refer the reader to \S4 for further details. 

Let $\mathcal{N}_{k}^{\leq r}(\Gamma)$ be the $\mathbb{C}_{\infty}$-vector space of nearly holomorphic Drinfeld modular forms of weight $k$, of any type and depth at most $r$ for $\Gamma$. Clearly, any Drinfeld modular form of weight $k$ for $\Gamma$ can be considered as an element of $\mathcal{N}_{k}^{\leq r}(\Gamma)$ for each $r\geq 0$ (see \S3 for more details on Drinfeld modular forms). In addition, if we let $E$ be \textit{the false Eisenstein series of Gekeler constructed from the $Y$ above}, one can consider the function $E_2:\Omega^{\varphi}(M)\to \mathbb{C}_{\infty}$ given by 
\begin{equation}\label{E:example}
E_2(z):=E(z)-\frac{1}{z-\varphi(z)}
\end{equation}
and show that it is indeed a nearly holomorphic Drinfeld modular form of weight two, type one and depth one for $\Gamma_Y$ as well as for each congruence subgroup of $\Gamma_Y$ (see Lemma \ref{L:E2}). Hence $E_2$ can be considered as a function field analogue of the non-holomorphic weight 2 Eisenstein series $\mathcal{G}_2$. We again refer the reader to \S4 for a detailed discussion on $E_2$.
\subsection{Special values at CM points}
In what follows, we introduce our first result which generalizes \cite[Thm.~6.2.17]{CG23} to the case of an arbitrary admissible coefficient ring. Let $\xi(\mathfrak{g}^{-1}\mathfrak{h})\in \mathbb{C}_{\infty}^{\times}$ be a period of the Drinfeld-Hayes $A$-module associated to the $A$-lattice $\mathfrak{g}^{-1}\mathfrak{h}$ (see \S2.2 for details on Drinfeld-Hayes $A$-modules). Let $t_{\Gamma}$ be a certain choice of a uniformizer for $\Gamma$ at the infinity cusp (see \eqref{E:unifdef} for its explicit definition). Due to the functional equation \eqref{E:funcequation}, $F$, as in \eqref{E:nearlholdef}, has a unique \textit{$t_{\Gamma}$-expansion} given by 
\[
F(z)=\sum_{j=0}^r\frac{1}{\xi(\mathfrak{g}^{-1}\mathfrak{h})^j(z-\varphi(z))^j}\sum_{i=0}^{\infty}a_{i,j}t_{\Gamma}^{i}(z)
\]
for some $a_{i,j}\in \mathbb{C}_{\infty}$ provided that the norm of $z$ is sufficiently large. We call $F$ \textit{arithmetic} if all $a_{i,j}$ lie in a fixed algebraic closure $\overline{K}$ of $K$ in $\mathbb{C}_{\infty}$. For any rigid analytic function $f$ on $\Omega^{\varphi}(M)$, we consider the Maass-Shimura operator $\delta_k^r$ defined, by 
\[
\delta_k^r(f):=\sum_{i=0}^r\binom{k+r-i}{i}\frac{\der^{r-i}f}{\xi(\mathfrak{g}^{-1}\mathfrak{h})^i(\Id-\varphi)^i}
\]
where $\der^{\ell}f$ is a constant multiple of the $\ell$-th hyperderivative of $f$ (see \S4 for more details). 

We call $z_0\in \Omega$ \textit{a CM point} if $K(z_0)$ is a quadratic extension of $K$ where the infinite place does not split. Furthermore, we emphasize that for each CM point, there exists an explicitly constructed field extension $M_{z_0}$ of $K_{\infty}^{\text{nr}}$ and an extension $\varphi_{z_0}$ of $\sigma$ as described above so that their evaluation at a nearly holomorphic Drinfeld modular form $F:\Omega^{\varphi_{z_0}}(M_{z_0})\to \mathbb{C}_{\infty}$ is well-defined (Lemma \ref{L:ext}).

Our first result, which will be restated as Theorem \ref{T:specvalue2} later, is described as follows.
\begin{theorem}\label{T:MR2} Let $z_0\in \Omega$ be a CM point and let $F\in \mathcal{N}_{k}^{\leq r}(\Gamma)$ be an arithmetic nearly holomorphic Drinfeld modular form as above. Then 
	\[
	F(z_0)=c\Bigg(\frac{w_{z_0}}{\xi(\mathfrak{g}^{-1}\mathfrak{h})}\Bigg)^{k} 
	\]
	for some $c\in \overline{K}$ and a period $w_{z_0}$ of a CM Drinfeld $A$-module. In particular, if $f$ is a Drinfeld modular form of weight $k$ for $\Gamma$, then 
	\[
	\delta^r_{k}(f)(z_0)=\tilde{c}\Bigg(\frac{w_{z_0}}{\xi(\mathfrak{g}^{-1}\mathfrak{h})}\Bigg)^{k+2r} 
	\]
	for some $\tilde{c}\in \overline{K}$. Furthermore, if $F(z_0)$ and $\delta_k^r(f)(z_0)$ are non-zero, then they are transcendental over $\overline{K}$.
\end{theorem}

The crucial point to prove Theorem \ref{T:MR2} relies on an analysis for special values of Drinfeld modular functions at CM points. More precisely, by Drinfeld \cite[Prop.~9.3]{Dri74}, there exists an explicit choice $\mathfrak{J}$ among Drinfeld modular functions for $\Gamma_Y$ so that the function field of the curve $\Gamma_Y\setminus \Omega$ may be realized as an algebraic extension of $\mathbb{C}_{\infty}(\mathfrak{J})$. Combining this with the work of Hayes \cite{Hay79} on Drinfeld-Hayes $A$-modules,  we obtain the algebraicity of the special values at CM points of Drinfeld modular functions for $\Gamma_Y$ having a particular rationality property (Proposition \ref{P:jinvariant}). We remark that this can be seen as a generalization of the result \cite[Satz~(4.3)]{Gek83} of Gekeler to the arbitrary admissible coefficient ring case. Using Proposition \ref{P:jinvariant} as well as implementing the strategy in \cite[\S6]{CG23}, we eventually obtain Theorem \ref{T:MR2}. 

\subsection{Algebraic description of nearly holomorphic Drinfeld modular forms} For any ideal $I$ of $A$, let $V(I)$ be the set of prime ideals of $A$ dividing $I$. Let $I\subset A$  be such that $|V(I)|\geq 2$ and  $\mathcal{M}^{2}_I$ be the moduli space parametrizing Drinfeld $A$-modules of rank two with a level $I$-structure defined over $A$-schemes   (see \cite[\S2]{Leh09} for more details). Due to Drinfeld \cite{Dri74}, we know that $\mathcal{M}^{2}_I$ is a smooth curve over $\Spec(A) \setminus V(I)$. He further constructs a  compactification $\overline{\mathcal{M}^{2}_I}$ of $\mathcal{M}^{2}_I$ over $\Spec(A)$, which is  smooth over $\Spec(A) \setminus V(I)$, by \textit{glueing} $\mathcal{M}^{2}_I$ with a finite disjoint union of formal schemes $\sqcup^{n_I}_{i=1} \mathfrak{M}_i$, where $n_I$ is an explicitly determined integer (see \cite[Chap.~5, Prop.~3.5]{Leh09}), $\mathfrak{M}_i$ are identical for all $i$ and each $\mathfrak{M}_i$ is equipped with a \textit{Tate-Drinfeld module} which could be seen as an analogue of Tate elliptic curves. We follow \cite{vdH03, Hat22} for our terminology. We also note  that the Tate-Drinfeld modules described in this paper are called \textit{universal Drinfeld module with bad reduction} in \cite[Chap. 5, \S2]{Leh09}. For more details on this method, we again refer the reader to \cite[Chap.~5]{Leh09}.  

Let $M^2_I$ be the generic fiber of $\mathcal{M}^{2}_I$, which parametrizes Drinfeld $A$-modules of rank two over $K$-schemes with a level $I$-structure. Since $M^2_I$ is a smooth curve over $K$, it admits a unique smooth compactification $\overline{M^2_I}$ which can be also realized by taking the generic fiber of $\overline{\mathcal{M}^{2}_I}$. For a field $L$, we denote by $\Spf(L[[X_i]])$ the affine formal scheme associated to the ideal $(X_i) \subset L[[X_i]]$ as defined in \cite[Chap.~4, \S1]{Leh09}, for $I=(X_i)$ and $R=L[[X_i]]$. Denoting $\mathcal{H}$ to be the ray class field over $K$ defining the moduli space parametrizing rank $1$ Drinfeld $A$-modules with a level $I$-structure over $K$-schemes and fixing an indeterminate $X_i$ for each $1\leq i \leq n_{I}$, we have that the generic fiber of the formal scheme $\mathfrak{M}_i$ is $\Spf(\mathcal{H}[[X_i]])$ and the Tate-Drinfeld modules are determined by certain maps $\mu_i:\Spec(\mathcal{H}((X_i))) \rightarrow M^2_I$. In particular, by pulling back the universal Drinfeld $A$-module $\mathbb{E}^{un}_I:=(\mathcal{L}^{un},\phi^{un})$ over $M_{I}^2$ via $\mu_i$, we obtain a Drinfeld $A$-module $(\phi_i,\lambda_i)$ of rank two over $\mathcal{H}((X_i))$ with a level $I$-structure. Moreover, the coefficients of $\phi_i$ indeed lie in $\mathcal{H}[[X_i]]$. By construction, the formal completion of  $\overline{M^2_I}$ at $\overline{M^2_I} \backslash M^2_I$ is isomorphic to $\sqcup ^{n_I}_{i=1} \Spf(\mathcal{H}[[X_i]])$ such that the induced map $\Spec(\mathcal{H}((X_i))) \rightarrow M^2_I$ is precisely $\mu_{i}$. We refer the reader to \S6 for more details. 

In what follows, we define the \textit{Hodge bundle} on $M^2_I$ to be the locally free sheaf  $\omega_{un}:=\Lie(\mathcal{L}^{un})^{\vee}$ of rank one. Let $\Gamma_{Y}(I)$ be the principal congruence subgroup of level $I$ (see \S2.6) and $M_Y$ be the connected component of $M^{2}_I \times_K \mathbb{C}_\infty$ associated to the class of $Y$ so that $M_Y(\mathbb{C}_\infty)=\Gamma_Y(I)\backslash \Omega$. Let $\omega_Y$   denote the restriction of $\omega_{un}$ to $M_Y$ after base change by $\mathbb{C}_\infty$. In a similar vein,  we denote by $\mu^{Y}_i:\Spec(\mathbb{C}_\infty((X_i))) \rightarrow M_Y$ the map $\mu_i$ after base change and restricted to $M_Y$. In Proposition \ref{u_expn_0}, using \cite[Thm.~4.16]{Boc02}, we show that each $\mu_i^Y$ corresponds to a unique cusp $b_i$ of $\Gamma_Y(I)$. Later on, this fact motivates us to call the map $\mu_i$ \textit{an algebraic cusp}.

Observe that $(\mu^Y_i)^\ast\omega_Y$ is a free sheaf of rank one generated by $dZ_i$ for some indeterminate $Z_i$. Then, inspired by the terminology used by Goss \cite[Def.~1.54]{Gos80} (see also \cite[pg.~35]{Hat21}), for any positive integer $k$ and $f \in H^0(M_Y,(\omega_Y)^{\otimes k})$, we define \textit{the $t$-expansion of $f \in H^0(M_Y,(\omega_Y)^{\otimes k})$ at $\mu^{Y}_i$} to be the unique Laurent series $P_f(X_i) \in \mathbb{C}_\infty((X_i))$ such that \[(\mu^Y_i)^\ast f  = P_f(X_i)(dZ_i)^{\otimes k}.\]

Next, we associate another unique Laurent series to $f \in H^0(M_Y,(\omega_Y)^{\otimes k})$. Let $\omega_Y^{an}$ be the analytification of  $\omega_Y$. 
Note that via a rigid analytification 
\[
H^0 (M_Y,(\omega_{Y})^{\otimes k}) \hookrightarrow H^0 (\Gamma_Y(I)\backslash \Omega,((\omega_Y)^{\otimes k})^{an}),
\]
the latter space may be identified with the space of weak Drinfeld modular forms of weight $k$ for $\Gamma_Y(I)$ (see \cite[\S1]{Gos80} and \cite[Lem.~10.6]{BBP21})). Therefore, for any given cusp $b_i$ of $\Gamma_{Y}(I)$, one can form a \textit{uniformizer $t_{b_i}$  at the cusp $b_i$ of $\Gamma_{Y}(I)$} so that there exists a unique Laurent series $Q_f$ in $t_{b_i}$ with coefficients in $\mathbb{C}_{\infty}$, which we call \textit{the $t_{b_i}$-expansion of $f$} (see \S3 and \S6 for further details). 

As the Tate elliptic curves are the objects encoding the $q$-expansion of classical modular forms, it is natural to ask whether the Tate-Drinfeld modules encode the $t_{b_i}$-expansions of Drinfeld modular forms. It turns out that one may indeed have an affirmative answer as observed by Goss \cite[Prop.~1.78]{Gos80}. More precisely, in Proposition \ref{u_expn_1}, we provide a certain link between these two expansions of $f$ in the following sense: $P_f$ has no principal part as a Laurent series in $X_i$  if and only if $Q_f$ has no principal part as a Laurent series in $t_{b_i}$.

\begin{remark} At this point, it is crucial to remark that in \cite[Prop.~1.78]{Gos80}, Goss provided a stronger relation between Tate-Drinfeld modules and $t_{b_i}$-expansion of weak Drinfeld modular forms. In particular, although he did not include much details, he stated that one may recover the $t_{b_i}$-expansion of $f$ via the substitution $X_i=t_{b_i}$ in its $t$-expansion at $\mu_i^Y$. In \S6, for the sake of completeness, we provide an explicit analysis on Tate-Drinfeld modules as well as their connection with weak Drinfeld modular forms to obtain Proposition \ref{u_expn_1}. We provide full details due to the lack of a good reference. Furthermore, although it is a weaker statement than \cite[Prop.~1.78]{Gos80}, our Proposition \ref{u_expn_1} is sufficient enough to deduce \cite[Thm.~1.79]{Gos80} as we will explain below. 
\end{remark} 

Let $\overline{\omega_{un}}$ be the unique extension of $\omega_{un}$ over $\overline{M^2_I}$, such that the formal completion of $\overline{\omega_{un}}$ at $\overline{M^2_I} \backslash M^2_I$ is given by $\oplus^{n_I}_{i=1} \mathcal{H}[[X_i]] dZ_i$. Let $\overline{M_Y}$ be the compactification of $M_Y$. We further denote by $\overline{\omega_Y}$ the resulting extension of $\omega_Y$ to $\overline{M_Y}$. Then, using the aforementioned link between $t$-expansion at $\mu_i^Y$ and $t_{b_i}$-expansion of elements of $H^0(M_Y,(\omega_{Y})^{\otimes k})$, one can form a natural isomorphism of $\mathbb{C}_{\infty}$-vector space $H^0(\overline{M_Y},(\overline{\omega_Y})^{\otimes k})$ and the $\mathbb{C}_{\infty}$-vector space of Drinfeld modular forms of weight $k$ with respect to $\Gamma_Y(I)$ (\cite[Thm.~1.79]{Gos80}). Again, we refer the reader to \S6 and \S8 for a thorough discussion.

The main strategy to describe nearly holomorphic Drinfeld modular forms algebraically relies on the previously described technique to realize Drinfeld modular forms as global sections of tensor powers of $\overline{\omega_{un}}$. We first consider the \textit{de Rham sheaf} $\mathbb{H}_{\DR,un}:=\mathbb{H}_{\DR}(\mathbb{E}^{un}_I)$ associated to $\mathbb{E}^{un}_I$ which is the locally free sheaf of rank two over $M_I^2$ introduced by Gekeler \cite{Gek90} (see \S2.4 and \S2.5 for more details). Let $\mathbb{H}_{\DR,Y}$ denote the restriction of $\mathbb{H}_{\DR,un}$ to $M_Y$ after base change by $\mathbb{C}_\infty$.  
Motivated by the classical theory of nearly holomorphic modular forms developed in \cite{Urb14}, for non-negative integers $k$ and $r$ so that $k\geq r$, we define the sheaf 
\[
\mathcal{H}^{r}_{k,Y}:=\Sym^r(\mathbb{H}_{\DR,Y}) \otimes (\omega_{Y})^{\otimes (k-r)}.
\]
Let $\h_{\DR}(\phi_i)$ be the de Rham module associated to the Tate-Drinfeld module $\phi_i$ (see \S2.4 for the explicit definition). From the Hodge decomposition of the de Rham modules, for each $1\leq i \leq n_{I}$, we have 
\[
\h_{\DR}(\phi_i) = \mathcal{H}((X_i))\eta_{i,1} \oplus \mathcal{H}((X_i))\eta_{i,2},
\] 
where $\eta_{i,1},\eta_{i,2}$ are explicitly determined biderivations for the Drinfeld $A$-module $\phi_i$ (Lemma \ref{L:basisderham}). Let $\overline{\mathbb{H}_{\DR,un}}$  be the unique extension of $\mathbb{H}_{\DR,un}$ over $\overline{M_I^2}$ such that its formal completion at  $\overline{M^2_I} \backslash M^2_I$ is given by $\mathcal{H}[[X_i]]\eta_{i,1} \oplus \mathcal{H}[[X_i]]\eta_{i,2}$ (see \S7.2 for the explicit construction).

Next, we carry the notion of $t$-expansion at $\mu_i^Y$ to weak nearly holomorphic Drinfeld modular forms of weight $k$ and depth less than or equal to $r$ for $\Gamma_Y(I)$ whose $\mathbb{C}_{\infty}$-vector space is denoted by $\mathcal{WN}_k^{\leq r}(\Gamma_Y(I))$ (see \S4 for their explicit definition). First we let $F\in H^0(M_Y, \mathcal{H}^r_{k,Y})$. After choosing a basis $\{\eta_{i,1},\eta'_{i,2}\}$ of $\h_{\DR}(\phi_i)$ that differs from $\{\eta_{i,1},\eta_{i,2}\}$ by an explicit unipotent matrix (see \S7.1), note that there exists an $(r+1)$-tuple of unique Laurent series $\{P^{(j)}_F(X_i)\}_{0 \leq j \leq r}$ such that 
    \[
 (\mu^{Y}_i)^\ast(f)=\sum^r_{j=0}P^{(j)}_F(X_i)(\eta_{i,1})^{\otimes (k-r+j)} \otimes (\eta'_{i,2})^{\otimes (r-j)}.
   \] 
We call the tuple $\{P^{(j)}_F(X_i)\}_{0 \leq j \leq r}$   \textit{the $t$-expansion of $F$ at the cusp $\mu^{Y}_i$}. On the other hand, via analytification, each $F\in H^0(M_Y, \mathcal{H}^r_{k,Y})$ may be considered as an element of $\mathcal{WN}_k^{\leq r}(\Gamma_Y(I))$ (Theorem \ref{T:weak}). Hence $F$ also admits a $t_{b_i}$-expansion which is given by an $(r+1)$-tuple of unique Laurent series $\{Q^{(j)}_F(t_{b_i})\}_{0 \leq j \leq r}$ (see Definition \ref{D:1}). Analogous to Proposition \ref{u_expn_1}, in Lemma \ref{L:nhf_t_expn}, we obtain the following link between $t$-expansion of $F$ at $\mu_i^Y$ and $t_{b_i}$-expansion of $F$: For all $0\leq j \leq r$,  $P^{(j)}_F(X_i)$ has no principal part as a Laurent series in $X_i$ if and only if $Q^{(j)}_F(t_{b_i})$ has no principal part as a Laurent series in $t_{b_i}$.

Let $\overline{\mathbb{H}_{\DR,Y}}$ be the corresponding extension of $\mathbb{H}_{\DR,Y}$ to $\overline{M_Y}$. We finally denote by $\overline{\mathcal{H}^r_{k,Y}}$ the extension of $\mathcal{H}^r_{k,Y}$ to $\overline{M_Y}$ which is given by
\[
\overline{\mathcal{H}_{k,Y}^r}:=\Sym^{r}(\overline{\mathbb{H}_{\DR,Y}}) \otimes \overline{\omega}^{\otimes (k-r)}_{Y}.
\]  
Using the aforementioned link between $t$-expansion at $\mu_i^Y$ and $t_{b_i}$-expansion of elements of $H^0(M_Y, \mathcal{H}^r_{k,Y})$, we obtain our next result (restated as Theorem \ref{T:main1} later) which recovers the space of nearly holomorphic Drinfeld modular forms analogous to \cite[Prop.~2.2.3]{Urb14}. We also remark that it also leads to the result of Goss \cite[Thm.~1.79]{Gos80}, alluded to before, on the space of Drinfeld modular forms by setting $r=0$.

\begin{theorem}\label{T:MR1}
	Let $I\subset A$  be such that $|V(I)|\geq 2$. Then there exists a natural isomorphism of $\mathbb{C}_\infty$-vector spaces
	\[H^{0}(\overline{M_Y},\overline{\mathcal{H}^r_{k,Y}}) \cong \mathcal{N}^{\leq r}_k(\Gamma_Y(I)).\]
\end{theorem}

\begin{remark}
      If $|V(I)|\geq 1$, then the moduli problem for Drinfeld $A$-modules of rank $r$ over $A[I^{-1}]$-schemes with level $I$-structure is still a fine moduli problem (see for instance \cite[Rem.~1.13]{Gos80}). In particular, the moduli space $M^2_I$ over $\Spec(K)$ makes sense even with $|V(I)|=1$. One can naturally ask whether Theorem \ref{T:MR1} holds true in this case. The proof of Theorem \ref{T:MR1} heavily relies on an analysis of Tate-Drinfeld modules for the case $|V(I)|\geq 2$ and the lack of an analogous theory for the case $|V(I)|=1$ refrains us from generalizing our result. We expect Theorem \ref{T:MR1} to hold true also in this general case,  once we develop a theory of Tate-Drinfeld modules in the case $|V(I)|=1$ to compactify $M^2_I$ analogous to what we will discuss in \S6. We hope to come back to this problem in the near future.
\end{remark}

\begin{remark} The authors expect that the results in this paper can be pursued further to gain a broader understanding in the theory of Drinfeld modular forms of arbitrary rank and their special values. We state here three directions one can pursue from the present work.

\begin{itemize}
    \item[(i)]  In the classical case, another way to obtain transcendence of special values of nearly holomorphic modular forms at CM points is to use the geometric description of such forms, as done in \cite[\S2.6]{Urb14}. In the function field setting, a similar approach has been employed by Ayotte in \cite{Ayo23} to obtain transcendence of special values of Drinfeld modular forms. It would be interesting to obtain Theorem \ref{T:MR2} by using this method.
\item[(ii)] Another interesting future research direction concerning our newly defined objects is the action of Hecke operators on them. In the classical setting, there have been results (see for example \cite[\S2]{BJTX12}) already established in this direction and one wonders whether a similar point of view may be pursued in the function field setting, namely defining a suitable Hecke action on both analytic and algebraic side (see also \cite[\S5]{BVdV25} for an analysis of Hecke operators in the setting of Drinfeld quasi-modular forms which are closely related to nearly holomorphic Drinfeld modular forms). We hope that this enables one to show the algebraicity of  eigenvalues of these Hecke operators due to the construction of the locally free sheaf $\overline{\mathcal{H}^{r}_k}$ on the compactification of the Drinfeld moduli space.     
    \item[(iii)] Let $n \geq 2$ and for a non-zero ideal $I$ of $A$, consider the fine moduli space $M^n_I$ parametrizing Drinfeld $A$-modules of rank $n$ with level $I$-structures over $K$-schemes. It admits a universal Drinfeld $A$-module $\mathbb{E}^{un,n}_I$. Similar to the rank two case discussed above, one can consider the following locally free sheaves of rank $n$ and rank one respectively, \[\mathbb{H}^{(n)}_{\DR,un}:=\mathbb{H}_{\DR}(\mathbb{E}^{un,n}_I),\ \omega^{(n)}_{un}:=\Lie(\mathbb{E}^{un,n}_I)^{\vee}.\] Consequently, for non-negative integers $k$ and $r$, one can define the coherent sheaf \[\mathcal{H}^{r,(n)}_k:=\Sym^r(\mathbb{H}^{(n)}_{\DR,un}) \otimes_{\mathcal{O}_{M^n_I}} (\omega^{(n)}_{un})^{\otimes (k-r)}.\]
   At this point, some natural questions arise:
    \begin{itemize}
        \item[(a)] Let $\Omega^{n-1}$ be the $(n-1)$-dimensional Drinfeld upper half plane and let $P \subset K^n$ be a rank $n$ projective $A$-module, so that $\Gamma_P(I) \backslash \Omega^{n-1} \hookrightarrow (M^n_I \times_K \mathbb{C}_\infty)^{an}$ is a connected component and we have a sequence of maps $\pi:\Omega^{n-1} \rightarrow \Gamma_P(I) \backslash \Omega^{n-1} \hookrightarrow (M^n_I \times_K \mathbb{C}_\infty)^{an}.$ Can one give a suitable description of a function $F$ on $\Omega^{n-1}$ arising as sections of $\pi^\ast((\mathcal{H}^{r,(n)}_{k,\mathbb{C}_\infty})^{an})$ similar to the description in \eqref{E:nearlholdef}? Furthermore, what would be a higher dimensional analogue of the space $\Omega^{nr}$? 
        \item[(b)] Consider the same $F$ as above. Does it admit a description similar to that in Proposition \ref{P:str}? In particular, what is the analogue of $E_2$ in this general case? We expect to obtain an affirmative answer to this question by using the false Eisenstein series of higher rank studied in the works of Chen and the first author \cite{CG21,CG22}.
    \end{itemize}
\end{itemize}
\end{remark}

\subsection{Outline of the paper}The outline of the present paper can be described as follows. In \S2, we introduce the necessary background on Drinfeld-Hayes $A$-modules, the de Rham module associated to Drinfeld $A$-modules as well as the main properties of coherent sheaves defined on the quotient space $\Gamma\setminus \Omega$. In \S3, we analyze Drinfeld modular forms and Drinfeld modular functions. In \S4, we continue by describing the basic properties of nearly holomorphic Drinfeld modular forms and discuss the fundamental example $E_2$ defined in \eqref{E:example}. In \S5, we analyze the special values of nearly holomorphic Drinfeld modular forms at CM points and prove Theorem \ref{T:MR2}. In \S6, we discuss  \textit{Tate-Drinfeld modules} (TD modules) which are crucial for extending sheaves on Drinfeld moduli spaces to their cusps. More precisely, we will compare the analytic and algebraic description of TD modules and describe the relationship between these two constructions.  In \S7, using our analysis in \S6, we extend the de Rham sheaf to $\overline{M_I^2}$ and  analyze the de Rham cohomology of Drinfeld $A$-modules defined over $\Omega$ as well as over its arithmetic quotients. Finally in \S8, we provide a proof for Theorem \ref{T:MR1}. 

\subsection*{Acknowledgments} The authors are indebted to Gebhard B\"{o}ckle for insightful discussions and  his careful reading of an earlier version of the paper as well as for his valuable comments. The authors are grateful to Florian Breuer, Shin Hattori and Yen-Tsung Chen for many useful suggestions and fruitful discussions. The authors also would like to thank anonymous referees for their careful reading of the manuscript and for their valuable suggestions that make the presentation of the results clearer. The first author was supported by NSTC Grant 113-2115-M-007-001-MY3. The first author (partially) and the second author acknowledge support by Deutsche Forschungsgemeinschaft (DFG) through CRC-TRR 326 `Geometry and Arithmetic of Uniformized Structures', project number 444845124. Data sharing is not applicable for this article because no datasets were generated or analyzed during the preparation of the manuscript.
\subsection*{Compliance with Ethical Standards} The authors would like to declare that they have no conflicts of interest.

\section{Preliminaries and Background}
The main goal of this section is to overview Drinfeld $A$-modules over $K$-schemes and the de Rham cohomology attached to them as well as Drinfeld-Hayes $A$-modules. Furthermore, we will briefly describe the theory of coherent sheaves on certain rigid analytic spaces which will be used in \S7.

We will mainly use \cite{Gek86,Boc02,Leh09} to bring materials together for preliminary sections and describe the notational differences when there is any.

\subsection{Drinfeld $A$-modules over an $A$-algebra} Let $B\subseteq \mathbb{C}_{\infty}$ be an $A$-algebra containing $A$.  We define the non-commutative power series ring $B[[\tau]]$ subject to the condition 
\[
\tau c=c^q\tau, \ \ c\in B
\] 
and let $B[\tau]\subset B[[\tau]]$ be the subring of polynomials in $\tau$. There exists an action of $B[\tau]$ on $B$ given by 
\[
u\cdot z:=u(z):=\sum_{i\geq 0}a_iz^{q^i}
\]
for each $u=\sum_{i\geq 0} a_i\tau^i\in B[\tau]$ and $z\in B$.

By \textit{an $A$-field}, we mean a field $L$ equipped with a ring homomorphism $i:A\to L$. If $\Ker(i)=\mathfrak{p}$ for some prime ideal $\mathfrak{p}$ of $A$, then we say $L$ has \textit{characteristic $\mathfrak{p}$}. If $\Ker(i)=(0)$, then we say $L$ has \textit{generic characteristic}. Clearly, any subfield of $\mathbb{C}_{\infty}$ containing $K$ is an $A$-field of generic characteristic. We also define the map $\partial:L[\tau]\to L$ sending each $u=\sum_{i=0}^mu_i\tau^i$ to the $\tau^{0}$-th coefficient $u_0$ of $u$.
\begin{definition} \label{D:11}
\begin{itemize}

\item[(i)] A \textit{Drinfeld $A$-module $\phi$ of rank $r\geq 1$ defined over an $A$-field $L$} is an $\mathbb{F}_q$-algebra homomorphism $\phi:A\to L[\tau]$ given by 
\[
\phi_{a}:=\phi(a):=i(a)+a_1\tau+\cdots +a_{r\deg(a)}\tau^{r\deg(a)}
\]
so that $a_{r\deg(a)}\neq 0$ and $\partial\circ \phi=i$. When $L$ has generic characteristic, we call $a_i$ \textit{the $i$-th coefficient of $\phi$}.
\item[(ii)] Let $L\subseteq \mathbb{C}_{\infty}$ be a field containing $K$. A homomorphism between Drinfeld $A$-modules $\phi$ and $\phi'$ (over $L$) is given by an element $u\in L[\tau]$ satisfying
\[
u \phi_a=\phi'_a u
\]
for each $a\in A$. Moreover, we call $\phi$ and $\phi'$ \textit{isomorphic} if $u\in L^{\times}$. We further denote \textit{the set of endomorphisms of $\phi$} by 
\[
\End(\phi):=\{u\in \mathbb{C}_{\infty}[\tau]\ \ | \ \ \phi_au=u\phi_a, \ \ a\in A\}.
\]
By Drinfeld \cite[\S2]{Dri74}, we know that $\End(\phi)$ is a commutative and projective $A$-module of rank less than or equal to $r$. Furthermore, we call $\phi$ a \textit{CM Drinfeld $A$-module} if $\End(\phi)$ has projective rank $r$ as an $A$-module.
\end{itemize}
\end{definition}

Let $\Lambda\subset \mathbb{C}_{\infty}$ be a projective $A$-module of rank $r$. We call $\Lambda$ \textit{an $A$-lattice of rank $r$} if its intersection with any ball of finite radius is finite. We further define \textit{a morphism} between $A$-lattices $\Lambda_1$ and $\Lambda_2$ to be an element $c\in \mathbb{C}_{\infty}^{\times}$ satisfying $c\Lambda_1\subseteq \Lambda_2$. Furthermore, we call $\Lambda_1$ and $\Lambda_2$ \textit{isomorphic} if $c\Lambda_1=\Lambda_2$ for some $c\in \mathbb{C}_{\infty}^{\times}$. 
For any $A$-lattice $\Lambda$, we also define its exponential function $\exp_{\Lambda}$ by 
\[
\exp_{\Lambda}(z):=z\prod_{\lambda\in \Lambda\setminus \{0\}}\Big(1-\frac{z}{\lambda}\Big).
\]

For a Drinfeld $A$-module $\phi$ defined over a field $L$ of generic characteristic, there exists $\exp_{\phi}=\sum_{i\geq 0}\beta_i\tau^i\in \mathbb{C}_{\infty}[[\tau]]$ uniquely defined by the conditions $\beta_0=1$ and 
\begin{equation}\label{E:funceq}
\exp_{\phi}  a=\phi_a \exp_{\phi}
\end{equation}
in $\mathbb{C}_{\infty}[[\tau]]$ for each $a\in A$. Moreover, it induces an entire function $\exp_{\phi}:\mathbb{C}_{\infty}\to \mathbb{C}_{\infty}$ given by
\[
\exp_{\phi}(z)=\sum_{i\geq 0}\beta_iz^{q^i}.
\]
We note that $\Ker(\exp_{\phi})$ is an $A$-lattice of rank $r$. We call each non-zero element in $\Ker(\exp_{\phi})$ \textit{a period of $\phi$}.

By the analytic uniformization of Drinfeld $A$-modules \cite[\S3]{Dri74}, the category of Drinfeld $A$-modules of rank $r$ defined over $\mathbb{C}_{\infty}$ is equivalent to the category of $A$-lattices of rank $r$. In particular, each $A$-lattice $\Lambda$ of rank $r$ corresponds to a unique Drinfeld $A$-module $\phi^{\Lambda}$ and its exponential function is $\exp_{\Lambda}$. Furthermore, each Drinfeld $A$-module of rank $r$ gives rise to the $A$-lattice $\Lambda=\Ker(\exp_{\phi})$ and hence $\exp_{\phi}=\exp_{\Lambda}$. 

We further set $\End(\Lambda):=\{c\in \mathbb{C}_{\infty}| \ \ c\Lambda\subseteq \Lambda\}$. Then there exists a ring isomorphism between $\End(\phi)$ and $\End(\Lambda)$ sending $u\in \End(\phi)$ to its constant term (see \cite[Thm.~13.25]{Ros02}).

\begin{remark}  Following \cite{Hay79}, a subring $R\subseteq K$ containing $1$ and whose fraction field is $K$ is \textit{an order for $A$}. Due to Drinfeld, we know that there exist non-trivial embeddings of $R$ into $B[\tau]$ whenever $B$ is an algebraically closed field. Hence, a theory of Drinfeld $R$-modules is valid by simply replacing $A$ with an order for $A$ in the above description and we refer the reader to \cite[\S1, 2, 4, 5]{Hay79} for further details. We note that such a theory will be used in the proof of Proposition \ref{P:jinvariant}. 
    
\end{remark}

\subsection{Drinfeld-Hayes $A$-modules} Our goal in this subsection is to introduce Drinfeld-Hayes $A$-modules which will be essential later on for the study of rationality properties of nearly holomorphic Drinfeld modular forms. Our exposition and notation are based on  \cite[Chap.~II, IV]{Gek86}. One can also refer to \cite{Hay79, Hay92} for additional details.

Consider a Drinfeld $A$-module $\rho$ of rank $r$ over a field $L$ containing $K$. Let $\mathfrak{b}$ be an integral ideal of $A$ and $I_{\rho,\mathfrak{b}}$ be the left ideal in $L[\tau]$ generated by $\rho_b$ for any $b\in \mathfrak{b}$. Since $L[\tau]$ is a left principal ideal domain \cite[Cor.~1.6.3]{Gos96}, there exists a unique $\rho_{\mathfrak{b}}\in L[\tau]$, which is monic, such that 
$
I_{\rho,\mathfrak{b}}=L[\tau]\rho_{\mathfrak{b}}.
$
Since the right multiplication of $\rho_{\mathfrak{b}}$ with $\rho_x$, for any $x\in A$, also lies in $I_{\rho,\mathfrak{b}}$, there exists a unique element $(\mathfrak{b}*\rho)_x\in L[\tau]$ such that 
\begin{equation}\label{E:isog}
(\mathfrak{b}*\rho)_x \rho_{\mathfrak{b}}=\rho_{\mathfrak{b}}\rho_x.
\end{equation}
Note that the map $(\mathfrak{b}*\rho):A \rightarrow L[\tau]$ forms a Drinfeld $A$-module of rank $r$ over $L$ and  it is indeed the unique Drinfeld $A$-module isogenous to $\rho$ via $\rho_{\mathfrak{b}}$ \cite[\S5]{Hay92}. 

\begin{definition}Let $U_{\infty}^{(1)}$ be the group of 1-units in $K_{\infty}$. \textit{A sign function} is a map $\sgn:K_{\infty}\to \mathbb{F}_{q^{\mathfrak{d}}}$ satisfying 
	\begin{itemize}
		\item[(i)] $\sgn(xy)=\sgn(x)\sgn(y), \ \ x,y\in K_{\infty}$,
		\item[(ii)] $\sgn(x)=1, \ \ x\in  U_{\infty}^{(1)}$,
		\item[(iii)] $\sgn(x)=x, \ \ x\in \mathbb{F}_{q^{\mathfrak{d}}}$.
	\end{itemize}
\end{definition}
Note that there are $(q^{\mathfrak{d}}-1)$-many sign functions. Throughout this paper, we fix a sign function, denoted by $\sgn$, and we assume that the uniformizer $\pi_{\infty}$ at $\infty$ maps to 1 under $\sgn$. 

Now we aim to introduce a certain choice of Drinfeld $A$-modules of rank one for each equivalence class in the class group $\Cl(A)$ of $A$. Firstly, consider the map $\mathfrak{L}:L[\tau]\to L$ sending each $u=\sum_{i=0}^mu_i\tau^i$ with $m\geq 0$ to its leading coefficient $u_m$. 
\begin{definition}A Drinfeld $A$-module $\rho$ of rank one is called \textit{a Drinfeld-Hayes $A$-module} if, for all $a\in A$, $\mathfrak{L}(\rho_a)=\psi\circ \sgn(a)$ for some $\psi\in \Gal(\mathbb{F}_{q^{\mathfrak{d}}}/\mathbb{F}_q)$. 
\end{definition}

By the seminal work of Hayes \cite[\S8]{Hay79} (see also \cite[\S15]{Hay92}), we know that each Drinfeld $A$-module of rank one is isomorphic to a Drinfeld-Hayes $A$-module defined over \textit{the Hilbert class field $H$ of $K$}, the maximal unramified extension of $K$ in which $\infty$ splits completely \cite[Chap.~IV, Cor.~2.11]{Gek86} (see also \cite[\S7.4]{Gos96}). We further call an $A$-lattice $\Lambda$ \textit{special} if $\phi^{\Lambda}$ is a Drinfeld-Hayes $A$-module defined over $H$. As an immediate consequence of Hayes's result, each equivalence class of rank one $A$-lattices contains special $A$-lattices which are conjugate by the elements of $\mathbb{F}_{q^{\mathfrak{d}}}^{\times}$ (\cite[Prop.~13.1]{Hay92}). 

Let $\Lambda\subset \mathbb{C}_{\infty}$ be an $A$-lattice of rank one isomorphic to a fractional ideal $\mathfrak{a}$ and let $\Lambda^{(\mathfrak{a})}$ be a fixed special $A$-lattice in the equivalence class of $\mathfrak{a}$ with the corresponding Drinfeld-Hayes $A$-module $\rho^{\Lambda^{(\mathfrak{a})}}$ defined over $H$. We will  denote $\rho^{\Lambda^{(\mathfrak{a})}}$ by $\rho^{(\mathfrak{a})}$ throughout the paper to ease the notation. We further define $\xi(\Lambda)\in \mathbb{C}_{\infty}^{\times}$ so that 
\[\label{xi}
\xi(\Lambda)\Lambda=\Lambda^{(\mathfrak{a})}.
\]
Since the elements in $\mathbb{F}_{q}^{\times}$ form the automorphism group of $\rho^{(\mathfrak{a})}$, $\xi(\Lambda)$ is determined uniquely up to multiplication by a  $(q-1)$-st root of unity. 
\begin{example} Let $A=\mathbb{F}_q[\theta]$. Fix a $(q-1)$-st root of $-\theta$ and define 
	\[
	\tilde{\pi}:=\theta(-\theta)^{1/(q-1)}\prod_{i=1}^{\infty}\Big(1-\theta^{1-q^i}\Big)^{-1}\in \CC_{\infty}^{\times}.
	\]
	It is known that $\tilde{\pi}A=\Lambda^{(1)}$ is a special $A$-lattice with the corresponding Drinfeld-Hayes $A$-module $C:=\rho^{(1)}$, known as \textit{the Carlitz module}, given by $C_{\theta}:=\theta+\tau$. Furthermore, for any $A$-lattice of rank one $\Lambda=\beta A$ for some $\beta\in \mathbb{C}_{\infty}^{\times}$, we have $\xi(\Lambda)=\tilde{\pi}\beta^{-1}$.
\end{example}

\begin{lemma}[{Yu, \cite{Yu86}}] \label{L:11} Let $\mathfrak{a}$ be a fractional ideal of $A$. Then $\xi(\mathfrak{a})$ is transcendental over $K$.
\end{lemma}
\begin{proof} For completeness, we recall some steps of the proof. Assume to the contrary that $\xi(\mathfrak{a})$ is algebraic over $K$. Let $c\in \mathfrak{a}\setminus \{0\}$. Since $\rho^{(\mathfrak{a})}$ is defined over $H\subset \overline{K}$, by \cite[Thm.~5.1]{Yu86}, we see that $\exp_{\rho^{(\mathfrak{a})}}(\xi(\mathfrak{a})c)$ is transcendental over $K$. However, by definition, $\exp_{\rho^{(\mathfrak{a})}}(\xi(\mathfrak{a})c)=0$: a contradiction.
\end{proof}

Let $\mathfrak{c}$ and $\mathfrak{u}$ be integral ideals of $A$. Consider the Drinfeld $A$-module $\mathfrak{c}*\rho^{(\mathfrak{u})}$. Since the leading coefficient of $\rho_{\mathfrak{c}}^{(\mathfrak{u})}$ is one, by \eqref{E:isog}, $\mathfrak{c}*\rho^{(\mathfrak{u})}$ is a Drinfeld-Hayes $A$-module over $H$. Moreover, comparing the $A$-lattices corresponding to $\rho^{(\mathfrak{c}^{-1}\mathfrak{u})}$ and $\mathfrak{c}*\rho^{(\mathfrak{u})}$ by using \cite[Thm.~8.14]{Hay92}, we see that $\rho^{(\mathfrak{c}^{-1}\mathfrak{u})}$ is isomorphic to $\mathfrak{c}*\rho^{(\mathfrak{u})}$ over $H$. Let $\mathcal{J}(\mathfrak{c},\mathfrak{u})\in H^{\times}$ be such an isomorphism, which is uniquely determined up to multiplication by a $(q-1)$-st root of unity (It is denoted by $\Theta(\mathfrak{c},\mathfrak{u})$ in \cite[Chap.~IV, (5.2)]{Gek86}).  Then by \cite[Chap.~IV, Prop.~5.4(i)]{Gek86}, we have 
\begin{equation}\label{E:relxi}
\xi(\mathfrak{c}^{-1}\mathfrak{u})=\mathcal{J}(\mathfrak{c},\mathfrak{u})\partial(\rho_{\mathfrak{c}}^{(\mathfrak{u})})\xi(\mathfrak{u}).
\end{equation}

\begin{lemma}\label{L:xi} Let $\mathfrak{a}$ and $\mathfrak{b}$ be fractional ideals of $A$. Then
\[
\xi(\mathfrak{a})=\alpha\xi(\mathfrak{b})
\]
for some $\alpha\in H^{\times}$.
\end{lemma}
\begin{proof} By comparing both $\xi(\mathfrak{a})$ and $ \xi(\mathfrak{b})$ with $\xi((1))$, we may assume that $\mathfrak{b}=(1)$. Now, since $\mathfrak{a}$ is a fractional ideal, we have $\xi(\mathfrak{a})=\mathfrak{k}\xi(\tilde{\mathfrak{a}})$  for a unique choice of $\mathfrak{k}\in A\setminus\{0\}$ satisfying $\mathfrak{k}\mathfrak{a}=\tilde{\mathfrak{a}}$ for an integral ideal $\tilde{\mathfrak{a}}$. On the other hand, applying \eqref{E:relxi} by choosing $\mathfrak{c}=\mathfrak{u}=\tilde{\mathfrak{a}}$, it follows that $\xi((1))=\alpha_1\xi(\tilde{\mathfrak{a}})$ for some $\alpha_1\in H^{\times}$.
\end{proof}

For any fractional ideal $\mathfrak{a}$, we consider $\mathfrak{a}$ as an $A$-lattice. We choose $\xi(\mathfrak{a})$ uniquely up to a multiple of $\mathbb{F}_q^{\times}$ and let
\begin{equation}\label{E:unifatideal}
t_{\mathfrak{a}}(z):=\exp_{\rho^{(\mathfrak{a})}}(\xi(\mathfrak{a})z)^{-1}=\exp_{\xi(\mathfrak{a})\mathfrak{a}}(\xi(\mathfrak{a})z)^{-1}, \ \ z\in \mathbb{C}_{\infty}\setminus \mathfrak{a}.
\end{equation}

We finish this subsection with our next proposition.
\begin{proposition}\label{P:expression} Let $\mathfrak{a}_1,\dots,\mathfrak{a}_n$ be non-zero fractional ideals of $A$. Then there exists an integral ideal $\mathfrak{m}\subseteq A$ such that, for each $1\leq i \leq n$, $t_{\mathfrak{a}_i}$ may be written as a power series in $t_{\mathfrak{m}}$ whose coefficients are in $H$.
\end{proposition}
\begin{proof} Note that there exists an element $c_i\in A\setminus\{0\}$ so that $c_i \mathfrak{a}_i$ is an integral ideal of $A$, say $\widetilde{\mathfrak{a}_i}$ and so, as in the proof of Lemma \ref{L:xi}, we have $\xi(\mathfrak{a}_i)=c_i\xi(\widetilde{\mathfrak{a}_i})$. 
	Since $\mathfrak{a}_i$ and $\widetilde{\mathfrak{a}_i}$ are in the same equivalence class and $\rho^{(\mathfrak{a}_i)}$ has the associated $A$-lattice $ \xi(\mathfrak{a}_i)\mathfrak{a}_i$, using the functional equation in \eqref{E:funceq}, we obtain 
    \begin{multline}\label{E:unifchange01}
        t_{\mathfrak{a}_i}(z)=\exp_{\rho^{(\mathfrak{a}_i)}}(\xi(\mathfrak{a}_i)z)^{-1}=\exp_{\rho^{(\widetilde{\mathfrak{a}_i})}}(\xi(\mathfrak{a}_i)z)^{-1}\\=\exp_{\rho^{(\widetilde{\mathfrak{a}_i})}}(c_i\xi(\widetilde{\mathfrak{a}_i})z)^{-1}=\frac{1}{\rho^{(\widetilde{\mathfrak{a}_i})}_{c_i}(\exp_{\rho^{(\widetilde{\mathfrak{a}_i})}}(\xi(\widetilde{\mathfrak{a}_i})z))}=\frac{1}{\rho^{(\widetilde{\mathfrak{a}_i})}_{c_i}(t_{\widetilde{\mathfrak{a}_i}}(z)^{-1})}.
	\end{multline}
	 Thus, by \eqref{E:unifchange01}, we have 
	\begin{equation}\label{E:algexpansion01}
	t_{\mathfrak{a}_i}=\frac{1}{a_{\deg(c_i)}}t_{\widetilde{\mathfrak{a}}_i}^{q^{\deg(c_i)}}+O(t_{\widetilde{\mathfrak{a}}_i}^{q^{\deg(c_i)}})
	\end{equation}
	where $a_{\deg(c_i)}$ is the leading coefficient of $\rho^{(\widetilde{\mathfrak{a}}_i)}_{c_i}\in H[\tau]$ and $O(t_{\widetilde{\mathfrak{a}}_i}^{q^{\deg(c_i)}})$ is a power series in $t_{\widetilde{\mathfrak{a}}_i}$ with coefficients in $H$ so that the smallest power of $t_{\widetilde{\mathfrak{a}}_i}$ with a non-zero coefficient is higher than $q^{\deg(c_i)}$.
    
Picking a non-zero element $\tilde{c_i}\in \widetilde{\mathfrak{a}_i}$ for each $1\leq i \leq n$, we further set $\tilde{c}:=\prod_{1\leq i \leq n}\tilde{c_i}$ and let $\mathfrak{m}:=(\tilde{c})\subseteq \widetilde{\mathfrak{a}_i}$. There exists an integral ideal $I_i$ of $A$ such that $\mathfrak{m}=I_i \widetilde{\mathfrak{a}_i}$. By \cite[Chap.~VI, (2.5)]{Gek86}, we have 
	\begin{equation}\label{E:unifchange21}
	t_{\widetilde{\mathfrak{a}_i}}(z)=\mathcal{J}(I_i,\mathfrak{m})^{-1}\frac{1}{\rho_{I_i}^{(\mathfrak{m})}(t_{\mathfrak{m}}(z)^{-1})}.
	\end{equation}
As in \eqref{E:unifchange01}, it is also clear from \eqref{E:unifchange21} that $t_{\widetilde{\mathfrak{a}_i}}$ may be written as a power series in $t_{\mathfrak{m}}$ with coefficients in $H$ similar to \eqref{E:algexpansion01}. Hence, combining \eqref{E:unifchange01}, \eqref{E:algexpansion01} and \eqref{E:unifchange21}, we obtain the desired statement.
\end{proof}

\subsection{Drinfeld $A$-modules over $K$-schemes and their moduli spaces} Let $S$ be a $K$-scheme via the structure map $j:K \rightarrow \Gamma(S,\mathcal{O}_S)$. In what follows, we focus on Drinfeld $A$-modules over $S$ and their properties. Our exposition is mainly based on \cite[\S5]{Dri74} and \cite[\S1]{Boc02}.

Let $\mathcal{L}$ be a line bundle over $S$ and  $\End(\mathcal{L})$ be the group of endomorphisms of the group scheme underlying $\mathcal{L}$. Let $\tau:\mathcal{L}\to \mathcal{L}^q$ be the map sending $x\to x^{q}$.  
Due to Drinfeld, we know that any element of $\End(\mathcal{L})$ may be written as a finite sum $\sum_{i\geq 0}\alpha_i\tau^i$ where $\alpha_i\in H^0(S,\mathcal{L}^{\otimes (1-q^i)})$.
\begin{definition} 
	\begin{itemize}
		\item[(i)]  A \textit{Drinfeld $A$-module of rank $r$ over $S$} is a pair $\mathbb{E}=(\mathcal{L},\phi)$ consisting of a line bundle $\mathcal{L}$ on $S$ and a ring homomorphism \[\phi:A \rightarrow \End(\mathcal{L}) 
		\]
		satisfying the following properties:
		\begin{enumerate}
			\item For any homomorphism $s:\Spec(L)\to S$, where $L$ is a field, the pullback  is a Drinfeld $A$-module of rank $r$ over $L$ in the sense of Definition \ref{D:11}. 
			\item For any $a\in A$, we have $\partial(\phi_a)=j(a)$, where, locally, we realize $\phi_a$ as a finite sum $\phi_a=\sum_{i\geq 0}\phi_{a,i}\tau^i\in \End(\mathcal{L})$  and $\partial(\phi_a):=\phi_{a,0}$.  
		\end{enumerate}
		\item[(ii)] Let $I$ be a non-zero ideal of $A$ and $\mathbb{G}_{a.S}$ be the additive group scheme over $S$. The finite subgroup scheme $\mathbb{E}[I] \subset \mathbb{G}_{a,S}$ is defined to be the unique scheme representing the functor (on $S$-schemes) \[ T \mapsto  \{x \in \mathbb{E}(T)|\ \ a\cdot x=0, \ \ a\in I\}.\]
	
		\item[(iii)] Let $T\subset \mathcal{L}$ be a subscheme which is finite flat over $S$. We denote by $[T]$ the corresponding relative Cartier divisor. \textit{A level $I$-structure} on $\mathbb{E}$ is an isomorphism of $A$-modules $$\lambda:(I^{-1}/A)^r \rightarrow \mathbb{E}[I](S)$$ which induces an equality of divisors 
		\[
		 [\mathbb{E}[I]] = \sum_{\substack{\alpha \in (I^{-1}/A)^r}} [\lambda(\alpha)].
		\]
 
 \item[(iv)] \textit{An isogeny} between Drinfeld $A$-modules $\mathbb{E}$ and $\mathbb{E}'=(\mathcal{L}',\phi')$ with level $I$-structure $\lambda$ and $\lambda'$ respectively is given by a homomorphism $\xi:\mathcal{L}\to \mathcal{L}'$ of commutative group schemes over $S$ satisfying 
 \begin{itemize}
     \item[(a)] $\xi\circ\phi_a = \phi_a'\circ \xi$ for all $a\in A$,
 \item[(b)] $\xi(S)\circ \lambda=\lambda'$ where we define $\xi(S):\mathcal{L}(S)\to \mathcal{L}'(S)$ to be the map induced by $\xi$.
\end{itemize}
We further say that $\mathbb{E}$ and $\mathbb{E}'$  are \textit{isomorphic} if $\xi$ is an isomorphism.
\end{itemize}
\end{definition}

Let $I$ be a non-zero proper ideal of $A$. With the above definitions, we define the following functor $$\mathbb{M}_I^r:\textbf{Sch}_K \rightarrow \textbf{Sets}$$ which sends a $K$-scheme $S$ to the set of isomorphism classes of Drinfeld $A$-modules of rank $r$ with a level $I$-structure over $S$. As a consequence of \cite[Prop.~5.3]{Dri74}, we obtain the following crucial result.

\begin{theorem}[cf. {\cite[Prop.~5.3]{Dri74}}]\label{T:fine}
	 The functor $\mathbb{M}^{r}_I$ is represented by a scheme $M^r_{I}$ which is affine and of finite type over $\Spec(K)$. Moreover, it is a smooth scheme of dimension $r-1$ over $K$. 
\end{theorem}

As a consequence of the above theorem, there exists an associated \textit{universal Drinfeld $A$-module} over $M^r_{I}$ which we denote by $\mathbb{E}^{un,r}_{I}=(\mathcal{L}^{un,r},\phi^{un,r})$. \\

\begin{remark}\label{R:triv_univ}
    Let $I$ be a non-zero proper ideal of $A$ and $\mathbb{E}=(\mathcal{L},\phi)$ be a Drinfeld $A$-module with a level $I$-structure $\lambda:(I^{-1}/A)^r \rightarrow \mathbb{E}[I](S)$. We note that any non-zero $x \in (I^{-1}/A)^r \setminus \{0\}$ gives rise to a non-vanishing section of the underlying locally free sheaf of $\mathcal{L}$. Consequently, $\mathcal{L}$ is isomorphic to the trivial line bundle. In particular, due to the existence of a level $I$-structure for $\mathbb{E}^{un,r}_I$, its underlying line bundle is trivial. 
\end{remark}

\subsection{de Rham cohomology of Drinfeld $A$-modules} The theory of de Rham cohomology for Drinfeld $A$-modules analogous to the classical setting was developed by Anderson, Deligne, Gekeler and Yu (see \cite{Gek89, Gek90}). In this subsection, we briefly describe the de Rham module  for Drinfeld $A$-modules defined over an affine $K$-scheme $S$ and construct a locally free $\mathcal{O}_{S}$-sheaf of rank $r$. Our exposition mainly follows \cite[\S3, 4]{Gek90}.

For a reduced $K$-algebra $B$, let $S=\Spec(B)$  and let $\mathbb{E}=(\mathcal{L},\phi)$ be a Drinfeld $A$-module over $S$. We are primarily interested in the case of the universal Drinfeld $A$-module and by Remark \ref{R:triv_univ}, we will assume that $\mathcal{L}$ is trivial. We define $\M(\mathbb{E},B) $ to be the set of $\mathbb{F}_q$-linear morphisms $\alpha:\mathcal{L}\to \mathbb{G}_{a,S}$ of $S$-group schemes. It is naturally equipped with the left $B$-module structure and also with a right $A$-module structure given by $\alpha\cdot a:=\alpha \circ \phi_a$ for all $a\in A$ which provides a $B\otimes A$-module structure on $\M(\mathbb{E},B) $.  Moreover, we set $\N(\mathbb{E},B):=\{\alpha\in \M(\mathbb{E},B)\ \ | \Lie(\alpha)=0\}$. Since the left and right actions of $\mathbb{F}_q$ on $\M(\mathbb{E},B)$ and $\N(\mathbb{E},B)$ coincide, we also consider them as $A\otimes A$-modules induced from their $B\otimes A$-module structure via the structural map $\gamma:A\to B$. We note that due to the above construction, throughout this section, all $A\otimes A$-modules may be also considered as $B\otimes A$-modules. One can observe (see for example \cite[Prop.~1.2]{Gos80}) that, if one has a trivialization $\mathcal{L}\cong \mathbb{G}_{a,S}$ over $\Spec(B)$, we have $\M(\mathbb{E},B)=B[\tau]$ and $\N(\mathbb{E},B)=B[\tau]\tau$.

We call an $\mathbb{F}_q$-linear map $\eta:A\to \N(\mathbb{E},B)$ \textit{a biderivation} if for all $a,b\in A$, it satisfies
\[
\eta(ab)=\gamma(a)\eta(b)+\eta(a)\phi_b.
\]
We let $\D(\mathbb{E},B)$ be the $A$-bimodule of biderivations. On the other hand, for each $m\in \M(\mathbb{E},B)$, one constructs a biderivation $\eta^{(m)}$ defined by 
\[
\eta^{(m)}(a):=\gamma(a)m-m \phi_a
\]
for each $a\in A$. We call any biderivation of the form $\eta^{(m)}$ \textit{inner} and moreover we call $\eta^{(m)}$ \textit{strictly inner} if $m\in \N(\mathbb{E},B)$. We also set $\D_i(\mathbb{E},B)$ ($\D_{si}(\mathbb{E},B)$ respectively) to be the $A$-bimodule of inner (strictly inner respectively) biderivations. Observe that
\begin{equation}\label{E:sidecom}
\D_i(\mathbb{E},B)=B\eta^{(1)} \oplus \D_{si}(\mathbb{E},B).
\end{equation}

We define \textit{the de Rham module $\h_{\DR}(\mathbb{E},B)$ of $\mathbb{E}$} by the quotient
\[
\h_{\DR}(\mathbb{E},B):=\D(\mathbb{E},B)/\D_{si}(\mathbb{E},B).\]
By \cite[Prop.~3.6]{Gek90}, we know that, when $B$ is a field, $\h_{\DR}(\mathbb{E},B)$ is a $B$-vector space of rank $r$. We further define \textit{the de Rham sheaf $\mathbb{H}_{\DR}(\mathbb{E})$ associated to $\mathbb{E}$} to be the coherent sheaf on $S$ whose global section is given by 
\[
\mathbb{H}_{\DR}(\mathbb{E})(S)=\h_{\DR}(\mathbb{E},B).
\]
By \cite[Thm.~3.5]{Gek90}, we know that, up to a unique isomorphism, $\mathbb{H}_{\DR}(\mathbb{E})$ is unique and moreover, we have the following result.

\begin{theorem}[{\cite[Cor.~3.7]{Gek90}}]\label{deRham}
	The coherent $S$-sheaf $\mathbb{H}_{\DR}(\mathbb{E})$ is a locally free $\mathcal{O}_S$-sheaf of rank $r$.  
\end{theorem}

\subsection{The decomposition of the de Rham module} 
In what follows, we decompose the set of biderivations into two subsets and this gives rise to a structural result on the de Rham module of a Drinfeld $A$-module as well as for the associated de Rham sheaf. One can refer to \cite[\S3]{Gek90} for further details. 

Let $S$ and $\mathbb{E}$ be as in \S2.4, that is, $S=\Spec(B)$ for a reduced $K$-algebra $B$ and the line bundle associated with $\mathbb{E}$ is trivial.  Let $a \in A \setminus \FF_q$. We call a biderivation $\eta$ \textit{reduced} (\textit{strictly reduced} respectively) if $\deg_{\tau}(\eta_a)\leq r\deg(a)$ ($\deg_{\tau}(\eta_a)< r\deg(a)$ respectively). By \cite[Prop.~3.9]{Gek90}, we know that the notion of reducedness is independent of the choice of $a$. Moreover, for any $\eta\in \D(\mathbb{E},B)$, there exists a unique $n\in \N(\mathbb{E},B)$ such that the biderivation $\eta-\eta^{(n)}$ is reduced. These properties indeed allow us to decompose $\D(\mathbb{E},B)$ as 
\[
\D(\mathbb{E},B)=\D_r(\mathbb{E},B) \oplus \D_{si}(\mathbb{E},B)
\]
where $\D_r(\mathbb{E},B)$ is the $B$-module of reduced biderivations. Let us further set $\D_{sr}(\mathbb{E},B)$ to be the $B$-module of strictly reduced biderivations. Observe that the biderivation $\eta^{(1)}$  is reduced but not strictly reduced. Since, by definition of Drinfeld $A$-modules, the leading coefficient of $\eta^{(1)}(a)$ as a polynomial in $\tau$ is a unit for each non-constant $a\in A$, one can obtain a decomposition of $\D_r(\mathbb{E},B)$ as 
\begin{equation}\label{E:decompose}
\D_r(\mathbb{E},B)=B\eta^{(1)}\oplus \D_{sr}(\mathbb{E},B).
\end{equation} 
Noting that $\h_{\DR}(\mathbb{E},B)=\D_r(\mathbb{E},B) \oplus \D_{si}(\mathbb{E},B)/\D_{si}(\mathbb{E},B)\cong \D_r(\mathbb{E},B)$, \eqref{E:decompose} corresponds to \textit{the Hodge decomposition} of $\h_{\DR}(\mathbb{E},B)$ given by 
\begin{equation}\label{E:Hodge0}
\h_{\DR}(\mathbb{E},B)=\h_1(\mathbb{E},B)\oplus \h_2(\mathbb{E},B)
\end{equation}
where $\h_1(\mathbb{E},B):=\D_{i}(\mathbb{E},B)/\D_{si}(\mathbb{E},B)\cong B\eta^{(1)}$ and $\h_2(\mathbb{E},B)\cong\D_{sr}(\mathbb{E},B)$. For each $i=1,2$, we further define the coherent sheaf $\mathbb{H}_i(\mathbb{E})$ on $S$ so that its global section is given by $\mathbb{H}_i(\mathbb{E})(S)=\h_i(\mathbb{E},B)$. Then, by \eqref{E:Hodge0}, we have
\begin{equation}\label{E:Hodge}
\mathbb{H}_{\DR}(\mathbb{E})=\mathbb{H}_1(\mathbb{E})\oplus \mathbb{H}_2(\mathbb{E}).
\end{equation}

\begin{remark}
    \label{H1_hodge} Using \eqref{E:sidecom} and  \cite[Lem.~2.21]{Hat21}, whose proof could be easily generalized to an arbitrary admissible coefficient ring setting, we see that there exists a natural isomorphism of line bundles \[\mathbb{H}_{1}(\mathbb{E},S) \cong \omega(\mathbb{E},S):=\Lie(\mathbb{E})^{\vee}.\]In particular, denoting $\mathbb{H}_{k,un}:=\mathbb{H}_{k}(\mathbb{E}^{un,r}_I)$ for $k=1,2$ and recalling the Hodge bundle $\omega_{un}$ defined in \S1, we obtain 
    \[
    \mathbb{H}_{1,un} \cong \omega_{un}.
    \]
\end{remark}

\subsection{Rigid analytic structure on $\Gamma \backslash \Omega$ and its compactification} In this subsection, we recollect some standard facts about the rigid analytic structure on $\Gamma \backslash \Omega$ where $\Gamma$ is a particular arithmetic subgroup of $\GL_2(K)$. Throughout our exposition, we precisely give proofs of these results whenever we could not locate a reference. 

Our first goal is to describe the Bruhat-Tits tree $\mathcal{T}$. We mainly follow  \cite[\S6]{Dri74} (see also \cite[\S1]{Gek97} and \cite[Chap.~3]{Boc02} for further details). Set $\mathcal{O}_{\infty}:=\mathbb{F}_{q^{\mathfrak{d}}}[[\pi_{\infty}]]\subset K_{\infty}$ and $W:=K_{\infty}\oplus K_{\infty}$ by realizing its elements as two dimensional column vectors. By \textit{an $A$-lattice $V$ in $W$}, we mean a free $\mathcal{O}_{\infty}$-lattice of rank 2. We call two $A$-lattices $V$ and $V'$ \textit{homothetic} if there exists $c\in K_{\infty}^{\times}$ so that $V=cV'$ and denote by $[V]$ the homothety class of $V$. By the elementary divisor theorem, for any given homothety classes $[V]$ and $[V']$,  there exist $V_1\in [V]$ and $V_2\in [V']$ such that $V_2\subseteq V_1$ and a non-negative integer $m$ such that $V_1/V_2\cong \mathcal{O}_{\infty}/\pi_{\infty}^m \mathcal{O}_{\infty}$. We further set $m:=d([V],[V'])$.

\textit{The Bruhat-Tits tree} $\mathcal{T}$ is the connected $(q^{\mathfrak{d}}+1)$-regular tree so that its set of vertices $\mathcal{T}_0$ is given by the homothety classes of $\mathcal{O}_\infty$-lattices in $W$ and its set of edges $\mathcal{T}_1$ is given by pairs of vertices $\{[V],[V']\}$  satisfying $d([V],[V'])=1$. For later use, we define $v_0,v_1\in \mathcal{T}_0$ as well as $e_0\in \mathcal{T}_1$ so that $v_0:=[\mathcal{O}_{\infty}\oplus \mathcal{O}_{\infty}]$, $v_1:=[\mathcal{O}_{\infty}\oplus \pi_{\infty}\mathcal{O}_{\infty}]$ and $e_0:=\{v_0,v_1\}$.

One further defines an action of $\GL_2(K_{\infty})$ on $\mathcal{T}$ as follows: Let $V=w_1\mathcal{O}_{\infty}\oplus w_2\mathcal{O}_{\infty}$ for some generators $w_1,w_2\in W$. Then for any $g=\begin{pmatrix}
    a&b\\
    c&d
\end{pmatrix}$, we let 
\[
g\cdot V:=gw_1\mathcal{O}_{\infty}\oplus gw_2\mathcal{O}_{\infty}.
\]
One sees that this defines a transitive action of $\GL_2(K_{\infty})$ on $\mathcal{T}_0$ and $\mathcal{T}_1$.  For any $e=\{v,v'\}\in \mathcal{T}$, set 
\[
\mathfrak{i}_{e}:=\{(\alpha_{v},\alpha_{v'}) \ \ |  \ \ \alpha_{v}+\alpha_{v'}=1, \ \ 0\leq \alpha_{v},\alpha_{v'}\leq 1\}.
\]
Then \textit{the geometric realization $|\mathcal{T}|$ of $\mathcal{T}$} is given by 
\[
|\mathcal{T}|:=\sqcup_{e\in \mathcal{T}_1}\mathfrak{i}_{e}/\thicksim
\]
where, we mean by $\thicksim$ the identification of each $v\in \mathcal{T}_0$ with the tuple $(\alpha_{v},\alpha_{v'})$ so that $\alpha_v=1$, $\alpha_{v'}=0$ and $\{v,v'\}\in \mathcal{T}_1$.  By \cite{GI63}, we further note that $|\mathcal{T}|$ may be also canonically identified with the set of equivalence classes of norms $|\cdot|$ on $W$. 

In \cite[\S6]{Dri74}, Drinfeld uses the Bruhat-Tits tree $\mathcal{T}$ to provide a rigid analytic structure on the Drinfeld upper half plane $\Omega=\mathbb{P}^1(\mathbb{C}_{\infty})\setminus \mathbb{P}^1(K_\infty)$. We recap the construction which will be used in \S2.7. 

For any $g=\begin{pmatrix}
    a&b\\c&d
\end{pmatrix}\in \GL_2(K_{\infty})$ and $z:=(z_1:z_2)\in \Omega$, let 
\[
g\cdot z:=(az_1+bz_2:cz_1+dz_2)\in \Omega.
\]
We define \textit{the reduction map} 
\[
\lambda:\Omega\to |\mathcal{T}|
\]
sending each $z\in\Omega$ to $\lambda(z)$, which could be identified as a norm on $W$, so that $\lambda(z)(u_1,u_2):=|u_1z+u_2|$. Note that $\lambda$ forms a $\GL_2(K_{\infty})$-equivariant map.

In the rest of this subsection, following \cite[\S3]{Boc02}, we construct an affinoid covering of $\Omega$. We set 
\[
\mathcal{W}_{v_0}:=\{\mathfrak{v}\in |\mathcal{T}| \ \ | d(\mathfrak{v},v_0)\leq 1/3\}
\]
and 
\[
\mathcal{W}_{e_0}:=\{(\alpha_{v_0},\alpha_{v_1}) \in |\mathcal{T}|\ \ | \  \alpha_{v_0}\geq 1/3, \ \ \alpha_{v_1}\geq 1/3, \ \  \alpha_{v_0}+\alpha_{v_1}=1\}.
\]
For an arbitrary vertex $v=\gamma \cdot v_0$ for some $\gamma \in \GL_2(K_{\infty})$, we set $\mathcal{W}_v:=\gamma \mathcal{W}_{v_0}$ and for an arbitrary edge $e=\gamma e_0\in \mathcal{T}_1$, we let $\mathcal{W}_{e}:=\gamma \mathcal{W}_{e_0}$. We note that the definition of $\mathcal{W}_v$ and $\mathcal{W}_e$ are independent of the chosen $\gamma$. 

In what follows, we introduce the barycentric subdivision of $\mathcal{T}$ and the nerve of a particular covering of $\Omega$ (see \cite{Mun84} for more details). The \textit{barycentric subdivision of $\mathcal{T}$} is the $1$-dimensional simplicial complex, whose $0$-simplices is the set $\mathcal{T}_0 \sqcup \mathcal{T}_1$ and the $1$-simplices consist of pairs $\{\nu_0,\nu_1\} \subset \mathcal{T}_0 \sqcup \mathcal{T}_1$ if and only if either $\nu_0 \subsetneq \nu_1$ or $\nu_1 \subsetneq \nu_0$ (that is, it is a pair of a vertex and an edge, the latter containing the former). For any $\nu\in \mathcal{T}_0\sqcup \mathcal{T}_1$, let $\mathfrak{U}_{\nu}:=\lambda^{-1}(\mathcal{W}_\nu)$ and consider the cover  $\mathfrak{U}:=\{\mathfrak{U}_\nu\ \ | \nu\in \mathcal{T}_0\sqcup \mathcal{T}_1\}$ of $\Omega$.  The \textit{nerve} of the covering $\mathfrak{U}$ is the set of all finite subsets $\{\nu_0,...,\nu_k\} \in \mathcal{T}_0 \sqcup \mathcal{T}_1$ such that $\mathfrak{U}_{\nu_0} \cap ... \cap \mathfrak{U}_{\nu_k} \neq \emptyset$. It is clear that this is a simplicial complex whose $0$-simplices lie in $\mathcal{T}_0 \sqcup \mathcal{T}_1$.
    
\begin{proposition}[{\cite[Prop.~6.2]{Dri74}}, {\cite[Prop.~3.11]{Boc02}}] \label{P: covering} The cover $\mathfrak{U}$ is an admissible affinoid cover of $\Omega$. Moreover the nerve of this covering is the barycentric subdivision of  $\mathcal{T}$. 
\end{proposition}

Our goal from now on is to provide an admissible covering for the quotient of $\Omega$ with certain arithmetic subgroups of $\GL_2(K)$. Recall that $Y$ is the projective $A$-module given as in \eqref{E:defofY}. For the convenience of the reader, we recall that
\begin{multline*}
\Gamma_Y=\GL(Y)=\{\gamma \in \GL_2(K) \ \ | \ \ Y \gamma =Y\}\\
= \Bigg\{ \begin{pmatrix}
a&b\\c&d
\end{pmatrix}\in \GL_2(K) \ \ | \ \  a,d\in A, \ \ ad-bc\in \mathbb{F}_q^{\times}, \ \ b\in \mathfrak{g}^{-1}\mathfrak{h}, \ \ c\in \mathfrak{g}\mathfrak{h}^{-1}  \Bigg\}
\end{multline*}
where the last equality follows from \cite[(7.1)]{Gek90}. For any non-zero ideal $I$ in $A$, we define \textit{the principal congruence subgroup of $\GL_2(K)$ of level $I$} by $\Gamma_Y(I):=\Ker (\Gamma_Y\to \GL(Y/IY))$. Throughout this subsection, we let $\Gamma:=\Gamma_Y(I)$. 

We recall that a map of rigid analytic spaces $f:X \rightarrow W$ over $L$ is said to be \textit{\'{e}tale} if for each $x \in X$, the map on stalks $\mathcal{O}_{W,f(x)} \rightarrow \mathcal{O}_{X,x}$ is flat and unramified (see \cite[\S8.1]{FvdP04} for more details).  Let $B$ be an affinoid algebra over a complete subfield $L$ of $\mathbb{C}_{\infty}$ equipped with an action of a finite group $G$ and let $B^G$ be the set of elements of $B$ invariant under the action of $G$. In \cite[Prop.~6.3]{Dri74}, Drinfeld shows that $B^G$ is also an affinoid algebra over $L$ and that $B^G \rightarrow B$ is a finite map. We further let $\Sp(B)$ be the set of all maximal ideals of $B$ equipped with a suitable ringed space structure (see \cite[Def.~3.3.1, 4.2.7]{FvdP04} for more details on $\Sp(B)$). Using a result of Mumford in \cite[Chap.~2, \S7]{Mum74}, we can enhance the finiteness of the map above to a finite \'{e}tale map in the special case when $G$ acts freely on $\Sp(B)$. 

\begin{lemma} \label{5.2} Let $X=\Sp(B)$ be the affinoid space equipped with a free left action of a finite group $G$. Then the quotient map $\mathtt{p}:\Sp(B):=X \rightarrow  G \backslash X=\Sp(B^G)$  is \'{e}tale.
	\end{lemma} 
\begin{proof} Let $x\in X$ and set $y:=\mathtt{p}(x)$. We claim that the map $\widehat{\mathcal{O}_{G \backslash X,y}} \rightarrow \widehat{\mathcal{O}_{X,x}}$ is an isomorphism. Assuming the claim, note that if $\mathtt{m}$ denotes the maximal ideal of $\widehat{\mathcal{O}_{G \backslash X,y}}$, then there is an isomorphism of residue fields $\mathcal{O}_{G\backslash X,y}/\mathtt{m} \cong \mathcal{O}_{X,x}/\mathtt{m}$. Moreover, the faithful flatness of $() \rightarrow \widehat{()}$ implies the flatness of $\mathcal{O}_{G \backslash X,y} \rightarrow \mathcal{O}_{X,x}$, consequently showing $\mathtt{p}$ to be \'{e}tale at $x$. Hence it suffices to show our claim to finish the proof. As above, let $\mathtt{m} \subset B^G$ denote the maximal ideal corresponding to the point $y \in G \backslash X$ and let $\widehat{B}$ ($\widehat{B^G}$ respectively) be the completion of $B$ ($B^G$ respectively) with respect to the ideal $\mathtt{m}B$ ($\mathtt{m}$ respectively). It is clear that the natural map $\widehat{B^G} \otimes_{B^G} B \cong \widehat{B}$ is an isomorphism and by \cite[Prop.~4.6.1(1)]{FvdP04} $\widehat{B^G}$ is isomorphic to the completion $\widehat{\mathcal{O}_{G \backslash X,y}}$ of the Noetherian local ring $\mathcal{O}_{G\backslash X,y}$. Since the elements in the preimage of $y$ under the map $\mathtt{p}$ are of the form $g\cdot x$ for some $g\in G$, which are all distinct due to the free action of $G$ on $X$, by the Chinese Remainder Theorem and again applying \cite[Prop.~4.6.1(1)]{FvdP04}, we have
$$\widehat{B} \cong \prod_{g \in G} \widehat{\mathcal{O}_{X,g\cdot x}}.$$
By the definition of $B^G$, we have the following equalizer diagram $$0 \rightarrow B^G \rightarrow B \rightarrow \prod_{g \in G} B$$ where we send $b \in B \mapsto (b-g\cdot b)_{g \in G} \in \prod_{g \in G} B$. On the other hand, since $B^G \rightarrow \widehat{B^G}$ is flat, we also obtain
$$0 \rightarrow \widehat{B^G} \rightarrow \widehat{B^G} \otimes_{B^G} B \rightarrow \prod_{g \in G} \widehat{B^G} \otimes_{B^G}B.$$  
Equivalently, $\widehat{B^G} = (\widehat{B^G} \otimes_{B^G} B)^G \cong (\widehat{B})^G$. We also have
\[
\widehat{B} \cong \prod_{g \in G} \widehat{\mathcal{O}_{X,g\cdot x}} \cong \prod_{g \in G} \widehat{\mathcal{O}_{X,x}}
\]
where we identify, by a slight abuse of notation,  $\mathcal{O}_{X,x} \cong \mathcal{O}_{X,g\cdot x}$. Thus, we get $(\prod_{g \in G} \widehat{\mathcal{O}_{X,x}})^G = \widehat{\mathcal{O}_{X,x}}$ given by the diagonal embedding of $\widehat{\mathcal{O}_{X,x}} \hookrightarrow \prod_{g \in G} \widehat{\mathcal{O}_{X,x}}$. Hence $\widehat{\mathcal{O}_{G \backslash X,y}} \cong \widehat{B^G} \cong \widehat{\mathcal{O}_{X,x}}$, finishing the proof of the lemma.
\end{proof}

There exists a canonical rigid analytic structure on $\Gamma \backslash \Omega$ induced via the quotient map \[\pi:\Omega \rightarrow \Gamma \backslash \Omega.\] 
In other words, $U \subset  \Gamma \backslash \Omega$ is an admissible open set if and only if $\pi^{-1}U \subset \Omega$ is an admissible open. Similarly, one can also define admissible coverings of $\Gamma\backslash \Omega$. Furthermore, the structure sheaf  $\mathcal{O}_{ \Gamma \backslash \Omega}$ on $\Gamma \backslash \Omega$ is given  by \[\mathcal{O}_{ \Gamma \backslash \Omega}(U):=\mathcal{O}_{\Omega}(\pi^{-1}U)^{\Gamma} \text{ for admissible $U$}.\]

\begin{proposition}[{cf. \cite[\S5, Thm.~2]{SS91}}]
	\label{a6}
	The quotient $\Gamma \backslash \Omega$ admits an admissible covering $\{\overline{\mathfrak{U}}_\nu\}_{\nu \in \mathcal{T}}$ where $\overline{\mathfrak{U}}_\nu \cong \Gamma_\nu \backslash \mathfrak{U}_\nu$ as affinoid spaces. Moreover $\pi:\Omega \rightarrow \Gamma\backslash \Omega$ is an \'{e}tale map. 
    
\end{proposition}
\begin{proof} Let $\nu \in \mathcal{T}$ be a simplex, that is, either an element in $\mathcal{T}_0$ or $\mathcal{T}_1$, and let $\Gamma_\nu$ denote the stabilizer of $\nu$ in $\Gamma$. Then we see that $\Gamma_\nu$ acts on $\mathcal{W}_\nu$ and hence also on  $\mathfrak{U}_\nu$. Moreover, by \cite[Lem.~3.17]{Boc02}, $\Gamma_\nu$ is a finite subgroup and hence being a $p$-group also acts freely on $\mathfrak{U}_\nu$. Consider the set \[\mathfrak{U}_{\Gamma,\nu}:= \sqcup_{g \in \Gamma / \Gamma_\nu} \mathfrak{U}_{g\cdot \nu} = \sqcup_{g \in \Gamma / \Gamma_\nu} g\mathfrak{U}_\nu.\] 
By Proposition \ref{P: covering}, $\mathfrak{U}_{\nu}$ has the barycentric subdivision of $\mathcal{T}$ as its nerve, which is locally finite. Hence for any arbitrary vertex $\nu$, we have $\mathfrak{U}_{\Gamma,\nu} \cap \mathfrak{U}_{y} = \sqcup_i \mathfrak{U}_{g_i\nu} \cap \mathfrak{U}_{y}$, where $\{g_i\}$ is a finite subset of representatives of $\Gamma / \Gamma_\nu$. Thus $\mathfrak{U}_{\Gamma,\nu}$ is an admissible subset of $\Omega$. 
Hence $\overline{\mathfrak{U}}_{\nu}:=\pi(\mathfrak{U}_{\Gamma,\nu})$ is an admissible subset of $\Gamma \backslash \Omega$. On the other hand, note that the surjection $\pi:\mathfrak{U}_{\Gamma,\nu} \rightarrow \overline{\mathfrak{U}}_{\nu}$ induces an isomorphism of rigid analytic spaces $\Gamma_\nu \backslash \mathfrak{U}_{\nu} \xrightarrow{\sim} \overline{\mathfrak{U}}_{\nu}$.
Since $\Gamma_\nu$ is a finite group, by \cite[Prop.~6.3]{Dri74}, $\overline{\mathfrak{U}}_{\Gamma,\nu}$ is an affinoid space. Therefore, by the definition of rigid analytic structure on $\Gamma \backslash \Omega$, the covering $\{\overline{ \mathfrak{U}}_{\nu}\}$ is an admissible covering of $\Gamma \backslash \Omega$. Finally, 
using the fact that $\Gamma_\nu$ acts freely on $\mathfrak{U}_\nu$, we apply Lemma \ref{5.2} to the map $\pi_{|\mathfrak{\mathfrak{U}_\nu}}:\mathfrak{U}_{\nu} \rightarrow \overline{\mathfrak{U}}_{\nu}$, which is a finite free quotient, to conclude that $\pi_{|\mathfrak{\mathfrak{U}_\nu}}$ is finite \'{e}tale. Hence the quotient $\pi:\Omega \rightarrow \Gamma \backslash \Omega$ is locally a finite \'{e}tale cover. 
\end{proof}

In the last part of the present subsection, analogous to the theory of compactification of quotients of the upper half plane by congruence subgroups in $\SL_2(\mathbb{Z})$, following \cite[\S3.3]{vdP97},
we introduce a rigid analytic structure on the compactification $\overline{\Gamma \backslash \Omega}$  of $\Gamma \backslash \Omega$. To do this, we first define 
\[\overline{\Gamma \backslash \Omega}:= \Gamma \backslash \Omega \sqcup \Gamma \backslash \mathbb{P}^1(K)\] 
and set $\Cusps^Y_{I}:=\Gamma \backslash \mathbb{P}^1(K)$ to be \textit{the set of cusps of $\Gamma=\Gamma_Y(I)$}. 

Observe that, by \cite[Prop.~1.69]{Gos80}, for any $\delta \in \GL_2(K)$, there exists a maximal fractional ideal $\mathfrak{a}_{\delta\Gamma \delta^{-1}}$ of $A$ so that 
\[
\Bigg\{ \begin{pmatrix}
	1&a\\0&1
	\end{pmatrix} \ \ | \ \ a\in\mathfrak{a}_{\delta\Gamma \delta^{-1}}  \Bigg\}\subset \delta\Gamma \delta^{-1}.
\]
By a slight abuse of notation, we identify $\mathfrak{a}_{\delta \Gamma \delta^{-1}}$ as a subgroup of $\delta\Gamma \delta^{-1}$ via  above.  
We further set 
\begin{equation}\label{E:unifdef}
t_{\delta\Gamma \delta^{-1}}(z):=t_{\mathfrak{a}_{\delta\Gamma \delta^{-1}}}(z)=\exp_{\rho^{(\mathfrak{a}_{\delta\Gamma \delta^{-1}})}}(\xi(\mathfrak{a}_{\delta\Gamma \delta^{-1}})z)^{-1}=\xi(\mathfrak{a}_{\delta\Gamma \delta^{-1}})^{-1}\exp_{\mathfrak{a}_{\delta\Gamma \delta^{-1}}}(z)^{-1}.
\end{equation}

Consider \textit{the imaginary distance} $|\cdot|_{\text{im}}$ defined by $\inorm{z}_{\text{im}}:=\inf\{\inorm{z-a}\ \ | \ \ a\in K_{\infty}\}$ for any $z\in \Omega$. For any $u\in \mathbb{Z}$, we define the \textit{horicycle neighborhood of infinity} by the set \[\mathcal{N}_u:=\{z \in \Omega \ \ | \ \  \inorm{z}_{\text{im}} \geq u\}.\] 
\begin{lemma}[{\cite[Prop.~1.65, Cor.~1.73]{Gos80}, \cite[Lem.~3.3]{vdP97}}]\label{L:inv} The following statements hold.
   \begin{itemize}
       \item[(i)] For each $u\in \mathbb{Z}$, $\mathcal{N}_u \subset \Omega$ is admissible open.
       \item[(ii)] There exists $u \gg 0$ such that 
       \begin{itemize}
           \item[(a)] $\mathcal{N}_u$ is invariant under the action of $\mathfrak{a}_{\delta \Gamma \delta^{-1}}$ and
           \item[(b)] if there exists $z_1,z_2 \in \mathcal{N}_u$ such that $g\cdot z_1=z_2$ for some $g \in \delta\Gamma \delta^{-1}$, then $g \in \mathfrak{a}_{\delta \Gamma \delta^{-1}}$.
       \end{itemize}
       In particular, the image of $\mathcal{N}_u$ under the map $\Omega\to  \delta\Gamma\delta^{-1}\backslash \Omega$ is an admissible open set, which is in a natural bijection with $\mathfrak{a}_{\delta \Gamma \delta^{-1}} \backslash \mathcal{N}_u$, thus endowing the latter space with a rigid analytic structure.  
   \end{itemize} 
\end{lemma}
Let $\theta \in \mathbb{C}_\infty$ be such that $|\theta|=q$. Letting $e$ and $h$ be non-negative integers and $X$ be an indeterminate over $\mathbb{C}_{\infty}$, we define the affinoid algebras
\[
\mathbb{C}_\infty\Big\langle \frac{X}{\theta^{e}}\Big\rangle:=\Bigg\{\sum_{k=0}^{\infty}a_kX^{k}  \ \ |   \lim_{k\to \infty}a_kq^{ek}=0 \Bigg\}
\] 
and
\[
\mathbb{C}_\infty \Big\langle \frac{X}{\theta^{e}},\frac{\theta^{h}}{X}\Big\rangle:=\Bigg\{\sum_{k=-\infty}^{\infty}a_kX^{k}  \ \ |  \ \  \lim_{j\to -\infty}a_jq^{jh} = \lim_{k\to \infty}a_kq^{ek}=0 \Bigg\}.
\]
Now let $\mathfrak{t}:\Omega\to \mathbb{C}_{\infty}$ be a rigid analytic function. For any $e\in \mathbb{Q}$, we denote the \textit{punctured disc at infinity} by \[\mathbb{D}^{\ast}_e(\mathfrak{t}) : = \{z \in \Omega\ \ | \ \ 0 < |\mathfrak{t}(z)| \leq q^e\} = \bigcup_{h \in \mathbb{Z}_{\leq 0}} \Sp\Big(\mathbb{C}_\infty\Big\langle \frac{\mathfrak{t}}{\theta^e}, \frac{\theta^h}{\mathfrak{t}}\Big\rangle\Big).\]
Choosing a large enough $u$, by \cite[Prop.~4.7(c)]{BBP21}, we see that the image of $\Omega \setminus \mathcal{N}_u = \{z \in \Omega\ \ | \ \  |z| _{\text{im}} < u\}$ under $t_{\delta\Gamma \delta^{-1}}^{-1}$ is bounded above. Thus, the image of $\mathcal{N}_u$ under $t_{\delta\Gamma \delta^{-1}}$ contains a punctured disc $\mathbb{D}^\ast_e(t_{\mathfrak{a}_{\delta \Gamma \delta^{-1}}})$ for some $e$. On the other hand, since $t_{\delta\Gamma \delta^{-1}}^{-1}:\mathfrak{a}_{\delta \Gamma \delta^{-1}} \backslash\mathbb{C}_\infty  \to \mathbb{C}_\infty$ is an isomorphism, we have the following lemma.

\begin{lemma}[{\cite[Thm.~1.76]{Gos80}}]\label{L:nbd_at_infty}
    Choose $u \gg 0$ so that Lemma \ref{L:inv}(ii) holds true. Then the map $t_{\mathfrak{a}_{\delta \Gamma \delta^{-1}}}:\mathcal{N}_u \rightarrow \mathbb{C}_\infty$ induces  an isomorphism of rigid analytic spaces \[\mathfrak{a}_{\delta \Gamma \delta^{-1}} \backslash \mathcal{N}_u \cong \mathbb{D}^\ast_e(t_{\mathfrak{a}_{\delta \Gamma \delta^{-1}}})\] for some $e \in \mathbb{Z}$. 
\end{lemma}

Now we are ready to define a rigid analytic structure on $\overline{\Gamma \backslash \Omega}$: Let $b $ be a cusp and $\gamma \in \GL_2(K)$ be such that $\gamma\cdot b=\infty$. Choose a large enough $u$ so that Lemma \ref{L:nbd_at_infty} holds true, after replacing $\delta$ with $\gamma$. Then  the map of rigid analytic spaces $\Gamma \backslash \Omega \rightarrow \gamma \Gamma\gamma^{-1} \backslash \Omega$ sending $[z] \mapsto [\gamma z]$ induces an isomorphism 
\begin{equation}\label{E:isomrigid}
\pi(\gamma^{-1}(\mathcal{N}_u)) \cong \mathfrak{a}_{\gamma \Gamma \gamma^{-1}} \backslash \mathcal{N}_u \cong \mathbb{D}^\ast_e(t_{\mathfrak{a}_{\gamma \Gamma \gamma^{-1}}}).
\end{equation}
Consequently,  \eqref{E:isomrigid} may be extended to a unique isomorphism so that 
\[
\pi(\gamma^{-1}(\mathcal{N}_u)) \cup \{b\}\cong \mathbb{D}_e(t_{\mathfrak{a}_{\gamma \Gamma \gamma^{-1}}}) := \Sp\Big(\mathbb{C}_\infty\Big\langle \frac{t_{\mathfrak{a}_{\gamma \Gamma \gamma^{-1}}}(z)}{\theta^e}\Big\rangle\Big).
\]
Thus, we realize $\pi(\gamma^{-1}(\mathcal{N}_u)) \cup \{b\}$ as an admissible open subset around $b$. Repeating this process for each cusp in $\Cusps^Y_{I}$  gives us a rigid analytic structure on $\overline{\Gamma \backslash \Omega}$. 

\subsection{Sheaves on $\Gamma \backslash \Omega$} We continue with the same notation as in \S2.6. In this subsection, we analyze coherent sheaves on the rigid analytic spaces $\Omega$ and $\Gamma \backslash \Omega$.

Let $X$ be a rigid analytic space equipped with a left action of $\Gamma$. We start with the definition of a $\Gamma$-sheaf on $X$.  

\begin{definition}\label{D: G_sheaf}
	 A coherent sheaf $\mathcal{F}$ on $X$ is called \textit{a (right) $\Gamma$-sheaf} if for every $\gamma\in\Gamma$ there exists an isomorphism $f_\gamma$ between sheaves $$f_\gamma:\mathcal{F} \rightarrow \gamma_\ast\mathcal{F}$$ such that $f_1=\Id$ and $\gamma_\ast (f_\delta) \circ f_\gamma = f_{\gamma\delta}$ for all $\gamma,\delta \in \Gamma$.  
\end{definition}

We denote by $\Coh^{\Gamma}_X$ the category of $\Gamma$-sheaves on $X$, which is the full subcategory of $\Gamma$-sheaves inside the category $\Coh_X$ of coherent sheaves on $X$.

\begin{lemma} For any sheaf $\mathcal{F}\in \Coh^{\Gamma}_X $, the presheaf $\mathcal{F}^\Gamma$ whose section for an admissible open subset $U \subset \Gamma \backslash \Omega$ is given by $\mathcal{F}^\Gamma(U):=(\mathcal{F}(\pi^{-1}U))^{\Gamma}$ is a sheaf.
\end{lemma}
\begin{proof}
	Observe that $()^\Gamma$ is a left exact functor. Thus, since $\mathcal{F}$ is a sheaf,  the sheaf axioms for $\mathcal{F}^\Gamma$ also hold true.
\end{proof}

By the rigid analytic structure on $\Gamma \backslash \Omega$, $\mathcal{F}^\Gamma$ is a coherent sheaf on $\Gamma \backslash \Omega$. This allows us to define the functor \[()^\Gamma:\Coh^{\Gamma}_{\Omega} \rightarrow \Coh_{\Gamma \backslash \Omega}\] sending $\mathcal{F} \mapsto \mathcal{F}^\Gamma$. In what follows, we aim to show that this functor induces an equivalence of categories with the quasi inverse being $\pi^\ast$.  
We prove this in the corollary below, by first stating the analogue of \cite[Chap.~2, Prop.~7.2]{Mum74} for affinoid spaces, whose proof holds true almost verbatim in our case.

\begin{proposition}\label{P:gamma}
	Let $X=\Sp(B)$ be an affinoid space, for an affinoid algebra $B$, equipped with a free left action of a finite group $G$. Then with respect to the quotient map $\mathtt{p}:X=\Sp(B) \rightarrow G \backslash X=\Sp(B^G)$ given as in Lemma \ref{5.2}, there exists an equivalence of categories $$\mathtt{p}^\ast:\Coh_{G\setminus X} \rightarrow \Coh^{G}_X$$ sending $\mathcal{G} \mapsto \mathtt{p}^{\ast}\mathcal{G}$. Moreover, the quasi inverse of $\mathtt{p}^\ast$ is given by sending $\mathcal{F} \mapsto (\mathtt{p}_{\ast}\mathcal{F})^G$. Furthermore, $\mathtt{p}^\ast$ commutes with $- \otimes -$ and $\Sym^{n}$ for any $n\in \mathbb{Z}_{\geq 1}$, and so does its quasi inverse. 
\end{proposition}

We  finish the present subsection with a proof of the next corollary. Recall the affinoid subsets $\overline{\mathfrak{U}}_\nu$ from \S2.6 and denote by $\pi_\nu:\mathfrak{U}_\nu \rightarrow \Gamma_\nu \backslash \mathfrak{U}_\nu \cong \overline{\mathfrak{U}}_\nu$ the natural finite quotient map.

\begin{corollary}
	\label{a12} The following statements hold.
	\begin{itemize}
		\item[(i)] $\mathcal{F}^\Gamma$ is the unique coherent sheaf on $\Gamma \backslash \Omega$ which, restricted to $ \overline{\mathfrak{U}}_\nu$ for all $\nu \in \mathcal{T}$, provides an isomorphism  $(\mathcal{F}^\Gamma) _{|\overline{\mathfrak{U}}_\nu} \cong (\pi_{\nu,\ast}(\mathcal{F}_{|\mathfrak{U}_\nu}))^{\Gamma_\nu}$.
		\item[(ii)] The functor $\pi^\ast:\Coh_{\Gamma \backslash \Omega} \rightarrow \Coh^{\Gamma}_{\Omega}$ sending  $\mathcal{G} \mapsto \pi^{\ast}\mathcal{G}$ is an equivalence of categories. Moreover, the quasi inverse of $\pi^\ast$ is given by $\mathcal{F} \mapsto \mathcal{F}^{\Gamma}$.
		\item[(iii)] Both functors defined in (ii) commute with applying $- \otimes -$ and $\Sym^{n}$, for any $n\in \mathbb{Z}_{\geq 1}$. 
	\end{itemize}
\end{corollary}

\begin{proof}
	We first prove (i).  Set $\overline{\mathfrak{U}}:=\{\overline{\mathfrak{U}}_\nu\}_{\nu \in \mathcal{T}_0\sqcup \mathcal{T}_1}$. Since $\mathcal{F}^{\Gamma}$ is by definition, a coherent sheaf on $\Gamma \backslash \Omega$, by \cite[Chap.~6, Cor.~5]{Bos14}, $\mathcal{F}^\Gamma$ is also a $\overline{\mathfrak{U}}$-coherent sheaf, that is for each $\nu \in \mathcal{T}$, $(\mathcal{F}^\Gamma)_{|\overline{\mathfrak{U}}_{\nu}}$ is the unique coherent sheaf on $\overline{\mathfrak{U}}_{\nu}$ associated to the $\mathcal{O}(\overline{\mathfrak{U}}_{\nu})$-module $\mathcal{F}^\Gamma(\overline{\mathfrak{U}}_{\nu})$. We analyze this latter module as follows. Note that the natural map \[\mathcal{F}(\mathfrak{U}_\nu) \hookrightarrow \prod_{g \in \Gamma / \Gamma_\nu} \mathcal{F}(\mathfrak{U}_{g\nu}) = \prod_{g \in \Gamma / \Gamma_\nu} (g^{-1})_\ast \mathcal{F}(\mathfrak{U}_\nu)\]
		sending $s \mapsto ((g^{-1})_{\ast}s)_{g\in \Gamma /\Gamma_{\nu}}$ induces an isomorphism \begin{equation}\label{E:finite_Gamma}
        \mathcal{F}(\mathfrak{U}_\nu)^{\Gamma_\nu} \cong (\prod_{g \in \Gamma / \Gamma_\nu} \mathcal{F}(\mathfrak{U}_{gv}))^{\Gamma} = \mathcal{F}^\Gamma(\overline{\mathfrak{U}}_\nu).\end{equation} 
        Consequently, we have an isomorphism
        \begin{equation}\label{E:isomorp}
(\pi_{\nu,\ast}(\mathcal{F}_{|\mathfrak{U}_\nu}))^{\Gamma_\nu} \cong (\mathcal{F}^{\Gamma})_{|\overline{\mathfrak{U}}_\nu}.
\end{equation}
		This finishes the proof of (i). On the other hand, Proposition \ref{P:gamma} and the isomorphism \eqref{E:isomorp} show that the functors defined in (ii) are quasi inverses to each other and hence the equivalence of the aforementioned categories is established. Finally, note that the functor $\pi^{\ast}:\Coh_{\Gamma \backslash \Omega} \rightarrow \Coh^{\Gamma}_{\Omega}$ commutes with $-\otimes -$ and $ \Sym^{n}$, hence so does its quasi inverse and this finishes the proof of (iii). 
\end{proof}

\subsection{Drinfeld $A$-modules over rigid analytic spaces}
Let $L$ be a complete subfield of $\mathbb{C}_{\infty}$ containing $K_\infty$ 
and $X$ denote a rigid analytic space over $L$. When $X=\Sp(\mathbb{C}_\infty)$, in \S2.1, we described the construction of a  Drinfeld $A$-module over $X$ via $A$-lattices inside $\mathbb{C}_\infty$. In what follows, using \cite[\S4]{Boc02}, we introduce a similar procedure to form Drinfeld $A$-modules over an arbitrary rigid analytic space over $L$.

Consider a locally free sheaf $\mathcal{M}$ of rank one on $X$. We define the non-commutative polynomial ring
\[
\Gamma_{X,\mathcal{M}}[\tau]:=\Bigg\{ \sum_{i\geq 0} f_i\tau^i \ \ | \ \ f_i\in \Gamma(X,\mathcal{M}^{\otimes (1-q^i)}), \ \ f_i\equiv 0 \text{ when } i\gg 0 \Bigg\} 
\]
subject to the condition $f_i\tau^if_j\tau^j=f_if_j^{q^i}\tau^{i+j}$ for each $i,j\in \mathbb{Z}_{\geq 0}$ and 
where the map $\tau: \mathcal{M} \rightarrow \mathcal{M}^q$ is defined so that for each admissible open subset $U\subset X$ and $x\in \mathcal{M}(U)$, $\tau(x)=x^q$. 
\begin{definition}\label{ana_drin} \textit{A Drinfeld $A$-module of rank $r$ over $X$} is a pair $(\mathcal{M},\phi)$ such that 
	$
\phi:A \rightarrow \Gamma_{X,\mathcal{M}}[\tau]
	$
is a ring homomorphism given by 
    \[
    \phi_a:=f_0+f_1\tau+\cdots+f_{r\deg(a)}\tau^{r\deg(a)}    \]
  satisfying the following properties:
	\begin{itemize}
    \item[(1)] $f_0=a$ and $f_{r\deg(a)}\in \Gamma(X,\mathcal{M}^{\otimes (1-q^{r\deg(a)})})$ is nowhere vanishing.
    \item[(2)] Let $\partial:\Gamma_{X,\mathcal{M}}[\tau] \rightarrow \mathcal{O}_X(X)$ be the map sending $\sum_{j\geq 0} g_j \tau^j \mapsto g_0$. Then the composition $\partial \circ \phi:A \rightarrow \Gamma(X,\mathcal{O}_X)$ is the map corresponding to $A \hookrightarrow L \rightarrow \mathcal{O}_X(X)$.
    \end{itemize}
   \end{definition} 

Analogous to the work of Drinfeld, we will now construct a Drinfeld $A$-module over $X$ associated to \textit{$A$-lattices over $X$}, which we will define in what follows.

\begin{definition}
\begin{itemize}
    \item[(i)]Let $B$ be an affinoid algebra over $L$ and $|.|_B$ denote a residue norm on $B$. Let $i_B:A \rightarrow B$ be the canonical map. Then an $A$-module $\Lambda \subset B$ is said to be \textit{an $A$-lattice of rank $r$ in $B$} if it satisfies the following:
    \begin{itemize}
        \item[(1)] $\Lambda$ is a projective $A$-module of rank $r$,
        \item[(2)] the elements of $\Lambda \backslash \{0\}$ are units,
        \item[(3)] for all $c \in \mathbb{R}_{>0}$, the set  $\{x \in \Lambda \backslash \{0\}\ \ | \ \ |x^{-1}|_B \geq  c\}$ is finite.
    \end{itemize}
\item[(ii)] An \textit{$A$-lattice of rank $r$ over $X$} is a triple $(\mathcal{M},\underline{\Lambda},s)$, consisting of
\begin{enumerate}
    \item a locally free sheaf $\mathcal{M}$ of rank one over $X$,
    \item a sheaf $\underline{\Lambda}$ of projective $A$-modules of rank $r$ over the rigid analytic site of $X$,
    \item a monomorphism $s:\underline{\Lambda} \rightarrow \mathcal{M}$ such that, for any affinoid open set $U$ of $X$, over which $\mathcal{M}$ can be trivialized, $\underline{\Lambda}(U)$ is  an $A$-lattice of rank $r$ in $\mathcal{M}(U)$ in the sense of part (i). 
\end{enumerate}
\end{itemize}
\end{definition}

Following \cite[Prop.~4.2]{Boc02}, we now sketch how to associate to an $A$-lattice over $X$ a Drinfeld $A$-module. Firstly, passing to a trivializing cover of $\mathcal{M}$, we assume that $X=\Sp(B)$.  Given an $A$-lattice $\Lambda$ in $B$, consider the exponential function $e_{\Lambda}:B\to B$ given by 
\[
e_{\Lambda}(x)=x \prod_{\substack{\lambda\in \Lambda \backslash\{0\}}}\Big(1-\frac{x}{\lambda}\Big).
\] This is an everywhere convergent power series and hence  defines a rigid analytic function on $B$, although unlike in the $\Sp(\mathbb{C}_\infty)$ case, $e_{\Lambda}$ may no longer be surjective.  
Then, for each $a\in A$, the polynomial $\phi_a(z) \in B[z]$ given by
\[
\phi_a(z):=az\prod_{\substack{h \in a^{-1}\Lambda /\Lambda \backslash \{0\}}}\Big(1-\frac{z}{e_{\Lambda}(h)}\Big)
\]
is the unique function satisfying $\phi_a(e_\Lambda(x)) = e_\Lambda(ax)$. Hence $\phi:A \to B[\tau]$ is a Drinfeld $A$-module over $\Sp(B)$. 

Now let $X$ be an arbitrary rigid analytic space over $L$ and  $(\mathcal{M},\underline{\Lambda},s)$ be an $A$-lattice of rank $r$ over $X$. We choose an admissible affinoid covering $\mathcal{U}:=\{U_i\}_{i \in \mathcal{Q}}$ of $X$ so that $\mathcal{M}$ is trivialized by $\mathcal{U}$. Each triple $(\mathcal{M}_{|U_i},\underline{\Lambda}_{|U_i},s_{|U_i})$  as described above determines a Drinfeld $A$-module $\phi_i:A \rightarrow \mathcal{M}(U_i)[\tau]$. Consequently, we obtain a Drinfeld $A$-module $\phi:A \rightarrow \Gamma_{X,\mathcal{M}}[\tau]$ such that $\phi_{|U_i}=\phi_i$ and it is independent of the choice of $\mathcal{U}$. Thus $\phi$ is a Drinfeld $A$ module over $X$ in the sense of Definition \ref{ana_drin} obtained from the $A$-lattice $(\mathcal{M},\underline{\Lambda},s)$. 

In what follows, we provide a fundamental example of a Drinfeld $A$-module defined over $\Omega$.

\begin{example} 
Recall the projective $A$-module $Y$ given as in \eqref{E:defofY} and 
denote by $\underline{Y}$ the constant sheaf on $\Omega$ determined by $Y$. We embed $\underline{Y}$ into $\mathcal{O}_\Omega$ via $s:\underline{Y} \rightarrow \mathcal{O}_\Omega$ sending $(g,h) \mapsto [z \mapsto gz+h]$.
One can see that the triple $(\mathcal{O}_\Omega,\underline{Y},s)$ is an $A$-lattice of rank $2$ over $\Omega$. Hence it gives rise to a Drinfeld $A$-module over $\Omega$, which we denote by $(\mathbb{G}_{a,\Omega},\bold\Psi^Y)$.  

 For each $i \geq 1$, we let $g_{i,a}:\Omega \rightarrow \mathbb{C}_\infty$ be the rigid analytic function so that \[\bold\Psi^{Y}_a=a + \sum_{i=1}^{r\deg(a)}g_{i,a}\tau^i.\] For any $z \in \Omega$, recall from \S2.1 that the $A$-lattice $Y_z:=\mathfrak{g}z+\mathfrak{h} \subset \mathbb{C}_\infty$ gives rise to a Drinfeld $A$-module of rank $2$. Then for any $a\in A$, we write
 \[
 \bold\Psi^{Y_z}_a=a + \sum_{i=1}^{r\deg(a)}g_{i,a}(z)\tau^i.
 \] 
 We refer the reader to Example \ref{Ex:1}(ii) for more details on the rigid analytic function $g_{i,a}$.
\end{example}

We can also describe an action of the group $\GL_2(K)$ on the Drinfeld $A$-module $(\mathbb{G}_{a,\Omega},\bold\Psi^Y)$ as follows. Firstly, we let 
\[
j(\gamma;z):=c_{\gamma}z+d_{\gamma}, \ \ \gamma=\begin{pmatrix}
a_{\gamma}&b_{\gamma}\\
c_{\gamma}& d_{\gamma}
\end{pmatrix}\in \GL_2(K).
\]
Now one can form the triple $\gamma_\ast(\mathcal{O}_\Omega,\underline{Y},s):=(\gamma_\ast\mathcal{O}_\Omega,\underline{Y}\gamma^{-1},j(\gamma;-)s^\gamma)$, where $s^\gamma:\underline{Y}\gamma^{-1} \rightarrow \gamma_\ast\mathcal{O}_\Omega$ sends $(g',h') \mapsto [z \mapsto g' (\gamma \cdot z) + h']$. Consequently, since $Y$ is preserved by $\Gamma_Y$,  we have an action of $\Gamma_Y$ on the Drinfeld $A$-module $(\mathbb{G}_{a,\Omega},\bold\Psi^Y)$. Due to the construction of the level $I$-structure on $(\mathbb{G}_{a,\Omega},\bold\Psi^Y)$ (see \cite[\S4.3]{Boc02} and \cite[\S7]{BBP21} for more details), we can descend the $(\mathbb{G}_{a,\Omega},\bold\Psi^Y)$, together with a canonical level $I$-structure, to a Drinfeld $A$-module with a level $I$-structure on $\Gamma_Y(I) \backslash \Omega$. Indeed, $(\mathbb{G}_{a,\Omega},\bold\Psi^Y)$ has a ``universal property'' which will be stated in our next proposition.  

Let $A_v$ ($K_v$ respectively) be the completion of $A$ ($K$ respectively) at a finite place $v$ and  $\hat{A}\cong \prod_{v\neq \infty}A_{v}$ be the profinite completion of $A$. We set $\mathbb{A}_f:=\hat{A}\otimes_{A} K$ and define $\GL_2(\mathbb{A}_f):=\prod_{v\neq \infty}' \GL_2(K_v)$ where $\prod'$ refers to the restricted product with respect to $\{ \GL_2(A_v)\}_{v\neq \infty}$. 

\begin{proposition}[{\cite[Thm.~4.15]{Boc02}, \cite[\S8]{BBP21}}]
	\label{2.4} Let $\mathcal{K}(I):=\{\gamma\in \GL_2(\hat{A})\ \ |\ \ \gamma \equiv \Id_2 \pmod{I}\}$. 
	There exists a natural isomorphism of rigid analytic spaces $$\GL_{2}(K) \backslash (\Omega \times \GL_{2}(\mathbb{A}_f)/\mathcal{K}(I)) \xrightarrow{\sim} M^{2}_{I}(\mathbb{C}_\infty).$$ 
 Let the composition of the above isomorphism with $\Omega \rightarrow \GL_{2}(K) \backslash (\Omega \times \GL_{2}(\mathbb{A}^f)/\mathcal{K}(I))$ sending $\omega \mapsto [(\omega,g)]$ be denoted by $\pi_g$. Then $$\pi^{\ast}_{g}(\mathbb{E}^{un}_{I}) \cong (\mathbb{G}_{a,\Omega},\bold\Psi^{M_g})$$ where $M_g$ is the projective $A$-module of rank 2 embedded in $K^2$ such that $M_g=\hat{A}^{2}g^{-1} \cap K^{2}$. 
\end{proposition} 

We note that one can choose an element $g \in \GL_2(\mathbb{A}_f)$ such that $M_g=Y$. In this case, by Proposition \ref{2.4}, there exists an isomorphism between $\Gamma_Y(I) \backslash \Omega$ and  the $\mathbb{C}_\infty$-points of a unique connected component $M_Y \subset M^2_{I} \times_K \mathbb{C}_\infty$. Hence this gives rise to a natural isomorphism 
\[
M_Y^{an} \cong \Gamma_Y(I) \backslash \Omega
\]
where  $M_Y^{an}$ is the analytification of $M_Y$. Moreover, the pullback of the universal Drinfeld $A$-module via the analytification map $\Gamma_Y(I) \backslash \Omega = M^{an}_Y \rightarrow M_Y$ is the descent of $(\mathbb{G}_{a,\Omega},\bold\Psi^Y)$ to $\Gamma_Y(I)\backslash \Omega$ described above.

Since $M_Y$ is a smooth curve over $\mathbb{C}_\infty$, it  admits a unique compactification, denoted by $\overline{M_Y}$.  Letting $\overline{M_Y}^{an}$ be the analytification of $\overline{M_Y}$, in \cite[Thm.~4.16]{Boc02}, B\"{o}ckle showed that the isomorphism $M_Y^{an} \cong \Gamma_Y(I) \backslash \Omega$ may be extended uniquely to an isomorphism 
\[
\overline{M_Y}^{an} \cong \overline{\Gamma_Y(I) \backslash \Omega}.
\]Note that the compactification $\overline{\Gamma_Y(I) \backslash \Omega}$ is denoted $\Gamma_Y(I) \backslash \overline{\Omega}$ in \cite[Thm.~4.16]{Boc02}.

\section{Drinfeld modular forms} Our next goal is to introduce Drinfeld modular forms and provide several examples of them. Throughout this subsection, we mainly follow \cite[Chap.~V, VI]{Gek86}. One can also refer to \cite{Gek90, Cor97, BBP21} for more details on Drinfeld modular forms.
 
Recall $ \Gamma_Y\subset \GL_2(K)$ from \S1.2. Since, by \cite[Part III]{BBP21}, for any integral ideal $\mathfrak{n}$, $\Gamma_{\mathfrak{n}^{-1}Y}=\Gamma_{Y}$, without loss of generality, we assume that $\mathfrak{g}$ and $\mathfrak{h}$ are integral ideals of $A$. We call $\Gamma \leqslant \GL_2(K)$ \textit{a congruence subgroup of $\GL_2(K)$} if for some non-zero ideal $\mathfrak{m}$ of $A$, we have $\Gamma_Y(\mathfrak{m})\leqslant \Gamma\leqslant \Gamma_Y$ where $\Gamma_Y(\mathfrak{m})$ is the principal congruence subgroup of level $\mathfrak{m}$ in $\Gamma_Y$. Note that if $Y=A^2$, then $\Gamma_Y=\GL_{2}(A)$.

From \S2.8, recall the function $j(\gamma;z)$ for any $\gamma\in \GL_2(K)$ and $z\in \Omega$. For any rigid analytic function $f:\Omega\to \mathbb{C}_{\infty}$ 
and for integers $k$ and $m$, we define \textit{the slash operator} by 
\[
f|_{k,m}\gamma:=\Big[z \mapsto j(\gamma;z)^{-k}\det(\gamma)^mf(\gamma\cdot z)=j(\gamma;z)^{-k}\det(\gamma)^mf\Big(\frac{a_{\gamma}z+b_{\gamma}}{c_{\gamma}z+d_{\gamma}}\Big)\Big].
\]
For $\gamma_1,\gamma_2\in \GL_2(K)$, we have (see \cite[(1.6)]{BBP21}) \begin{equation}\label{E:propslash}
    (f|_{k,m}\gamma_1\gamma_2)(z)=(f|_{k,m}\gamma_1)|_{k,m}\gamma_2)(z) \ \ z\in \Omega.
\end{equation}

Let $b$ be a cusp of $\Gamma$ and $\delta\in \GL_2(K)$ be such that $\delta\cdot b=\infty$. Recall from \S2.6 the fractional ideal $\mathfrak{a}_{\delta\Gamma \delta^{-1}}$ of $A$ as well as the function $ t_{\delta\Gamma \delta^{-1}}$ which is, for any $z\in \Omega$, given by
\begin{equation}
t_{\delta\Gamma \delta^{-1}}(z)=t_{\mathfrak{a}_{\delta\Gamma \delta^{-1}}}(z)=\exp_{\rho^{(\mathfrak{a}_{\delta\Gamma \delta^{-1}})}}(\xi(\mathfrak{a}_{\delta\Gamma \delta^{-1}})z)^{-1}=\xi(\mathfrak{a}_{\delta\Gamma \delta^{-1}})^{-1}\exp_{\mathfrak{a}_{\delta\Gamma \delta^{-1}}}(z)^{-1}.
\end{equation}
 We further call $t_{\delta\Gamma \delta^{-1}}$ \textit{the uniformizer at the cusp $b$ of $\Gamma$}.

  The next lemma will be crucial later to determine arithmetic properties of special values of Drinfeld modular forms.

\begin{lemma}\label{L:algoft} There exists a proper ideal $\mathfrak{m}$ of $A$ such that $t_{\Gamma}$ and $t_{\Gamma_{Y}}$ may be written as a power series in $t_{{\Gamma_Y(\mathfrak{m})}}$ whose coefficients are in $H$.
\end{lemma}
\begin{proof} Since, by definition, $t_{\Gamma_Y}=t_{\mathfrak{g}^{-1}\mathfrak{h}}$ and $t_{\Gamma}=t_{\mathfrak{a}_{\Gamma}}$ for some fractional ideal $\mathfrak{a}_{\Gamma}$ of $A$, the lemma follows from Proposition \ref{P:expression}.
\end{proof}
We call a rigid analytic function $f:\Omega\to \mathbb{C}_{\infty}$ \textit{a weak Drinfeld modular form of weight $k$ and type $m\in \mathbb{Z}/(q-1)\mathbb{Z}$ for $\Gamma$} if $f(\gamma\cdot z)=j(\gamma;z)^k\det(\gamma)^{-m}f(z)$ for any  $\gamma \in \Gamma$.

For $\delta\in \GL_2(K)$, note that, by \eqref{E:propslash}, we have $(f|_{k,m}\delta^{-1})(z+a)=(f|_{k,m}\delta^{-1})(z)$ for all $a\in \mathfrak{a}_{\delta\Gamma \delta^{-1}} $. Then there exists 
a unique power series expansion 
\[
f|_{k,m}\delta^{-1}=\sum_{i=-\infty}^{\infty}a_it_{\delta\Gamma \delta^{-1}}^i
\]
for some $a_i\in \mathbb{C}_{\infty}$, where the right hand side is an infinite sum of positive radius of convergence. We call such an infinite sum \textit{the $t_{\delta\Gamma \delta^{-1}}$-expansion of $f$}. Moreover,
we say that $f|_{k,m}\delta^{-1}$ is \textit{holomorphic at $\infty$} if 
\begin{equation}\label{E:holatinf}
f|_{k,m}\delta^{-1}=\sum_{i=0}^{\infty}a_it_{\delta\Gamma \delta^{-1}}^i.
\end{equation}

\begin{definition} We call a weak Drinfeld modular form of weight $k$ and type $m$ for $\Gamma$ \textit{a Drinfeld modular form of weight $k$ and type $m$ for $\Gamma$} if for any $\delta\in \GL_2(K)$, $f|_{k,m}\delta$ is holomorphic at $\infty$. 
\end{definition}
We denote by $\mathcal{M}_{k}^{m}(\Gamma)$ the $\mathbb{C}_{\infty}$-vector space generated by Drinfeld modular forms of weight $k$ and type $m\in \mathbb{Z}/(q-1)\mathbb{Z}$ for $\Gamma$. We also set $\mathcal{M}_{k}(\Gamma):=\cup_{m\in \mathbb{Z}/(q-1)\mathbb{Z}}\mathcal{M}_{k}^{m}(\Gamma)$.  Furthermore, for any field $L\subseteq \mathbb{C}_{\infty}$, we denote by $\mathcal{M}_{k}^{m}(\Gamma;L)$ the $L$-vector space of Drinfeld modular forms of weight $k$ and type $m$ for $\Gamma$ whose coefficients in the  $t_{\Gamma}$-expansion (the unique power series expansion at the cusp $\infty$) coefficients lie in $L$.

\begin{definition}  Let  $\overline{K}$ be a fixed algebraic closure of $K$ in $\mathbb{C}_{\infty}$. We call any rigid analytic function $f\in \mathcal{M}_k^m(\Gamma;\overline{K})$ \textit{an arithmetic Drinfeld modular form of weight $k$ and type $m$ for $\Gamma$}.
\end{definition}

We continue with our next definition.

\begin{definition} We call a meromorphic function $f:\Omega\to \mathbb{C}_{\infty}$ \textit{a meromorphic Drinfeld modular form of weight $k$ and type $m$ for $\Gamma$ over $L$} if it can be written as $f=h_1/h_2$ for some $h_1\in \mathcal{M}_{k+\tilde{k}}^{m+\tilde{m}}(\Gamma;L)$ and $h_2\in \mathcal{M}_{\tilde{k}}^{\tilde{m}}(\Gamma;L)$ for some $\tilde{k}\in \mathbb{Z}_{\geq 1}$ and $\tilde{m}\in \mathbb{Z}/(q-1)\mathbb{Z}$. The $L$-vector space spanned by all the meromorphic Drinfeld modular forms of weight $k$ and type $m$ for $\Gamma$ over $L$ is denoted by $\mathcal{A}_{k}^{m}(\Gamma;L)$. We further let 
\[
\mathcal{A}_k(\Gamma;L):=\bigcup_{m\in \mathbb{Z}/(q-1)\mathbb{Z}}\mathcal{A}_k^m(\Gamma;L) \text{ and } \mathcal{A}_k(L):=\bigcup_{\substack{\Gamma \text{ a congruence subgroup of } \Gamma_Y\\ m\in \mathbb{Z}/(q-1)\mathbb{Z}}}\mathcal{A}_k^m(\Gamma;L)
\]
We further call elements of $\mathcal{A}_0(\Gamma;L)$ \textit{Drinfeld modular functions for $\Gamma$ over $L$}.
\end{definition}

Our next goal is to provide some examples of Drinfeld modular forms and Drinfeld modular functions. 
\begin{example}\label{Ex:1}
\begin{itemize}
\item[(i)] 
Let $v +Y$ be a coset in $K^2$. We define \textit{the Eisenstein series of weight $k$ corresponding to $v$} by 
\[
E_{k,v}(z):=\sum_{(0,0)\neq (x_1,x_2)\in v +Y}\frac{1}{(x_1z+x_2)^k}.
\] 
Indeed, we know that, if $v\in \mathfrak{m}^{-1}Y\setminus Y$ then  
$E_{k,v}$ is a Drinfeld modular form of weight $k$ for $\Gamma_Y(\mathfrak{m})$. Moreover, by \cite[Chap.~VI, (3.9)]{Gek86}, \cite[Cor.~13.7]{BBP21}  and Lemma \ref{L:xi}, we have $\xi(\mathfrak{g}^{-1}\mathfrak{h})E_{1,v}\in \mathcal{M}_{1}^{0}(\Gamma_Y(\mathfrak{m});\overline{K})$.

Now for simplicity, set $E_k:=E_{k,(0,0)}$. Then $E_k$ is a Drinfeld modular form of weight $k$ and type 0 for $\Gamma_Y$. Moreover, we have 
\begin{equation}\label{E:eisenstein}
\begin{split}
E_k(z)&=\sum_{b\in \mathfrak{h}\setminus \{0\}}\frac{1}{b^k}+\sum_{a\in \mathfrak{g}\setminus \{0\}}\sum_{b\in \mathfrak{h}}\frac{1}{(az+b)^k}\\
&=\xi(\mathfrak{h})^k\sum_{b\in \mathfrak{h}\setminus \{0\}}\frac{1}{(\xi(\mathfrak{h})b)^k}+\xi(\mathfrak{h})^k\sum_{a\in \mathfrak{g}\setminus\{0\}}\sum_{b\in \mathfrak{h}}\frac{1}{(\xi(\mathfrak{h})az+\xi(\mathfrak{h})b)^k}\\
&=\xi(\mathfrak{h})^k\sum_{b\in \mathfrak{h}\setminus \{0\}}\frac{1}{(\xi(\mathfrak{h})b)^k}+\xi(\mathfrak{h})^k\sum_{a\in \mathfrak{g}\setminus\{0\}}G_k(t_{\mathfrak{h}}(az))
\end{split}
\end{equation}
where $G_k(X)\in H[X]$ is \textit{the $k$-th Goss polynomial} attached to the $\mathbb{F}_q$-vector space $\xi(\mathfrak{h})\mathfrak{h}$ (see \cite[Prop.~3.4]{Gek88} for more details on Goss polynomials). By \eqref{E:funceq}, the exponential series $\exp_{\rho^{(\mathfrak{h})}}$ of $\rho^{(\mathfrak{h})}$ is defined as a power series in $H[[\tau]]$. Hence, by \cite[Lem.~3.4.10]{Bas14}, we have 
$
\sum_{b\in \mathfrak{h}\setminus \{0\}}\frac{1}{(\xi(\mathfrak{h})b)^k}\in H.
$ On the other hand, by \cite[Chap.~VI, (3.3)]{Gek86} (see also \eqref{E:unifchange21}), for any non-zero $a\in \mathfrak{g}$, we have 
\begin{equation}\label{E:uniformizer}
t_{\mathfrak{h}}(az)=\mathcal{J}((a)\mathfrak{m}\mathfrak{g}^{-1},\mathfrak{m}\mathfrak{g}^{-1}\mathfrak{h})^{-1}\frac{1}{\rho_{(a)\mathfrak{m}\mathfrak{g}^{-1}}^{(\mathfrak{m}\mathfrak{g}^{-1}\mathfrak{h})}(t_{\Gamma_Y(\mathfrak{m})}(z)^{-1})}
\end{equation}
where $(a)=\mathfrak{g}\mathfrak{a}'$ for some integral ideal $\mathfrak{a}'$ of $A$. Thus, $t_{\mathfrak{h}}(az)$ may be written as a power series in $t_{\Gamma_Y(\mathfrak{m})}(z)$ with coefficients in $H$. Hence, by the above discussion and Lemma \ref{L:xi}, we obtain that $\xi(\mathfrak{g}^{-1}\mathfrak{h})^{-k}E_k\in \mathcal{M}_{k}^{0}(\Gamma_Y(\mathfrak{m});\overline{K})$.

\item[(ii)] For any $z\in \Omega$, recall from \S2.8 that $Y_z =\mathfrak{g}z+\mathfrak{h}$ and the Drinfeld $A$-module $\bold\Psi^{Y_{z}}$ given by 
\[
\bold\Psi^{Y_z}_a=a+\sum_{i=1}^{2\deg(a)}g_{i,a}(z)\tau^i
\]
for each $a\in A$. Here we realize $g_{i,a}$, which we call \textit{the $i$-th coefficient form},  as a rigid analytic function $g_{i,a}:\Omega\to \mathbb{C}_{\infty}$  sending $z$ to the $i$-th coefficient $g_{i,a}(z)$ of $\bold\Psi^{Y_z}_a$. Using the definition of Drinfeld $A$-modules, observe that $g_{a,2\deg(a)}$ is a non-vanishing function on $\Omega$. A simple observation on isomorphic $A$-lattices and their corresponding Drinfeld $A$-modules (see also \cite[Prop.~15.12]{BBP21}) implies 
\[
g_{i,a}(\gamma \cdot z)=j(\gamma;z)^{q^{i}-1}g_{i,a}(z), \ \ \gamma\in \Gamma_{Y}.
\]
Indeed, $g_{i,a}\in \mathcal{M}_{q^{i}-1}^{0}(\Gamma_{Y})$. Moreover, by \cite[pg.~251]{Gek90}, we have 
\begin{equation}\label{E:coefEis}
aE_{q^{i}-1}(z)=\sum_{\ell=0}^iE_{q^{\ell}-1}(z)(g_{i-\ell,a}(z))^{q^\ell}.
\end{equation}
Thus, by part (i) and \eqref{E:coefEis}, we further obtain 
\[
\xi(\mathfrak{g}^{-1}\mathfrak{h})^{1-q^{i}}g_{i,a}\in \mathcal{M}_{q^{i}-1}^{0}(\Gamma_Y;\overline{K}).
\]

\item[(iii)] Let $a\in A$ be a non-constant element. Consider the function $\mathfrak{j}_{i,w}:\Omega \to \mathbb{C}_{\infty}$ given by 
\[
\mathfrak{j}_{i,w}(z):=\frac{g_{i,a}(z)^{\ell_1}}{g_{w,a}(z)^{\ell_2}}
\]
so that $(q^i-1)\ell_1=(q^w-1)\ell_2$ and $\ell_1$ and $\ell_2$ are relatively prime. Then, by (i), for any $1\leq i,w \leq 2\deg(a)$, one can see that $\mathfrak{j}_{i,w}$ is indeed a Drinfeld modular function for $\Gamma_Y$. By the previous example, we see that $\mathfrak{j}_{i,w}\in \mathcal{A}_0(\Gamma_{Y};\overline{K})$. We also simply set
\[
\mathfrak{J}:=\frac{g_{\deg(a),a}^{q^{\deg(a)}+1}}{g_{2\deg(a),a}}.
\]
\item[(iv)]  Let $v\in \mathfrak{m}^{-1}Y\setminus Y$ and set $E_{v}:=E_{1,v}$. By \cite[Cor.~13.7]{BBP21}, $E_{v}$ is nowhere vanishing on $\Omega$. We consider the function $\mathbf{f}_v:\Omega \to \mathbb{C}_{\infty}$ given by 
\[
\mathbf{f}_v(z):=\frac{g_{1,a}(z)}{E_{v}(z)^{q-1}}.
\]
Then, by (i) and (ii), we have  $\mathbf{f}_v\in \mathcal{A}_0(\Gamma_Y(\mathfrak{m}))$. Moreover, if we let $\mathfrak{m}'$ be an integral ideal as in Lemma \ref{L:algoft} chosen for $\Gamma_Y(\mathfrak{m})$ and $\Gamma_{Y}$, we further see that  $\mathbf{f}_v\in \mathcal{A}_0(\Gamma_Y(\mathfrak{m}');\overline{K})$.
\end{itemize}
\end{example}

Let $\mathbb{C}_{\infty}(X(\Gamma_Y(\mathfrak{m})))$ be the function field  of the smooth model $X(\Gamma_Y(\mathfrak{m}))$  of $\Gamma_Y(\mathfrak{m})\setminus \Omega$. 

\begin{lemma}\label{L:funcfield} Choose a non-constant $a\in A$ of degree $d$. We have 
	\[
	\mathbb{C}_{\infty}(X(\Gamma_Y))=\mathbb{C}_{\infty}(\mathfrak{j}_{i,w}\mid 1\leq i,w \leq 2d).
	\]
\end{lemma}
\begin{proof} By \cite[Chap.~VII, Prop.~1.3]{Gek86}, we know that $\mathbb{C}_{\infty}(X(\Gamma_Y))$ is the field generated over $\mathbb{C}_{\infty}$ by Drinfeld modular functions in $\mathbb{C}_{\infty}(g_{i,a}| \ \ 1\leq i \leq 2d)$. Hence, one side of the inclusion is clear. We show the other direction. Let $\mathbb{C}_{\infty}[X_1,\dots,X_{2d}]$ be the graded polynomial ring where $X_i$ is of weight $q^i-1$ for each $1\leq i \leq 2d$. For each $\mu=1,2$, we let
	\[
	F_{\mu}:=\sum_{\underline{i}=(i_1,\dots,i_{2d})\in \mathbb{Z}^{2d}_{\geq 0}}\alpha_{\mu,\underline{i}}X_1^{i_1}\cdots X_{2d}^{i_{2d}}\in\mathbb{C}_{\infty}[X_1,\dots,X_{2d}]
	\]
	be a graded polynomial of weight $w$ namely, $i_1(q-1)+\dots+i_{2d}(q^{2d}-1)=w$ for each $\underline{i}$ so that $\alpha_{\mu,\underline{i}}\neq 0$. Set $J:=F_1(g_{1,a},\dots,g_{2d,a})/F_2(g_{1,a},\dots,g_{2d,a})$. Then $J\in \mathbb{C}_{\infty}(X(\Gamma_Y))$. We write $J=g_{1,a}^{-w}F_1(g_{1,a},\dots,g_{2d,a})/g_{1,a}^{-w}F_2(g_{1,a},\dots,g_{2d,a})$. Since $\gcd(q^i-1| \ \ 1\leq i \leq 2d)=q-1$ and each monomial in $F_{\mu}$ has degree $w$, one can see that 
	\[
	J\in \mathbb{C}_{\infty}(\mathfrak{j}_{i,w}\mid 1\leq i,w \leq 2d),
	\]
	finishing the proof of the lemma.
\end{proof}

Using \cite[Chap.~VII, Prop.~1.6]{Gek86} and Lemma \ref{L:funcfield}, one immediately deduces our next lemma.
\begin{lemma}\label{L:funcfieldN} We have 
	\[
	\mathbb{C}_{\infty}(X(\Gamma_Y(\mathfrak{m})))=\mathbb{C}_{\infty}(\mathfrak{j}_{i,w},\mathbf{f}_{v}\mid 1\leq i,w \leq 2d,\ \  v=(v_1,v_2)\in \mathfrak{m}^{-1}Y\setminus Y ).
	\]
\end{lemma}

Let $\Delta_a:=g_{2\deg(a),a}\in \mathcal{M}_{q^{2\deg(a)}-1}^{0}(\Gamma_Y)$. In what follows, we determine a product formula for $\Delta_a$. For any integral ideal $\mathfrak{p}$ of $A$, let
\[
\mathcal{P}_{\mathfrak{p}}^{(\mathfrak{g}^{-1}\mathfrak{h})}(X):=\rho_{\mathfrak{p}}^{(\mathfrak{g}^{-1}\mathfrak{h})}(X^{-1})X^{q^{\deg(\mathfrak{p})}}.
\]
Observe, by \cite[Chap.~VI, (1.3)]{Gek86}, that  $\mathcal{P}_{\mathfrak{p}}^{(\mathfrak{g}^{-1}\mathfrak{h})}(X)\in H[X]$ is a polynomial of degree $q^{\deg(\mathfrak{p})}-1$ satisfying $\mathcal{P}_{\mathfrak{p}}^{(\mathfrak{g}^{-1}\mathfrak{h})}(0)=1$. 

\begin{proposition} [{\cite[Chap.~VI, (4.12)]{Gek86}}] \label{P:prod} For any $z\in \Omega$ in some neighborhood of infinity, we have 
	\[
	\xi(\mathfrak{h})^{1-q^{2\deg(a)}}\Delta_a(z)=ct_{\Gamma_Y}(z)^k\prod_{\mathfrak{p}}\mathcal{P}_{\mathfrak{p}}^{(\mathfrak{g}^{-1}\mathfrak{h})}(t_{\Gamma_Y}(z))^{(q^{2\deg(a)}-1)(q-1)}
	\]
	where the product runs over the integral ideals $\mathfrak{p}$ of $A$ satisfying $\mathfrak{p}=(m)\mathfrak{g}^{-1}$ for $(q-1)$-many different elements $m\in \mathfrak{g}$. In the formula, $k$ is a non-negative integer divisible by $q^{\mathfrak{d}}-1$ and $c\in \mathbb{F}_{q^{\mathfrak{d}}}^{\times}$. Moreover, $\Delta_a$ has $t_{\Gamma_Y}(z)$-powers not divisible by $q$ and the factors $\xi(\mathfrak{h})^{1-q^{2\deg(a)}}$, $c$ as well as $k$ do not depend on the choice of Drinfeld-Hayes $A$-module $\rho^{(\mathfrak{g}^{-1}\mathfrak{h})}$.
\end{proposition}

\section{Nearly holomorphic Drinfeld modular forms}
Our goal in this section is to illustrate the fundamental objects of the present paper, namely the nearly holomorphic Drinfeld modular forms. We note that many of the results here follow mainly from the ideas in \cite[\S3]{CG23} and hence we will refer the reader to the suitable references. However, there is a certain difference between this general setting and the $A=\mathbb{F}_q[\theta]$ case on determining a non-trivial nearly holomorphic Drinfeld modular form and analyzing its behavior at the cusp $\infty$. It will be emphasized explicitly in this section.

Recall the Frobenius map $\sigma: \widehat{K_{\infty}^\text{nr}} \to \widehat{K_{\infty}^\text{nr}}$ from \S1 which is defined by 
\[
\sigma\Big(\sum_{i\geq i_0}a_i\pi_{\infty}^i \Big):=\sum_{i\geq i_0}a_i^{q^{\mathfrak{d}}}\pi_{\infty}^i.
\] 
It is easy to see that $\sigma$ is a continuous map on $\widehat{K_{\infty}^\text{nr}}$ and it fixes elements of $K_{\infty}$. Let $M$ be a fixed extension of $\Knr$ and $\varphi$ be a fixed extension of $\sigma$ as defined in \S1. Recall also that $\Omega^{\varphi}(M)=M\setminus M^{\varphi}$. In what follows, we state a useful lemma whose proof is along the same lines as the proof of \cite[Thm.~6.2.1]{CG23} after noting that $K_{\infty}=\mathbb{F}_{q^{\mathfrak{d}}}((\pi_{\infty}))$.
\begin{lemma}\cite[Thm.~6.2.1]{CG23} \label{L:ext} For any CM point $z_0\in \Omega$, there exists an extension $M_{z_0}$ of $\Knr$ containing $z_0$ and an extension $\varphi_{z_0}$ of $\sigma$.
\end{lemma}

We summarize the main results derived from \cite[\S3.1]{CG23} in our general setting as follows.
\begin{theorem}[{see \cite[Thm.~3.1.4, Prop. 3.1.9, Thm. 3.1.10]{CG23}}]\label{T:1} Let $C(\Omega^{\varphi}(M),\mathbb{C}_{\infty})$ be the space of $\mathbb{C}_{\infty}$-valued continuous functions on $\Omega^{\varphi}(M)$ and $\mathcal{O}$ be the ring of rigid analytic functions on $\Omega$. 
	\begin{itemize}
 \item[(i)] There exists no meromorphic function $f$ on $\Omega$ such that $f{|_{\Omega^{\varphi}(M)}}=\varphi$.
		\item[(ii)] The restriction map $\mathcal{O}\to C(\Omega^{\varphi}(M),\mathbb{C}_{\infty})$ is injective.
		\item[(iii)] Let $X$ be a variable over $\mathcal{O}$. The function $\iota:\mathcal{O}[X]\to  C(\Omega^{\varphi}(M),\mathbb{C}_{\infty})$ sending $X\to \frac{1}{\Id-\varphi}$ is an injective ring homomorphism.
	\end{itemize}
\end{theorem}
For any $f:\Omega^{\varphi}(M)\to \mathbb{C}_{\infty}$ and $m,k\in \mathbb{Z}$, we consider the slash operator
\[
(f|_{k,m}\gamma)(z):=j(\gamma;z)^{-k}\det(\gamma)^{m}f\Big(\frac{a_{\gamma}z+b_{\gamma}}{c_{\gamma}z+d_{\gamma}}\Big), \ \ \gamma=\begin{pmatrix}
a_{\gamma}&b_{\gamma}\\
c_{\gamma}& d_{\gamma}
\end{pmatrix}\in \GL_2(K).
\]

\begin{definition}\label{D:1} Let $\Gamma $ be a congruence subgroup of $\Gamma_Y$.  A continuous function $F:\Omega^{\varphi}(M)\to \mathbb{C}_{\infty}$ is called \textit{a weak nearly holomorphic Drinfeld modular form of weight $k$, type $m\in \mathbb{Z}/(q-1)\mathbb{Z}$ and depth $r$ for $\Gamma$} if 
	\begin{itemize}
		\item[(i)] for any $\gamma \in \Gamma$,
		\[
		F|_{k,m}\gamma=F,
		\] 
		\item[(ii)] there exist rigid analytic functions $f_0,\dots,f_r$ with $f_r\neq 0$ satisfying 
		\[
		F(z)=\sum_{i=0}^r\frac{f_i(z)}{(\xi(\mathfrak{g}^{-1}\mathfrak{h})z-\xi(\mathfrak{g}^{-1}\mathfrak{h})\varphi(z))^i}.
		\]
\end{itemize}

In what follows, we state an important property of the function $1/(\Id-\varphi)$.
\begin{lemma}\label{L:1} For any $z\in \Omega^{\varphi}(M)$ and $\gamma \in \GL_2(K_{\infty})$, we have 
	\[
	\frac{1}{\gamma\cdot z-\varphi(\gamma\cdot z)}=j(\gamma;z)^2\det(\gamma)^{-1}\Big( \frac{1}{z-\varphi(z)}-\frac{c_{\gamma}}{j(\gamma;z)} \Big).
	\]
\end{lemma}
\begin{proof} Since $\GL_2(K_{\infty})$ is generated by $\begin{pmatrix}
	0&1\\
	1&0
	\end{pmatrix}$ and  upper triangular matrices, the lemma follows from the same calculation done in the proof of \cite[Prop.~3.2.5]{CG23}.
\end{proof}

The next lemma immediately follows from \cite[Prop.~3.2.12]{CG23}.
\begin{lemma}\label{L:2} Let $F=\sum_{i=0}^r\frac{f_i}{\xi(\mathfrak{g}^{-1}\mathfrak{h})^i(\Id-\varphi)^i}$ be a weak nearly holomorphic Drinfeld modular form of weight $k$, type $m$ and depth $r$ for $\Gamma$. For any $\gamma=\begin{pmatrix}
    a&b\\
    c&d
\end{pmatrix}\in \Gamma$ and $z\in \Omega$, set  $\mathscr{J}(\gamma;z):=\frac{c}{j(\gamma;z)}$. Then we have
	\begin{align*}
	f_{r-i}(\gamma \cdot  z)&=j(\gamma;z)^{k-2(r-i)}\det(\gamma)^{-m+r-i}\bigg(f_{r-i}(z)+\mathscr{J}(\gamma;z)\binom{r-i+1}{1}f_{r-i+1}(z)\\
	&\ \  \ \ \ +\mathscr{J}(\gamma;z)^{2}\binom{r-i+2}{2}f_{r-i+2}(z)+\dots +\mathscr{J}(\gamma;z)^{i}\binom{r}{i}f_{r}(z)\bigg).
	\end{align*}
\end{lemma}

For any $\alpha\in \GL_2(K)$ and a weak nearly holomorphic Drinfeld modular form $F$ of weight $k$, type $m$ and depth $r$ for $\Gamma$, using \eqref{E:propslash}, which can be easily applied to our setting, we note that $F|_{k,m}\alpha$ is also a weak nearly holomorphic Drinfeld modular form $F$ of weight $k$, type $m$ and depth $r$ for $\alpha^{-1}\Gamma\alpha$ and we  write 
\[
F|_{k,m}\alpha=\sum_{i=0}^r\frac{\mathfrak{G}_i}{\xi(\mathfrak{g}^{-1}\mathfrak{h})^i(\Id-\varphi)^i}
\]
for some rigid analytic functions $\mathfrak{G}_0,\dots,\mathfrak{G}_r$ on $\Omega$. 
On the other hand, letting $\gamma=\begin{pmatrix}
    1&a\\
    0&1
\end{pmatrix}\in \alpha^{-1}\Gamma\alpha$, by Lemma \ref{L:2}, we see that 
\[
\mathfrak{G}_i(\gamma\cdot z)=\mathfrak{G}_i(z+a)=\mathfrak{G}_i(z)
\]
for all $i=1,\dots,r$ and $z\in \Omega$. Therefore, as discussed after Lemma \ref{L:algoft}, each $\mathfrak{G}_i$ admits a $t_{\alpha^{-1}\Gamma \alpha}$-expansion and hence there exists a Laurent series expansion 
		\[
			(F|_{k,m}\alpha)(z)=\sum_{i=0}^r\frac{1}{(\xi(\mathfrak{g}^{-1}\mathfrak{h})z-\xi(\mathfrak{g}^{-1}\mathfrak{h})\varphi(z))^{i}}\sum_{\ell=-\infty}^{\infty}a_{\alpha,i,\ell}t_{\alpha^{-1}\Gamma\alpha}(z)^{\ell}
		\]
		for some $a_{\alpha,i,\ell}\in \mathbb{C}_{\infty}$ valid on $\{z\in \Omega^{\varphi}(M)\ \ | \ \ \inorm{t_{\alpha^{-1}\Gamma\alpha}(z)}<c_\alpha\}$ for some $c_{\alpha}>0$. For each $0\leq i \leq r$,  set $Q_F^{(i)}(t_{\alpha^{-1}\Gamma \alpha}):= \sum_{\ell=-\infty}^{\infty}a_{\alpha,i,\ell}t_{\alpha^{-1}\Gamma\alpha}(z)^{\ell}$. We further call the family $\{Q_F^{(i)}(t_{\alpha^{-1}\Gamma \alpha})\}_{0\leq i \leq r}$ \textit{the $t_{\alpha^{-1}\Gamma \alpha}$-expansion of $F$}.	
        
\end{definition}

In what follows, we define nearly holomorphic Drinfeld modular forms.
\begin{definition}\label{D:2} We call a weak nearly holomorphic Drinfeld modular form $F$ of weight $k$, type $m$ and depth $r$ for $\Gamma$  \textit{a nearly holomorphic Drinfeld modular form} if for any $\alpha\in \GL_2(K)$, we have 
		\[
		(F|_{k,m}\alpha)(z)=\sum_{i=0}^r\frac{1}{(\xi(\mathfrak{g}^{-1}\mathfrak{h})z-\xi(\mathfrak{g}^{-1}\mathfrak{h})\varphi(z))^{i}}\sum_{\ell= 0}^{\infty}a_{\alpha,i,\ell}t_{\alpha^{-1}\Gamma\alpha}(z)^{\ell},
		\]
for some $a_{\alpha,i,\ell} \in \mathbb{C}_{\infty}$		 whenever $|z|_{\text{im}}$ is sufficiently large.
\end{definition}
\begin{remark} By Theorem \ref{T:1}(ii), each $f_i$ is uniquely determined by $F$. 
\end{remark}

We denote by $\mathcal{N}_k^{m,\leq r}(\Gamma)$  the $\mathbb{C}_{\infty}$-vector space of nearly holomorphic Drinfeld modular forms of weight $k$, type $m$ and depth less than or equal to $r$ for $\Gamma$. 

We further let 
\[
\mathcal{N}:=\bigcup_{\substack{k\in \mathbb{Z}_{\geq 0}\\ m\in \mathbb{Z}/(q-1)\mathbb{Z}\\ r\in \mathbb{Z}_{\geq 0} \\ \Gamma \text{ a congruence subgroup of $\GL_2(K)$}}}\mathcal{N}_k^{m,\leq r}(\Gamma).
\]

\begin{remark} \label{R:mf_nhdf}

Putting $r=0$, we have by Theorem \ref{T:1}(ii) that $\mathcal{N}^{m,\leq 0}_k(\Gamma)\cong\mathcal{M}^m_k(\Gamma)$. Hence the notion of nearly holomorphic Drinfeld modular forms generalizes Drinfeld modular forms. 
\end{remark}

Our goal now is to introduce an example of an element in $\mathcal{N}$ of depth one. Let $a$ be a non-constant element in $A$ and $\partial_z(g)$ denote the derivative of a rigid analytic function $g$ with respect to $z$. Consider the function $E:\Omega\to \mathbb{C}_{\infty}$ given by 
\[
E:=\frac{1}{\xi(\mathfrak{g}^{-1}\mathfrak{h})}\frac{\partial_z(\Delta_a)}{\Delta_a}.
\]
Since $\Delta_a$ is nowhere vanishing on $\Omega$, $E$ is a rigid analytic function. Moreover, by \cite[(6.7)]{Gek90}, $E$ does not depend on the choice of $a$.  Using the definition, for any $\gamma\in \GL_2(K)$, we immediately see that 
\begin{equation}\label{4.5}
	E(\gamma\cdot z)=
	j(\gamma;z)^2\det(\gamma)^{-1}\Big(E(z)-\frac{\xi(\mathfrak{g}^{-1}\mathfrak{h})^{-1}c_{\gamma}}{j(\gamma;z)}\Big).
\end{equation}
Furthermore, since the product formula for $\Delta_a$ given in Proposition \ref{P:prod} contains powers of $t$ not divisible by $q$, $\partial_z(\Delta_a)$ is not identically zero. Then taking  the logarithmic derivative of $\Delta_a$ implies that $E$ is holomorphic at $\infty$.

In what follows, we record a fundamental property of $E$.

\begin{lemma}\label{fEis} We have 
	\[
	E(z)=-\frac{1}{\mathcal{J}(\mathfrak{g}^{-1}\mathfrak{h})\partial(\rho_{\mathfrak{g}}^{(\mathfrak{h})})}\sum_{a\in \mathfrak{g}\setminus \{0\}}at_{\mathfrak{h}}(az)
	\]
	provided that the right hand side converges for all $z\in \Omega$ so that $|z|_{\text{im}}$ is sufficiently large.
\end{lemma}
\begin{proof}   By \cite[Cor.~7.10]{Gek90}, we have that 
	\[
	\xi(\mathfrak{g}^{-1}\mathfrak{h})E(z)=-\lim_{S\to \infty}\sum_{\substack{a\in \mathfrak{g},b\in \mathfrak{h}\\(a,b)\neq (0,0)\\\inorm{az+b}\leq S}}\frac{a}{az+b}.
	\]
	Note first that, by Example \ref{Ex:1}(ii), since $\Delta_a(z)$ has a unique $t_{\mathfrak{h}}(z)$-expansion, $E(z)$ also has a unique $t_{\mathfrak{h}}(z)$-expansion which determines $E$ whenever $|z|_{\text{im}}$ is sufficiently large. Hence, for simplicity, let $z\in \Omega$ be sufficiently large and be such that $\log_q|z|\in \mathbb{Q}\setminus \mathbb{Z}$ which guarantees that $|z|_{\text{im}}=|z|$. Then $|az+b|=\max\{|az|,|b|\}$. Let $S_1$ and $S_2$ be non-negative integers, depending on $S$ and $z$, such that if $|az+b|\leq S$, then $\deg(a)\leq S_1$ and $\deg(b)\leq S_2$. Define the finite $\mathbb{F}_q$-vector space $\mathfrak{h}_{S_2}:=\{b\in \mathfrak{h} | \ \ \deg(b)\leq S_2\}$. Then, letting $\exp_{\mathfrak{h}_{S_2}}(z):=z\prod_{\lambda\in \mathfrak{h}_{S_2}\setminus\{0\}}\Big(1-\frac{z}{\lambda}\Big)$, by \cite[Prop.~3.4 (i, v)]{Gek88}, we obtain 
\[
	\frac{1}{\exp_{\mathfrak{h}_{S_2}}(z)}=\sum_{b\in \mathfrak{h}_{S_2}}\frac{1}{z-b}.
	\]    
	Therefore, we have 
\[
\sum_{\substack{a\in \mathfrak{g}, b\in \mathfrak{h}\\(a,b)\neq (0,0)\\\inorm{az+b}\leq S}}\frac{a}{az+b}=\sum_{\substack{a\in \mathfrak{g}\setminus \{0\}\\ \deg(a)\leq S_1}}\frac{a}{\exp_{\mathfrak{h}_{S_2}}(az)}.
\]    
Thus, by letting $S_1,S_2\to \infty$, we now obtain
\[
\xi(\mathfrak{g}^{-1}\mathfrak{h})E(z)=-\lim_{S\to \infty}\sum_{\substack{a\in \mathfrak{g},b\in \mathfrak{h}\\(a,b)\neq (0,0)\\\inorm{az+b}\leq S}}\frac{a}{az+b}=-\sum_{a\in \mathfrak{g}\setminus \{0\}}\frac{a}{\exp_{\mathfrak{h}}(az)}=-\xi(\mathfrak{h})\sum_{a\in \mathfrak{g}\setminus \{0\}}at_{\mathfrak{h}}(az)
\]
where the last equality follows from the fact that $t_{\mathfrak{h}}(az)=\xi(\mathfrak{h})^{-1}\exp_{\mathfrak{h}}(az)^{-1}$ for all $a\in A\setminus\{0\}$. Now, using \eqref{E:relxi}, we finally obtain the desired result. 
\end{proof}

The next result is a consequence of Lemma \ref{L:1} and the above discussion.
\begin{lemma}\label{L:E2} The function $E_2: \Omega^{\varphi}(M)\to \mathbb{C}_{\infty}$ given by 
	\[
	E_2(z):=E(z)-\frac{1}{\xi(\mathfrak{g}^{-1}\mathfrak{h})z-\xi(\mathfrak{g}^{-1}\mathfrak{h})\varphi(z)}
	\]
	is an element in $\mathcal{N}_2^{1,\leq 1}(\Gamma_Y)$.
\end{lemma}

We finish this section with a result on the decomposition of the space $\mathcal{N}_k^{m,\leq r}(\Gamma)$. We note that it follows easily from Lemma \ref{L:2}, noting that $f_r\in \mathcal{M}_{k-2r}^{m-r}(\Gamma)$ if $F\in \mathcal{N}_{k}^{m,\leq r}(\Gamma)$, and the method used in the proof of \cite[Thm.~3.2.18,  Cor.~3.2.20]{CG23}.

\begin{proposition}\label{P:str}
	 Any $F\in\mathcal{N}_k^{m,\leq r}(\Gamma)$ can be uniquely expressed as
	\begin{equation}\label{E:Repr1}
	F=\sum_{0\leq j\leq r}g_jE_2^j
	\end{equation}
	with $g_j\in\mathcal{M}_{k-2j}^{m-j} (\Gamma)$. In particular, $r\leq k/2$.
\end{proposition}

\begin{definition} We call $F\in \mathcal{N}_k^{m,\leq r}(\Gamma)$ \textit{an arithmetic nearly holomorphic Drinfeld modular form of weight $k$, type $m$ and depth at most $r$ for $\Gamma$} if it can be written as in \eqref{E:Repr1} where each $g_j$ is an arithmetic Drinfeld modular form.
\end{definition}

\begin{remark} We remark that any arithmetic Drinfeld modular form is also an arithmetic nearly holomorphic Drinfeld modular form. As it is clear from the definition, $E_2$ is an arithmetic nearly holomorphic Drinfeld modular form.
\end{remark}

\subsection{Maass-Shimura operators} In this subsection, following \cite[\S4]{CG23}, we introduce the Maass-Shimura operator $\delta_k^r$ for any $k,r\in \mathbb{Z}_{\geq 0}$.

Let $f:\Omega\to \mathbb{C}_{\infty}$ be a holomorphic function and $n$ be a non-negative integer. We define \textit{the $n$-th hyperderivative $\D^nf:\Omega\to \mathbb{C}_{\infty}$ of $f$} by 
\[
f(z+\epsilon)=\sum_{n \geq 0}(\D^nf)(z)\epsilon^n
\]
where $\epsilon\in \mathbb{C}_{\infty}$, so that $\inorm{\epsilon}$ is sufficiently small. Observe that $\D^1=\partial_z$. We refer the reader to \cite[\S3.1]{BP08}  for more details on hyperderivatives.

In what follows, let $\der^n:=\frac{\D^n}{\xi(\mathfrak{g}^{-1}\mathfrak{h})^n}$. 
\begin{definition} 
	\begin{itemize}
		\item[(i)] Let $\mu\in \mathbb{Z}_{\geq 0}$ be such that that $k\geq 2\mu$. We define \textit{the Maass-Shimura operator} $\delta_k^r$  by $\delta_k^r:=\Id$ for $r=0$ and 
\begin{multline*}
\delta_k^r\Bigg(\frac{f}{(\xi(\mathfrak{g}^{-1}\mathfrak{h})\Id-\xi(\mathfrak{g}^{-1}\mathfrak{h})\varphi)^{\mu}}\Bigg):=\\\frac{1}{(\xi(\mathfrak{g}^{-1}\mathfrak{h})\Id-\xi(\mathfrak{g}^{-1}\mathfrak{h})\varphi)^{\mu}}\sum_{i=0}^r\binom{k-\mu+r-1}{i}\frac{\der^{r-i}f}{(\xi(\mathfrak{g}^{-1}\mathfrak{h})\Id-\xi(\mathfrak{g}^{-1}\mathfrak{h})\varphi)^i}, \ \ r\geq 1.
\end{multline*}
 For convenience, we further set $\der:=\der^1$ and  $\delta_{k}:=\delta_{k}^{1}$.
	\item[(ii)] For any $F=\sum_{\mu=0}^r\frac{f_\mu}{(\xi(\mathfrak{g}^{-1}\mathfrak{h})\Id-\xi(\mathfrak{g}^{-1}\mathfrak{h})\varphi)^{\mu}}\in \mathcal{N}_{k}^{m,\leq r}(\Gamma)$, we set 
	\[
	\delta_{k}^r(F):=\sum_{\mu=0}^r\delta_{k}^r\Bigg(\frac{f_{\mu}}{(\xi(\mathfrak{g}^{-1}\mathfrak{h})\Id-\xi(\mathfrak{g}^{-1}\mathfrak{h})\varphi)^{\mu}}\Bigg).
	\]
\end{itemize}
\end{definition}
 
In our general setting, one can still apply the idea of the proof of \cite[Lem.~4.1.2, Prop. 4.1.3]{CG23} as well as Lemma \ref{L:xi} to obtain the following properties of $\delta_k^r$.

\begin{lemma}\label{L:MaassShimura} 
\begin{itemize}
\item[(i)] Let $f,g\in \mathcal{O}$. For any non-negative integers $k$ and $\ell$, we have 
\[
\delta_{k+\ell}(fg)=f\delta_k(g)+g\delta_{\ell}(f).
\]
\item[(ii)] For any $\gamma\in \GL_2(K)$, we have 
\[
\delta_k^r(f)|_{k+2r,m+r}\gamma =\delta_k^r(f|_{k,m}\gamma).
\]
\item[(iii)] Let $\Gamma$ be a congruence subgroup. The Maass-Shimura operator $\delta_k^r$ sends an arithmetic Drinfeld modular form of weight $k$ for $\Gamma$ to an arithmetic nearly holomorphic Drinfeld modular form of weight $k+2r$ for $\Gamma$.
\end{itemize}
\end{lemma}
\begin{proof} We only provide a proof for the third assertion. By (ii), we see that  $\delta_k^r$ sends Drinfeld modular forms of weight $k$ for $\Gamma$ to nearly holomorphic Drinfeld modular forms of weight $k+2r$ for the same congruence subgroup. On the other hand, by an immediate modification of \cite[Lem.~3.5]{BP08} to our setting, we see that if $f=\sum_{i\geq 0}a_it_{\Gamma_{Y}}^i$ where $a_i\in \overline{K}$, then for any $n\geq 0$, we have 
	$\D^n(f)=\sum_{i\geq 0}\xi(\mathfrak{a}_{\Gamma})^{n}c_it_{\Gamma_{Y}}^i$ where $c_i\in \overline{K}$. Since, by Lemma \ref{L:xi}, $\xi(\mathfrak{a}_{\Gamma})$ is an algebraic multiple of $\xi(\mathfrak{g}^{-1}\mathfrak{h})$, we obtain that the $t_{\Gamma}$-expansion of $\der^n(f)$ has algebraic coefficients. Now by using the definition of the Maass-Shimura operator $\delta_k^r$ and $E_2$, one can write $\delta_k^r(f)=\sum_{i=0}^{r} g_iE_2^i$ for some arithmetic Drinfeld modular forms $g_i$ for $\Gamma$ as desired.
\end{proof}

\section{Special values of nearly holomorphic Drinfeld modular forms at CM points}
Our aim in this section is to achieve analogues of some classical results on the special values of nearly holomorphic Drinfeld modular forms. Recall the projective $A$-module $Y$ given as in \eqref{E:defofY}. For any $z\in \Omega$, recall from \S2.8 that $Y_z=\mathfrak{g}z+\mathfrak{h}$. Observe that the uniformizer at the $\infty$-cusp in this case is given by $t_{\Gamma_Y}=\exp_{\rho^{(\mathfrak{g}^{-1}\mathfrak{h})}}(\xi(\mathfrak{g}^{-1}\mathfrak{h})z)^{-1}$.

Recall from \S1.4 that for any ideal $I$ of $A$, $V(I)$ is the set of prime ideals of $A$ dividing $I$. Assume that $|V(I)|\geq 2$. Consider the ring $\mathcal{S}$ so that $\Spec(\mathcal{S})=M^{2}_{I}$. Recall the universal Drinfeld $A$-module $\mathbb{E}_{I}^{un,2}=(\mathcal{L}^{un,2},\phi^{un,2})$ over $M^{2}_{I}$. From now on, to ease the notation, we set
\[
\mathbb{E}_{I}^{un}:=(\mathcal{L}^{un},\phi^{un}):=(\mathcal{L}^{un,2},\phi^{un,2}).
\]
Fixing a non-constant $a \in A$, we let
 \[
\phi^{un}_a:=\sum^{2\deg(a)}_{i=0} \tilde{g}_{i,a}\tau^i\in \End(\mathcal{L}^{un}) .
 \] 

Let $1\leq i,w \leq 2\deg(a)$. As in Example \ref{Ex:1}(iii), consider the relatively prime positive integers $\ell_1$ and $\ell_2$ such that $(q^i-1)\ell_1=(q^w-1)\ell_2$.
Recalling from Remark \ref{R:triv_univ} that $\mathcal{L}^{un}$ is trivial, we set
\[
\tilde{\mathfrak{j}}_{i,w}:=\frac{\tilde{g}_{i,a}^{\ell_1}}{\tilde{g}_{w,a}^{\ell_2}} \in \Frac(\mathcal{S}).
\]
 On the other hand, observe that by definition, since $\tilde{g}_{2\deg(a),a} \in H^{0}(M^2_I,(\mathcal{L}^{un})^{\otimes(1-q^{2\deg(a)})})$ is a non-vanishing section, we have 
\[
\tilde{\mathfrak{J}}:=\tilde{\mathfrak{j}}_{\deg(a),2\deg(a)} \in \mathcal{S}.
\]

Our next lemma is a consequence of a result of Drinfeld (see \cite[Prop.~9.3(i)]{Dri74}) combined with \cite[Prop.~5.1(iii)]{AM69} in the rank two case (see also \cite[Chap.~4, Prop.~2.3(iii)]{Leh09}).
\begin{lemma} \label{2.13}
 The injective map $A[\widetilde{\mathfrak{J}}] \hookrightarrow \mathcal{S}$ is a finite and an integral map. In particular, the extension $\Frac(\mathcal{S})/K(\tilde{\mathfrak{J}})$ of function fields is finite. Consequently, $\tilde{\mathfrak{j}}_{i,w} \in \Frac(\mathcal{S})$ is algebraic over $K(\widetilde{\mathfrak{J}})$.   
\end{lemma}

Let $\overline{K}(X(\Gamma_Y(\mathfrak{m})))$ be the function field over $\overline{K}$ of the smooth model $X(\Gamma_Y(\mathfrak{m}))$ over $\overline{K}$ of $\Gamma_Y(\mathfrak{m})\setminus \Omega$. Our next proposition is the key step to prove Theorem \ref{T:specvalue2}.
\begin{proposition}\label{P:jvalues} The following statements hold.
	\begin{itemize}
		\item[(i)] Let $\mathfrak{m}$ be an ideal of $A$ and $f\in \mathcal{A}_0(\Gamma_Y(\mathfrak{m}),\overline{K})$. We have 
		\[
		f\in \overline{K}(X(\Gamma_Y(\mathfrak{m})))=\overline{K}(\mathfrak{j}_{i,w},\mathbf{f}_{v}\mid 1\leq i,w \leq 2\deg(a),\ \  v=(v_1,v_2)\in \mathfrak{m}^{-1}Y\setminus Y).
		\]
        where $a$ is any non-constant element of $A$.
		\item[(ii)] Let $I$ be an ideal as before. Consider the function $\mathfrak{J}:=\frac{g_{i,a}^{q^{\deg(a)}+1}}{g_{2\deg(a),a}}$. Then the field $ \overline{K}(X(\Gamma_Y))=\overline{K}(\mathfrak{j}_{i,w} \ \ | \ \  1\leq i,w \leq 2\deg(a))$ is a finite algebraic extension of $\overline{K}(\mathfrak{J})$. In particular, the field $ \overline{K}(X(\Gamma_Y(I)))$ is a finite algebraic extension of $\overline{K}(\mathfrak{J})$.
	\end{itemize}
	\begin{proof} The first assertion follows from  Lemma \ref{L:funcfieldN} and the same idea used in the proof of \cite[Lem.~6.1.2]{CG23}. For the first part of the second assertion, note that, by Proposition \ref{2.4} and the remark below it, there exists a $g \in \GL_2(\mathbb{A}_f)$ such that $\pi^{\ast}_{g}(\mathbb{E}^{un}_I) \cong (\mathbb{G}_{a,\Omega},\phi^{Y})$. Then there exists an explicit trivialization $\pi^{\ast}_{g}\mathcal{L}^{un}  \cong \mathcal{O}_{\Omega}.$ Note that, by  \cite[\S10]{BBP21}, the section $\tilde{g}_i \in H^{0}(M^{2}_I,(\mathcal{L}^{un})^{\otimes (1-q^i)})$ induces the holomorphic function $\pi^{\ast}_g (\tilde{g}_i)=g_i:\Omega \rightarrow \mathbb{C}_\infty$.  Consequently, we see that $\pi^{\ast}_{g}(\tilde{\mathfrak{j}}_{i,w}) = \mathfrak{j}_{i,w}$. Since $\pi^{\ast}_{g}$ is also a ring homomorphism, together with Lemma \ref{2.13}, we see that any algebraic relation satisfied by the set $\{\tilde{\mathfrak{j}}_{i,w} \}$ over the field $K(\tilde{\mathfrak{J}})$ gives rise to an algebraic relation of $\{\mathfrak{j}_{i,w} \}$ over $K(\mathfrak{J})$ and hence this implies the desired statement. Finally the last assertion of (ii) follows from the fact that the surjective map $X(\Gamma_Y(I))\to X(\Gamma_Y)$ is finite.
	\end{proof}
	
\end{proposition}

Our next proposition can be seen as a generalization of a result of Gekeler \cite[Satz ~(4.3)]{Gek83} for the case $A=\mathbb{F}_q[\theta]$.
	
\begin{proposition}\label{P:jinvariant} Let $J$ be a Drinfeld modular function for $\Gamma_Y$ which is well-defined at a CM point $z_0$ and whose $t_{\Gamma_Y}$-expansion coefficients lie in $\overline{K}$. Then  $J(z_0)\in \overline{K}$. 
\end{proposition}
\begin{proof} Our proof follows from the classical methods as well as Hayes's theory developed for rank one Drinfeld $R$-modules whenever $R$ is an order for an admissible coefficient ring. 

By Lemma \ref{L:funcfield} and Proposition \ref{P:jvalues}(ii), it suffices to show the proposition for $\mathfrak{J}$. Assume that $\phi$ is a CM Drinfeld $A$-module of rank two defined over $\mathbb{C}_{\infty}$ whose corresponding $A$-lattice is denoted by $Y_{z_0}$. Then, by \cite[Cor.~4.7.15]{Gos96}, the endomorphism ring $R:=\End(\phi)\cong \End(Y_{z_0})$ defines an order for an admissible coefficient ring whose field of fractions is a quadratic extension of $K$, in which the infinity place does not split. Hence one can consider $\phi$ as a Drinfeld $R$-module of rank one, via the embedding $R\hookrightarrow \mathbb{C}_{\infty}[\tau]$. Let $\psi\in \Aut(\mathbb{C}_{\infty}/\Frac(R))$. We set $\phi^{\psi}$ to be the Drinfeld $R$-module constructed by applying $\psi$ to the coefficients of $\phi_a$ for each $a\in R$.  Then by \cite[Prop.~8.1]{Hay79}, $\phi^{\psi}$ is isomorphic to a Drinfeld $R$-module $\tilde{\phi}$ of rank one. Consequently by \cite[Cor.~5.12]{Hay79}, the set $\{\phi^{\psi}\ \ | \ \ \psi\in \Aut(\mathbb{C}_{\infty}/\Frac(R))\}$ has finitely many isomorphism classes of Drinfeld $R$-modules of rank one. On the other hand, defining $\mathfrak{J}^\phi$ and $\mathfrak{J}^{\tilde{\phi}}$ similar to Example \ref{Ex:1}(iii) so that $\mathfrak{J}^\phi = \mathfrak{J}(z_0)$, we have $\mathfrak{J}^{\phi}=\mathfrak{J}^{\tilde{\phi}}$ whenever $\phi$ and $\tilde{\phi}$ are isomorphic over $\mathbb{C}_{\infty}$. Thus, we conclude that $\{\mathfrak{J}^{\phi^{\psi}}\ \ | \ \ \psi\in \Aut(\mathbb{C}_{\infty}/\Frac(R)) \}$ is a finite set. Since $\mathfrak{J}^{\phi^{\psi}}=\psi(\mathfrak{J}^{\phi})$ and $\psi$ is arbitrary, the algebraicity of $\mathfrak{J}^{\phi}$ follows.
\end{proof}

In what follows, we describe special values of arithmetic Drinfeld modular forms at CM points in terms of the ratio of periods of Drinfeld $A$-modules up to an algebraic constant. 

\begin{theorem}\label{T:specvalue} Let $f$ be an arithmetic Drinfeld modular form of weight $k$ for a congruence subgroup $\Gamma$ of $\Gamma_Y$. Then for any CM point $z_0\in \Omega$, there exists $w_{z_0}\in \mathbb{C}_{\infty}^{\times}$ such that 
	\[
	f(z_0)=c\Bigg(\frac{w_{z_0}}{\xi(\mathfrak{g}^{-1}\mathfrak{h})}\Bigg)^{k} 
	\]
	for some $c\in \overline{K}$. In particular, if $f(z_0)$ is non-zero, then it is transcendental over $\overline{K}$.
\end{theorem}
\begin{proof} Our proof extends the strategy of the proof of \cite[Thm.~2.1.2]{Cha12}. Let $a\in A$ be of degree $d\geq 1$. Consider
\[
\Psi^{Y_{z_0}}_{a}=a+g_{1,a}(z_0)\tau+\cdots+g_{2d,a}(z_0)\tau^{2d}.
\]
For a fixed non-zero $\alpha\in \mathfrak{h}$, choose an element $\eta\in \mathbb{C}_{\infty}^{\times}$ such that $\eta^{q^{2d}-1}g_{2d,a}(z_0)=\alpha^{q^{2d}-1}$.  Set $\phi:=\eta^{-1}\Psi^{Y_{z_0}}\eta$, which is a Drinfeld $A$-module of rank two. Then 
\begin{multline*}
\phi_a=a+\alpha^{q-1}\sqrt[q^{2d}-1]{\frac{g_{1,a}(z_0)^{q^{2d}-1}}{g_{2d,a}(z_0)^{q-1}}}\tau +\cdots + \alpha^{q^{j}-1}\sqrt[q^{2d}-1]{\frac{g_{j,a}(z_0)^{q^{2d}-1}}{g_{2d,a}(z_0)^{q^j-1}}}\tau^j \\ +\cdots+\alpha^{q^{2d-1}-1}\sqrt[q^{2d}-1]{\frac{g_{2d-1,a}(z_0)^{q^{2d}-1}}{g_{2d,a}(z_0)^{q^{2d-1}-1}}}\tau^{2d-1}+\alpha^{q^{2d}-1}\tau^{2d}.
\end{multline*}
Setting $\mathfrak{c}_{i}$ to be the greatest common divisor of $q^{i}-1$ and $q^{2d}-1$, $\ell_1=\frac{q^{2d}-1}{\mathfrak{c}_i}$ and $\ell_2=\frac{q^{i}-1}{\mathfrak{c}_i}$, observe that each function $\frac{g_{i,a}^{q^{2d}-1}}{g_{2d,a}^{q^i-1}}=\Big( \frac{g_{i,a}^{\ell_1}}{g_{2d,a}^{\ell_2}}\Big)^{\mathfrak{c}_i}$ is a Drinfeld modular function for $\Gamma_Y$. By Proposition \ref{P:jinvariant}, their values at $z_0$ are elements in $\overline{K}$ and hence $\phi$ is a Drinfeld $A$-module defined over $\overline{K}$. We further set $w_{z_0}:=\alpha/\eta$, which is a period for $\phi$. Observe that
\begin{equation}\label{E:wz0}
g_{2d,a}(z_0)=w_{z_0}^{q^{2d}-1}.
\end{equation}
Since $\xi(\mathfrak{g}^{-1}\mathfrak{h})$ is transcendental over $K$, by Proposition \ref{P:prod}, we see that $\xi(\mathfrak{g}^{-1}\mathfrak{h})^{1-q^{2d}}g_{2d,a}$ is a Drinfeld modular form with algebraic $t_{\Gamma_Y}$-expansion coefficients at the cusp $\infty$.

Since, by assumption, $f$ has algebraic $t_{\Gamma}$-expansion coefficients, by Lemma \ref{L:algoft}, there exists an integral ideal $\mathfrak{m}$ such that $f$ has algebraic $t_{\Gamma_Y(\mathfrak{m})}$-expansion coefficients. Similarly, by Lemma \ref{L:algoft} and Proposition \ref{P:prod}, $\xi(\mathfrak{g}^{-1}\mathfrak{h})^{1-q^{2d}}g_{2d,a}$ has algebraic $t_{\Gamma_Y(\mathfrak{m})}$-expansion coefficients. Thus,
\[
f^{q^{2d}-1}/(\xi(\mathfrak{g}^{-1}\mathfrak{h})^{1-q^{2d}}g_{2d,a})^k\in \mathcal{A}_0(\Gamma_Y(\mathfrak{m});\overline{K}).
\]
Observe that by using \eqref{E:funceq}, $t_{\Gamma_Y(\mathfrak{m_1})}$ can be written as a power series in $t_{\Gamma_Y(\mathfrak{m_2})}$ with algebraic coefficients whenever $\mathfrak{m_1}\subseteq \mathfrak{m_2}\subset A$. Hence without loss of generality, we further assume that $|V(\mathfrak{m})|\geq 2$.  Then  by Proposition  \ref{P:jinvariant}, we see that 
\[
f^{q^{2d}-1}(z_0)/(\xi(\mathfrak{g}^{-1}\mathfrak{h})^{1-q^{2d}}g_{2d,a}(z_0))^k\in \overline{K}.\] Thus, by \eqref{E:wz0}, we obtain the first part. The second assertion is a consequence of the first assertion and \cite[Thm.~5.4]{Yu86}.
\end{proof}

\begin{remark} We note that, when $A=\mathbb{F}_q[\theta]$, Theorem \ref{T:specvalue} is proved by Chang in \cite[Thm.~2.2.1]{Cha12}.  
\end{remark}

Let $a\in A$ be non-constant. Recall that $\Delta_a=g_{2\deg(a),a}$. Let $z_0\in \Omega$ be a CM point and consider the extension $M_{z_0}$ of $\Knr$ containing $z_0$ and the extension $\varphi_{z_0}$ from Lemma \ref{L:ext}. Set $\mathfrak{f}:=\frac{\Delta_a}{\Delta_a|_{q^{2d}-1,0}\alpha_{z_0}}$ where 
\[
\alpha_{z_0}:=\begin{pmatrix}
\Tr(z_0)&-\Nr(z_0)\\
1&0
\end{pmatrix}\in \GL_2(K).
\]
Observe that $\alpha_{z_0}\cdot z_0=z_0$ and $\det(\alpha_{z_0})=\mathfrak{z}(z_0)z_0$ where $\mathfrak{z}:K(z_0)\to K(z_0)$ is the generator of $\Gal(K(z_0)/K)$. By Lemma \ref{L:MaassShimura}(i), we have 
\[
\delta_{q^{2d}-1}(\Delta_a)=\delta_{q^{2d}-1}(\mathfrak{f}\Delta_a|_{q^{2d}-1,0}\alpha_{z_0})=\delta_{q^{2d}-1}(\Delta_a|_{q^{2d}-1,0}\alpha_{z_0})\mathfrak{f}+\Delta_{a}|_{q^{2d}-1,0}\alpha_{z_0}\xi(\mathfrak{g}^{-1}\mathfrak{h})^{-1}\der\mathfrak{f}.
\]
Dividing above by $\Delta_a=\Delta_a|_{q^{2d}-1,0}\alpha_{z_0}\mathfrak{f}$ and using 
\[
\delta_{q^{2d}-1}(\Delta_a)=
\frac{\partial_z(\Delta_a)}{\xi(\mathfrak{g}^{-1}\mathfrak{h})}-\frac{1}{\xi(\mathfrak{g^{-1}\mathfrak{h}})}\frac{\Delta_a}{\Id-\varphi}=\Big(\frac{1}{\xi(\mathfrak{g}^{-1}\mathfrak{h})} \frac{\partial_z(\Delta_a)}{\Delta_a}-\frac{1}{\xi(\mathfrak{g^{-1}\mathfrak{h}})}\frac{1}{\Id-\varphi}\Big)\Delta_a=\Delta_a E_2
\]
yield
\begin{equation}\label{E:eq1}
E_2=\frac{\delta_{q^{2d}-1}(\Delta_a|_{q^{2d}-1}\alpha_{z_0})}{\Delta_a|_{q^{2d}-1,0}\alpha_{z_0}}+\frac{1}{\xi(\mathfrak{g}^{-1}\mathfrak{h})}\frac{\der \mathfrak{f}}{\mathfrak{f}}.
\end{equation}
Moreover, by Lemma \ref{L:MaassShimura}(ii), we have 
\begin{equation}\label{E:eq2}
(E_2|_{2,1}\alpha_{z_0})(\Delta_a|_{q^{2d}-1,0}\alpha_{z_0})=(E_2 \Delta_a)|_{q^{2d}+1,1}\alpha_{z_0}=\delta_{q^{2d}-1}(\Delta_a)|_{q^{2d}+1,1}\alpha_{z_0}=\delta_{q^{2d}-1}(\Delta_a|_{q^{2d}-1,0}\alpha_{z_0}).
\end{equation}
Set $\mathcal{G}:=\frac{-1}{\xi(\mathfrak{g}^{-1}\mathfrak{h})}\frac{\der\mathfrak{f}}{\mathfrak{f}}$. It is easy to see that $\mathcal{G}\in \mathcal{A}_2(\overline{K})$. Moreover, substituting \eqref{E:eq2} to \eqref{E:eq1} yields
\[
\mathcal{G}(z)=E_2|_{2,1}\alpha_{z_0}(z)-E_2(z).
\]
Using the action of $\alpha_{z_0}$ on $z_0$, we now obtain the following.
\begin{lemma}\label{L:specvalue} We have 
\[
\mathcal{G}(z_0)=\frac{\mathfrak{z}(z_0)-z_0}{z_0}E_2(z_0).
\]
\end{lemma}

We end this section with our next theorem.

\begin{theorem}\label{T:specvalue2} Let $z_0\in \Omega$ be a CM point and $F:\Omega^{\varphi_{z_0}}(M_{z_0})\to \mathbb{C}_{\infty}$ be an arithmetic nearly holomorphic Drinfeld modular form of weight $k$ and depth $r$ for a congruence subgroup $\Gamma$ of $\Gamma_Y$. Then 
\[
F(z_0)=c\Bigg(\frac{w_{z_0}}{\xi(\mathfrak{g}^{-1}\mathfrak{h})}\Bigg)^{k} 
\]
for some $c\in \overline{K}$. In particular, if $f$ is a Drinfeld modular form of weight $k$ for $\Gamma$, then 
\[
\delta^r_{k}(f)(z_0)=\tilde{c}\Bigg(\frac{w_{z_0}}{\xi(\mathfrak{g}^{-1}\mathfrak{h})}\Bigg)^{k+2r} 
\]
for some $\tilde{c}\in \overline{K}$. Furthermore, if $F(z_0)$ and $\delta_k^r(f)(z_0)$ are non-zero, then they are transcendental over $\overline{K}$.
\end{theorem}
\begin{proof} By Proposition \ref{P:str}, we can write $F:\Omega^{\varphi_{z_0}}(M_{z_0})\to \mathbb{C}_{\infty}$ as	\begin{equation}\label{E:repF} 
F=\sum_{0\leq j\leq r}g_iE_2^i
\end{equation}
with $g_i\in\mathcal{M}_{k-2i}(\Gamma;\overline{K})$. Choose some ideal $\mathfrak{n}\subseteq A$ such that  $\Gamma_Y(\mathfrak{n})\subseteq \Gamma$. For some $u\in \mathfrak{n}^{-1}Y\setminus Y$, consider $E_u$, the Eisenstein series of weight 1 for $\Gamma_Y(\mathfrak{n})$ given as in Example \ref{Ex:1}(i, iv).  By \eqref{E:repF}, we now have
\[
F(z_0)=\sum_{0\leq j\leq r}g_i(z_0)E_u(z_0)^{2i}\Bigg(\frac{E_2(z_0)}{E_u^2(z_0)}\Bigg)^i.
\] 
By Theorem \ref{T:specvalue}, $g_i(z_0)E_u(z_0)^{2i}$ is equal to $(w_{z_0}/\xi(\mathfrak{g}^{-1}\mathfrak{h}))^k$ up to some algebraic constant. On the other hand, since $\mathcal{G}/E_u^2\in \mathcal{A}_0(\overline{K})$, by Proposition \ref{P:jinvariant}, we have $\mathcal{G}(z_0)/E_u^2(z_0)\in \overline{K}$. Thus, by Lemma \ref{L:specvalue}, we also have $E_2(z_0)/E_u^2(z_0)\in \overline{K}$. Hence, we obtain the desired result. The second assertion is a simple consequence of Lemma \ref{L:MaassShimura}(ii) and the first assertion. 
\end{proof}

\section{Tate-Drinfeld modules and expansions at cusps}
Let $I$ be an ideal of $A$ such that $|V(I)| \geq 2$. From now on, in the present paper, we aim to introduce the necessary tools to describe nearly holomorphic Drinfeld modular forms as global sections of an explicit sheaf defined over a compactification of the Drinfeld moduli space $M_I^2$ (Theorem \ref{T:main1}). We would like to emphasize that we use the theory of Tate-Drinfeld modules, which is available in literature only with the assumption $|V(I)| \geq 2$. We expect that such a theory of Tate-Drinfeld modules may be established in the case $|V(I)|=1$ for moduli spaces over $\Spec(K)$. However, due to lack of a reference, we have refrained from assuming this  condition.

\subsection{Universal Tate-Drinfeld modules} In this subsection, our goal is to construct  a universal Tate-Drinfeld module  with  a level $I$-structure. Here, we follow mainly \cite[Chap.~3-5]{Leh09} and the reader is referred there for more details (see also \cite[\S5.6]{vdH03}). We remark that in contrast with \cite{Leh09}, our Tate-Drinfeld modules are base change over $A$ by $K$ of Tate-Drinfeld modules discussed therein, as the Drinfeld moduli space $M^2_I$ in our setting is defined over $\Spec(K)$. We also refer the reader to \cite[\S4]{Hat22} for a thorough discussion  when $A=\mathbb{F}_q[\theta]$. 

Let us fix an element $c\in I\setminus\{0\}$. Then $(c)I^{-1}$ is an integral ideal of $A$ and we set $\mathfrak{a}:=(c)I^{-1}$. Thus, the map $I^{-1}\cong \mathfrak{a}$, given by multiplication with $c$, is an $A$-module isomorphism which we denote by $\eta$. Define the ideal $\mathfrak{b}:=\eta(A)\subseteq \mathfrak{a}$. Recall that $M^{1}_I=\Spec(\mathcal{H})$, where $\mathcal{H}$ is the ray class field of $K$ with a suitable conductor totally split at $\infty$ (see \cite[Rem.~4.13]{Boc02}, \cite[Thm.~1]{Dri74}), is the moduli space of Drinfeld $A$-modules of rank one over a $K$-scheme with a level $I$-structure. Recall the universal Drinfeld $A$-module $\mathbb{E}^{un,1}_{I}$ over $M^{1}_I$ with a level $I$-structure $\lambda_0$. Recalling the $*$-action of integral ideals on Drinfeld $A$-modules defined in \S2.2, set $(\psi,\lambda'):=\mathfrak{a}*(\phi^{un,1},\lambda_0)$ where $\psi:=\mathfrak{a}*\phi^{un,1}$ and the level structure $\lambda'$ is determined by the isogeny $\phi^{un,1}_{\mathfrak{a}}$ described in \S2.2. Let $X$ be a formal variable and $\mathcal{H}((X))$ be the ring of formal Laurent series. Consider $\mathcal{H}((X))$ as an $A$-module via the $\psi$-action on $\mathcal{H}$. Let $\mathcal{H}((X))^{\text{sep}}$ be a fixed separable closure of $\mathcal{H}((X))$. In this context, an $A$-lattice of rank $s$ is a projective $A$-submodule of  $\mathcal{H}((X))^{\text{sep}}$ of rank $s$ which is $\Gal(\mathcal{H}((X))^{\text{sep}}/\mathcal{H}((X)))$-stable and has finite intersection with any ball of finite radius. Note that, in contrast to \cite[Chap.~5, \S2]{Leh09}, we use the formal variable $X$ instead of $1/X$ used there. 

We  define the rank one $A$-lattice $\Lambda^{\mathfrak{b}}_\psi \subset \mathcal{H}((X))^{\text{sep}}$ as 
\begin{equation}\label{E:lattice}
\Lambda^{\mathfrak{b}}_\psi:=\{\psi_{b}(1/X) \ \ | \ \ b\in \mathfrak{b}\}.
\end{equation}
 Via Tate-Drinfeld uniformization \cite[Prop.~7.2]{Dri74}, we can associate to the data $(\psi,\Lambda^{\mathfrak{b}}_\psi,\lambda')$ a  Drinfeld $A$-module $(\phi,\lambda)$ of rank two over $\mathcal{H}((X$)) with a level $I$-structure $\lambda$ as follows: Let $Z$ be another formal variable, independent from $X$, and
define formally the \textit{exponential function of $\Lambda^{\mathfrak{b}}_\psi$} by
\[
e_{\Lambda^{\mathfrak{b}}_\psi}(Z):=Z\prod_{\alpha \in \Lambda^{\mathfrak{b}}_\psi\setminus \{0\} }\Big(1-\frac{Z}{\alpha}\Big)\in \mathcal{H}[[X]][[Z]].
\] Note that $e_{\Lambda^{\mathfrak{b}}_\psi}$ is $\mathbb{F}_q$-linear in $Z$ and $e_{\Lambda^{\mathfrak{b}}_\psi} \equiv Z (\text{ mod }X )$.

For any $a\in A$, we further consider
\begin{equation}\label{E:Tate}
\phi_a(Z):=e_{\Lambda^{\mathfrak{b}}_\psi}(\psi_{a}(e_{\Lambda^{\mathfrak{b}}_\psi}^{-1}(Z))).
\end{equation}
Then by \cite[Chap.~5, Prop.~2.3]{Leh09}, $\phi$ is a Drinfeld $A$-module of rank two over $\mathcal{H}((X))$. Moreover, the level $I$-structure $\lambda'$ on $\psi$ gives rise to a canonical level $I$-structure $\lambda$ on $\phi$ by 
\[\lambda(x,\underline{y}):=e_{\Lambda^{\mathfrak{b}}_{\psi}}(\lambda'(x)+ \psi_y(1/X)) \in \mathcal{H}((X))\]
where $x \in I^{-1}/A$ and for $y \in I^{-1}$,  $\underline{y}$ is its image in $I^{-1}/A$.

Recall the profinite completion $\widehat{A}$ of $A$, the ring of finite adeles $\mathbb{A}_f$ and the subgroup $\mathcal{K}(I)\leq \GL_{2}(\widehat{A})$ from \S2.8. Set $\overline{\GL^{0}_2}:=\GL_{2}(\widehat{A})\mathbb{A}^{\times}_{f}/(\mathcal{K}(I)K^{\times})$. Then there exists an action of $\overline{\GL^{0}_2}$ on the Drinfeld moduli space $M^{2}_I$ as follows: 
\begin{itemize}
\item[(i)] First we let $S$ be a $K$-scheme and  $(\mathbb{E},\nu)\in M^{2}_I(S)$ be a Drinfeld $A$-module with a level $I$-structure $\nu$ over $S$. Then for any $\alpha \in \GL_{2}(\widehat{A})/(\mathcal{K}(I)\mathbb{F}^{\times}_{q})$, we define
\[
\alpha_\ast(\mathbb{E},\nu):=(\mathbb{E},\nu\circ\alpha^{-1}).
\]
where $\alpha$ acts on $v$ via the identification $I^{-1}\widehat{A}/\widehat{A} \cong I^{-1} /A$. 
\item[(ii)] Secondly, note that $\mathbb{A}^{\times}_{f}/\widehat{A}^{\times}K^{\times} \cong \Cl(A)$, where an isomorphism between the two groups is given by a map sending each class $[\mathbf{a}]\in \mathbb{A}^{\times}_{f}/\widehat{A}^{\times}K^{\times}$ to $[\mathbf{a}\widehat{A} \cap K]$. By \cite[Lem.~5.6.4]{vdH03}, one can choose a representative $\mathbf{a} \in \mathbb{A}^{\times}_f$ so that $\mathbf{a}\widehat{A} \cap K$ is an ideal of $A$ that is relatively prime to $I$. For such an $[\mathbf{a}]$, we define 
\[[\mathbf{a}]_\ast(\mathbb{E},\nu):=[\mathbf{a}\widehat{A}\cap K]\ast(\mathbb{E},\nu).\]
\end{itemize}

Let $\overline{G}\leq \overline{\GL^{0}_2}$ be the subgroup consisting of the image of matrices of the form 
$\begin{pmatrix}
    c_1 & c_2\\
    0 & 1
\end{pmatrix} \in \GL_2(\widehat{A})$ for $c_1,c_2\in \widehat{A}$ inside $\overline{\GL^{0}_2}$ and set $n_I:=[\overline{\GL^{0}_2} : \overline{G}]$. By \cite[Chap.~5, Prop.~3.5]{Leh09}, we know that $n_{I}$ is a finite number which can be described explicitly. For any $\alpha \in \overline{G}$, from the proof of \cite[Chap.~5, Prop.~2.5]{Leh09}, we see that $\alpha_\ast(\phi,\lambda)=(\phi,\lambda\circ \alpha^{-1})$. Consequently, for any two representatives $\sigma,\sigma' \in \overline{\GL^0_2}$ of the same coset in $\overline{\GL^0_2}/ \overline{G}$, we have that the underlying Drinfeld $A$-modules of $\sigma_\ast(\phi,\lambda)$ and $\sigma'_\ast(\phi,\lambda)$ are the same except that their level $I$-structure differs by an element in $\overline{G}$. We fix, $\sigma_1=1,\sigma_2,\dots,\sigma_{n_I}$, a set of left coset representatives of $\overline{\GL^{0}_2} / \overline{G}$, and each $(\sigma_{i})_{\ast}(\phi,\lambda)$ gives rise to a Drinfeld $A$-module $(\phi_i,\lambda_i)$ of rank two defined over $\mathcal{H}((X))$, which we call \textit{a Tate-Drinfeld module} (abbreviated as $\TD$-module). For $1\leq i \leq n_{I}$, let $X_i$ and $Z_i$ be indeterminates over $\mathcal{H}$. To distinguish these $\TD$-modules, we further consider $(\phi_i,\lambda_i)$ to be the $\TD$-module defined over $\mathcal{H}((X_i))$ so that for any $a\in A$, $(\phi_i)_{a}(Z_i)$ is an $\mathbb{F}_q$-linear polynomial in $Z_i$ with coefficients in $\mathcal{H}((X_i))$. By our convention, $X_1=X$ and $Z_1=Z$, so that $(\phi_1,\lambda_1)=(\phi,\lambda)$.

In what follows, we define \textit{a universal $\TD$-module} by the set 
\[\TD_I:=\{(\phi_i,\lambda_i) \ \ | \ \ \  1\leq i \leq n_I\}.
\]
Note that $\TD_I$ may be regarded as a  Drinfeld $A$-module of rank two with a level $I$-structure  over $\Spec(\oplus_{i=1}^{n_{I}} \mathcal{H}((X_i)))$  so that its restriction to $\Spec(\mathcal{H}((X_i)))$ is given by $(\phi_i,\lambda_i)$. 

We finish this subsection with some remarks on $\TD_I$.
\begin{remark}\label{R:uniTD0}
\begin{enumerate}
    \item[(i)] If one chooses a different set of representatives $\{\sigma'_i\}_{1 \leq i \leq n_I}$ of $\overline{\GL^0_2} / \overline{G}$, then, in a similar way, we have an induced map $\sqcup^{n_I}_{i=1} \Spec(\mathcal{H}((X'_i))) \rightarrow M^{2}_I$, for some indeterminates $X'_i$'s. This is determined by the set $\{(\tilde{\phi}_i,\tilde{\lambda}_i) \ \ | \ \ \  1\leq i \leq n_I\}$ where $(\tilde{\phi}_i,\tilde{\lambda}_i)$ corresponds to  $(\sigma'_{i})_\ast(\phi,\lambda)$. From \cite[Chap.~5, Prop.~2.5]{Leh09}, already alluded to above, we get a unique isomorphism $\sqcup^{n_I}_{i=1} \Spec(\mathcal{H}((X'_i))) \cong \sqcup^{n_I}_{i=1} \Spec(\mathcal{H}((X_i)))$, making the obvious triangle with $M^2_I$ commutative. 
    \item[(ii)] We note that one can obtain a Drinfeld $A$-module with \textit{bad reduction} via a pull-back of $\TD_{I}$ to an appropriate space and this fact indeed motivates the terminology  \textit{universal} for such objects. We refer the reader to  \cite[Chap.~5, Prop.~2.6]{Leh09} and \cite[\S5.7]{vdH03} for further details.
\end{enumerate} 
\end{remark}

\subsection{Algebraic cusps of $\Gamma_{Y}(I)$} Let $\mu_i:\Spec(\mathcal{H}((X_i))) \rightarrow M^{2}_I$ be the map so that   the pull back of the universal Drinfeld $A$-module $\mathbb{E}^{un}_I$ via $\mu_i$ is given by $(\phi_i,\lambda_i)$.  We call each such $\mu_i$ \textit{an algebraic cusp} of $M^{2}_I$ and set 
\[
\AlgCusps_{I}:=\{\mu_i\ \ | \ \ 1\leq i \leq n_{I}\}.
\]

\begin{remark}\label{R:uniTD} As described in \cite[Chap.~5~\S3]{Leh09} (see also \cite[\S5.6-9]{vdH03}), one can use  $\TD_I$ to construct a compactification  $\overline{M^{2}_I}$ of $M^{2}_I$ such that $\overline{M^{2}_I} \backslash M^{2}_I$  consists of $n_I$ points.  the formal completion $\overline{M^2_I}$ at the closed subscheme $\overline{M^{2}_I} \backslash M^2_I$ is isomorphic to $\sqcup^{n_I}_{i=1} \Spf(\mathcal{H}[[X_i]])$ and the induced map
    \[\mathcal{O}(M^2_I) \hookrightarrow K(\overline{M^2_I}) \hookrightarrow \Frac(\widehat{\mathcal{O}_{\overline{M^2_I}}}) \rightarrow 
\oplus^{n_I}_{i=1} \Frac(\mathcal{H}[[X_i]])=\oplus^{n_I}_{i=1} \mathcal{H}((X_i))    \] 
    where $\widehat{\mathcal{O}_{\overline{M^2_I}}}$ is the formal completion of the structure sheaf $\mathcal{O}_{\overline{M^2_I}}$ at $\overline{M^{2}_I} \setminus M^2_I$, is the map determined by the universal $\TD$-module $\TD_I$. Consequently, there exists an induced bijection between $\overline{M^2_I} \backslash M^2_I$ and $\AlgCusps_I$. 
\end{remark}

 Recall the projective $A$-module $Y$ given as in \eqref{E:defofY}. Consider the base change $M^2_{I,\mathbb{C}_{\infty}}=M_I^2\times_{K} \mathbb{C}_{\infty}$ and let $(M^{2}_{I,\mathbb{C}_\infty})^{an}$ be the analytification of $M^2_{I,\mathbb{C}_{\infty}}$. Let $M_Y \subset (M^{2}_{I,\mathbb{C}_\infty})^{an}$ be the connected component of $M^2_{I,\mathbb{C}_{\infty}}$ so that $M_{Y}(\mathbb{C}_\infty) = \Gamma_{Y}(I) \setminus \Omega$. Since, by \cite[Part III]{BBP21}, for any integral ideal $\mathfrak{n}$, $\Gamma_{\mathfrak{n}^{-1}Y}=\Gamma_{Y}$,  without loss of generality, we assume that $\mathfrak{g}$ and $\mathfrak{h}$ are integral ideals. Note that $\mathcal{H}$ is the field of constants of the moduli space $M^2_I$ so that there is an embedding $\mathcal{H} \rightarrow \mathbb{C}_\infty$ such that $M_Y =M^{2}_I \times_{\mathcal{H}} \mathbb{C}_\infty$ and $\overline{M_Y}:=\overline{M^2_I} \times_\mathcal{H} \mathbb{C}_\infty$(see \cite[\S8]{BBP21}). In what follows, whenever we mention a base change $\times_\mathcal{H} \mathbb{C}_\infty$, we will mean the base change with respect to the aforementioned embedding of $\mathcal{H}$ into $\mathbb{C}_\infty$. Consequently, after the base change $\times_\mathcal{H} \mathbb{C}_\infty$, the set of algebraic cusps will allow us to construct the following set:
\[
\AlgCusps^{Y}_I:=\{\mu^{Y}_i:=\mu_i \times_\mathcal{H} \mathbb{C}_\infty:\Spec(\mathbb{C}_\infty((X_i))) \rightarrow M_Y\ \ | \ 1\leq i \leq n_{I} \}.
\]
By Remark \ref{R:uniTD}, there also exists a bijection between $\overline{M_Y} \backslash M_Y$ and $\AlgCusps^Y_I$.

For any $1\leq i \leq n_{I}$, we denote by $(\phi_{i,Y},\lambda_{i,Y})$ the pullback of the Drinfeld $A$-module $\mathbb{E}^{un}_I \times_\mathcal{H} \mathbb{C}_\infty$ with a level $I$-structure  via $\mu^{Y}_i$, which is indeed the base change $\times_\mathcal{H} \mathbb{C}_\infty$ of $(\phi_i,\lambda_i)$. Observe that $(\phi_{i,Y},\lambda_{i,Y})$ is a Drinfeld $A$-module of rank two defined over $\mathbb{C}_{\infty}((X_i))$ with a level $I$-structure $\lambda_{i,Y}$. Moreover, letting $\psi_{Y}$ be the base change $\times_\mathcal{H} \mathbb{C}_\infty$ of $\psi$, we see that $(\phi_{1,Y},\lambda_{1,Y})$ also admits a triple $(\psi_{Y},\Lambda^{\mathfrak{b}}_{\psi_{Y}},\lambda'_{Y})$, which is indeed the base change of $(\psi,\Lambda^{\mathfrak{b}}_{\psi},\lambda')$ by $\times_\mathcal{H} \mathbb{C}_\infty$, so that $(\phi_{1,Y},\lambda_{1,Y})$ corresponds to the data $(\psi_{Y},\Lambda^{\mathfrak{b}}_{\psi_{Y}},\lambda'_{Y})$  as in \S6.1.  Furthermore, there exists a fractional ideal $\mathfrak{c}$ of $A$ and $\beta(\mathfrak{c}) \in \mathbb{C}^{\times}_\infty$ such that the $A$-lattice corresponding to $\psi_{Y}$ is $\beta(\mathfrak{c})\mathfrak{c}$. 

To ease the notation, in what follows, we set $\mathfrak{c}:=\mathfrak{c}_1$.  Let $\mathfrak{m}$ be a fractional ideal of $A$. Recall from \eqref{E:unifatideal} the definition of $t_{\mathfrak{m}}$ and for any $z\in\Omega$, let
\begin{equation}\label{unif_2}
t^{\mathfrak{c}}_{\mathfrak{m}}(z):= \exp_{\ \beta(\mathfrak{c})\mathfrak{m}}(\beta(\mathfrak{c}) z)^{-1}=\frac{\xi(\mathfrak{m})}{\beta(\mathfrak{c})}t_\mathfrak{m}(z).
\end{equation}
It is clear that the ratio $\frac{\xi(\mathfrak{m})}{\beta(\mathfrak{c})}$ lies in $\overline{K}$. Therefore, $t^{\mathfrak{c}}_\mathfrak{m}$ and $t_\mathfrak{m}$ differ up to scaling by an element in $\overline{K}$. 

In our next lemma, we obtain an analytic description for the Drinfeld $A$-module $(\phi_{i,Y},\lambda_{i,Y})$. Before stating it, recall from \S2.1 that for an $A$-lattice $\Lambda$ in $\mathbb{C}_{\infty}$, $\phi^{\Lambda}$ is the Drinfeld $A$-module corresponding to $\Lambda$.
\begin{lemma}\label{L:subs}
    Let $z \in \Omega$ be such that $\inorm{z}_{\text{im}} \gg 0$. For any $a\in A$, let 
    \[
    (\phi_{i,Y})_{a}(Z_i):=aZ_i+a_{i,1}(X_i)Z_i^q+\dots+a_{i,2\deg(a)}(X_i)Z_i^{q^{2\deg(a)}}.
    \]
Then there exist fractional ideals $\mathfrak{b}_i$ and $\mathfrak{c}_i$ of $A$ such that via the substitution $X_i \mapsto t^{\mathfrak{c}}_{\mathfrak{c}_i}(z)$, $\phi_{i,Y}$ is isomorphic to the Drinfeld $A$-module $\phi^{\mathfrak{b}_i z + \mathfrak{c}_i}$ over $\mathbb{C}_{\infty}$.
\end{lemma}  
\begin{proof} First we analyze the case $i=1$. Let $\mathfrak{b}_1:=\mathfrak{b}$ and define 
\[
\exp_{\beta({\mathfrak{c}})\mathfrak{c}}(\beta(\mathfrak{c})\mathfrak{b}z):=\{\exp_{\beta({\mathfrak{c}})\mathfrak{c}}(b\beta(\mathfrak{c})z)\ \ | \ \ b\in \mathfrak{b} \}.
\]
Since we have the inclusion $\beta(\mathfrak{c})\mathfrak{c}\subset \beta(\mathfrak{c})\mathfrak{b}z + \beta(\mathfrak{c})\mathfrak{c}$ of $A$-lattices, by \cite[Prop.~2.3(a)]{BBP21}, we see that $\exp_{\beta({\mathfrak{c}})\mathfrak{c}}(\beta(\mathfrak{c})\mathfrak{b}z)$ is a discrete set in $\mathbb{C}_{\infty}$, and moreover, over $\mathbb{C}_{\infty}[[Z]]$, we obtain
\begin{equation}\label{E:lattices}
\exp_{\beta(\mathfrak{c})\mathfrak{b}z+\beta(\mathfrak{c})\mathfrak{c}}(Z) = \exp_{\exp_{\beta({\mathfrak{c}})\mathfrak{c}}(\beta(\mathfrak{c})\mathfrak{b}z)}( \exp_{\beta({\mathfrak{c}})\mathfrak{c}}(Z)).
\end{equation}
On the other hand, using \eqref{E:funceq},  we see that 
\[
\Lambda^{\mathfrak{b}}_{\psi_Y}|_{X =t^{\mathfrak{c}}_\mathfrak{c}(z)}= \exp_{\beta({\mathfrak{c}})\mathfrak{c}}(\beta(\mathfrak{c})\mathfrak{b}z).
\]
Hence, upon the substitution $X \mapsto t^{\mathfrak{c}}_\mathfrak{c}(z)$, by \eqref{E:Tate}, we see that 
\begin{equation}\label{E:latticeeq}
(\phi_{1,Y})_a(Z)=\exp_{\exp_{\beta({\mathfrak{c}})\mathfrak{c}}(\beta(\mathfrak{c})\mathfrak{b}z)}(\psi_a(\exp^{-1}_{\exp_{\beta({\mathfrak{c}})\mathfrak{c}}(\beta(\mathfrak{c})\mathfrak{b}z)}(Z))).
\end{equation}
Let $\tilde{\phi}_a$ be an $\mathbb{F}_q$-linear power series in $\mathbb{C}_{\infty}[[Z]]$ satisfying
\begin{equation}\label{E:lattice2}
\tilde{\phi}_a(\exp_{\beta(\mathfrak{c})\mathfrak{b}z+\beta(\mathfrak{c})\mathfrak{c}}(Z))=(\exp_{\exp_{\beta({\mathfrak{c}})\mathfrak{c}}(\beta(\mathfrak{c})\mathfrak{b}z)}\psi_a\exp_{\exp_{\beta({\mathfrak{c}})\mathfrak{c}}(\beta(\mathfrak{c})\mathfrak{b}z)}^{-1})(\exp_{\beta(\mathfrak{c})\mathfrak{b}z+\beta(\mathfrak{c})\mathfrak{c}}(Z)).
\end{equation}
Then by \eqref{E:funceq} and \eqref{E:lattices}, \eqref{E:lattice2} becomes
\begin{equation}\label{E:lattice3}
\tilde{\phi}_a(\exp_{\beta(\mathfrak{c})\mathfrak{b}z+\beta(\mathfrak{c})\mathfrak{c}}(Z))=\exp_{\exp_{\beta({\mathfrak{c}})\mathfrak{c}}(\beta(\mathfrak{c})\mathfrak{b}z)}(\exp_{\beta({\mathfrak{c}})\mathfrak{c}}(aZ))=\exp_{\beta(\mathfrak{c})\mathfrak{b}z+\beta(\mathfrak{c})\mathfrak{c}}(aZ).
\end{equation}
Since such $\tilde{\phi}_a$ satisfying \eqref{E:lattice3} for any $a\in A$ is unique, by \eqref{E:funceq}, we obtain 
\begin{equation}\label{E:lattice4}
    \tilde{\phi}_a(Z)=\phi_a^{\beta(\mathfrak{c})\mathfrak{b}z + \beta(\mathfrak{c})\mathfrak{c}}(Z)=\exp_{\exp_{\beta({\mathfrak{c}})\mathfrak{c}}(\beta(\mathfrak{c})\mathfrak{b}z)}(\psi_a(\exp_{\exp_{\beta({\mathfrak{c}})\mathfrak{c}}(\beta(\mathfrak{c})\mathfrak{b}z)}^{-1}(Z))).
    \end{equation}
Hence, by \eqref{E:latticeeq} and \eqref{E:lattice4}, we finish the proof of the lemma for $i=1$. We now prove the lemma for the case $2\leq i \leq n$. Observe that any $g\in \overline{\GL^{0}_2}\setminus \overline{G}$ acts on $(\phi_{1,Y},\lambda_{1,Y})$ either as in the case (i) described in \S6.1 and hence $\phi_{1,Y}=\phi_{i,Y}$, or $g$ acts on $(\phi_{1,Y},\lambda_{1,Y})$ as in the case (ii) and thus we obtain $J_i \ast \phi_{1,Y}=\phi_{i,Y}$ for some integral ideal $J_i$ of $A$ relatively prime to $I$. Then, by \cite[Cor.~4.9.5(2)]{Gos96}, setting  $\mathfrak{b}_i:=J_i^{-1}\mathfrak{b}$ and $\mathfrak{c}_i:=J_i^{-1}\mathfrak{c}$, we have
\begin{equation}\label{E:isom}
J_i \ast \phi^{\beta(\mathfrak{c})\mathfrak{b}z + \beta(\mathfrak{c}) \mathfrak{c}} = \phi^{\alpha_i (\beta(\mathfrak{c})\mathfrak{b}_iz + \beta(\mathfrak{c}) \mathfrak{c}_i)}
\end{equation}
for some constant $\alpha_i\in \mathcal{H}^{\times}$ described in terms of a certain coefficient of the isogeny $\phi^{\beta(\mathfrak{c})\mathfrak{b}z + \beta(\mathfrak{c}) \mathfrak{c}}_{J_i}$. Consequently, by \eqref{E:isog}, $\phi^{\alpha_i (\beta(\mathfrak{c})\mathfrak{b}_iz + \beta(\mathfrak{c}) \mathfrak{c}_i)}$ is the unique Drinfeld $A$-module satisfying
\begin{equation}\label{E:isogeny}
\phi_a^{\alpha_i (\beta(\mathfrak{c})\mathfrak{b}_iz + \beta(\mathfrak{c}) \mathfrak{c}_i)} \phi^{\beta(\mathfrak{c})\mathfrak{b}z + \beta(\mathfrak{c}) \mathfrak{c}}_{J_i} = \phi^{\beta(\mathfrak{c})\mathfrak{b}z + \beta(\mathfrak{c}) \mathfrak{c}}_{J_i} \phi^{\beta(\mathfrak{c})\mathfrak{b}z + \beta(\mathfrak{c}) \mathfrak{c}}_a, \ \ a\in A.
\end{equation} 
Since the Drinfeld $A$-module $\phi_{1,Y}$ over $\mathbb{C}_{\infty}((X))$ is, via the substitution $X=t_{\mathfrak{c}}^{\mathfrak{c}}$, equal to $ \phi^{\beta(\mathfrak{c})\mathfrak{b}z + \beta(\mathfrak{c}) \mathfrak{c}}$, regarding $\phi^{\alpha_i(\beta(\mathfrak{c})\mathfrak{b}_iz + \beta(\mathfrak{c}) \mathfrak{c}_i)}$ as a Drinfeld $A$-module $\widetilde{\phi_{i,Y}}$ over $\mathbb{C}_\infty((X_i))$ with $X_i=t^{\mathfrak{c}}_{\mathfrak{c}_i}(z)$, we see, by the  uniqueness of $J_i\ast\phi_{1,Y}$, that $\phi_{i,Y}=\widetilde{\phi_{i,Y}}$. Thus, upon the substitution $X_i \mapsto t^{\mathfrak{c}}_{\mathfrak{c}_i}(z)$, $\phi_{i,Y}(a)=\phi_a^{\alpha_i (\beta(\mathfrak{c})\mathfrak{b}_iz + \beta(\mathfrak{c}) \mathfrak{c}_i)}$ for all $a\in A$ and hence is isomorphic to $\phi^{\mathfrak{b}_iz +  \mathfrak{c}_i}$, via $\alpha_i\beta(\mathfrak{c})\in \mathbb{C}_{\infty}^{\times}$, as desired.    
\end{proof}

\subsection{Correspondence between algebraic cusps and  cusps of $\Gamma_{Y}(I)$} In what follows, we fix an algebraic cusp $\mu^{Y}_i \in \AlgCusps^Y_I$ and describe how to associate an element in $\Cusps^Y_I$ corresponding to it. 

Recall, from \S2.6, $\theta\in \mathbb{C}_{\infty}^{\times}$ so that $\inorm{\theta}=q$ as well as the affinoid algebras $\mathbb{C}_\infty\big\langle \frac{X_i}{\theta^{\mathfrak{r}}}\big\rangle$ and $
\mathbb{C}_\infty \big\langle \frac{X_i}{\theta^{\mathfrak{r}}},\frac{\theta^{\mathfrak{s}}}{X_i}\big\rangle$ where $\mathfrak{r},\mathfrak{s}\in \mathbb{Z}_{\geq 0}$. As an immediate consequence of Lemma \ref{L:subs}, there exists $ \mathfrak{r}$ such that the Drinfeld $A$-module $\phi_{i,Y}$ has coefficients in $\mathbb{C}_\infty\Big\langle \frac{X_i}{\theta^{\mathfrak{r}}}\Big\rangle$. 
We further consider the ring $\mathcal{R}$ of formal Laurent series  given by
\[
\mathcal{R}:=\Bigg\{\sum_{k\geq i_0}a_kX_{i}^k \in \mathbb{C}_{\infty}((X_i)) \ \ | \ \ i_0\in\mathbb{Z}_{\leq 0}\ , \ \  \lim_{k\to \infty}a_kq^{\mathfrak{r}k}=0\Bigg\} \subset \mathbb{C}_{\infty}((X_i)).
\]

In particular, $\mu^Y_i$ factors through the map $\tilde{p}_i:\Spec(\mathcal{R}) \rightarrow M_Y$. Now, for each $\mathfrak{s}$, we have a composition of maps of ringed spaces 
\[
\Sp\Big(\mathbb{C}_\infty\Big\langle \frac{t^{\mathfrak{c}}_{\mathfrak{c}_i}}{\theta^{\mathfrak{r}}},\frac{\theta^\mathfrak{s}}{t^{\mathfrak{c}}_{\mathfrak{c}_i}}\Big\rangle\Big) \rightarrow \Spec\Big(\mathbb{C}_\infty\Big\langle \frac{t^{\mathfrak{c}}_{\mathfrak{c}_i}}{\theta^{\mathfrak{r}}},\frac{\theta^\mathfrak{s}}{t^{\mathfrak{c}}_{\mathfrak{c}_i}}\Big\rangle\Big)  \rightarrow \Spec(\mathcal{R}) \xrightarrow{\tilde{p}_i} M_Y
\]
where the second map is determined by sending $X_i \mapsto t^{\mathfrak{c}}_{\mathfrak{c}_i}$. By the universal property of rigid analytification, it lifts to a unique morphism  
\[
p^{(\mathfrak{s})}_i:\Sp\Big(\mathbb{C}_\infty\Big\langle \frac{t^{\mathfrak{c}}_{\mathfrak{c}_i}}{\theta^{\mathfrak{r}}},\frac{\theta^\mathfrak{s}}{t^{\mathfrak{c}}_{\mathfrak{c}_i}}\Big\rangle\Big)\to M^{an}_Y=\Gamma_Y(I) \backslash \Omega
\]
of rigid analytic spaces. Thus,  we have the following commutative diagram 
\begin{equation}\label{E:1}
\begin{tikzcd}
    \Sp\Big(\mathbb{C}_\infty\Big\langle \frac{t^{\mathfrak{c}}_{\mathfrak{c}_i}}{\theta^{\mathfrak{r}}},\frac{\theta^\mathfrak{s}}{t^{\mathfrak{c}}_{\mathfrak{c}_i}}\Big\rangle\Big) \arrow[r,"p^{(\mathfrak{s})}_{i}"] \arrow[d] 
    & M^{an}_Y=\Gamma_Y(I) \backslash \Omega \arrow[d]\\
    \Spec(\mathcal{R}) \arrow[r,"\tilde{p}_i"] 
    & M_Y.
\end{tikzcd}
\end{equation}
Recall the punctured disc $\mathbb{D}^{\ast}_{\mathfrak{r}}(t^{\mathfrak{c}}_{\mathfrak{c}_i})\subset \Omega$ from \S2.6. Note that, as $\mathfrak{s} \to -\infty$,  $p^{(\mathfrak{s})}_{i}$ gives rise to a map 
\[
p_i:\mathbb{D}^{\ast}_{\mathfrak{r}}(t^{\mathfrak{c}}_{\mathfrak{c}_i}) \rightarrow \Gamma_Y(I) \backslash \Omega.
\]
By the construction of $p_i$, the image of the composition $\mathbb{D}^{\ast}_{\mathfrak{r}}(t^{\mathfrak{c}}_{\mathfrak{c}_i}) \xrightarrow{p_i} M^{an}_Y \rightarrow \overline{M}^{an}_Y$ contains a punctured disc around the point in $ \overline{M}^{an}_Y \backslash M^{an}_Y=\overline{M}_Y \backslash M_Y$ corresponding to $\mu^{Y}_i$. Furthermore, by \cite[Thm.~4.16]{Boc02}, such a point indeed corresponds to a unique $b_i \in \Cusps^Y_{I}$.
We summarize the above discussion in our next proposition.
\begin{proposition}  \label{u_expn_0} There exists a one-to-one correspondence between $\AlgCusps^{Y}_I$ and $\Cusps^Y_{I}$ sending each $\mu^Y_i\in \AlgCusps^{Y}_I $ to $b_i\in\Cusps^Y_{I} $ as above.
\end{proposition}
\subsection{$t$-expansion} Recall \textit{the Hodge bundle} $\omega_{un}$ from \S1. We denote by $\omega_Y$ the restriction of $\omega_{un}$, after base change with $\mathbb{C}_\infty$, to $M_Y$. In what follows, for each $1\leq i \leq n_{I}$, we let $dZ_i$ be the differential on the sheaf $(\mu^{Y}_i)^\ast(\omega_Y)$ of invariant differentials so that $(\mu^{Y}_i)^\ast(\omega_Y)=\mathbb{C}_\infty((X_i))dZ_i$. 
\begin{definition}\label{6.5}
       Let $k$ be a positive integer and $f \in H^{0}(M_Y,(\omega_Y)^{\otimes k})$.  Denote by $P_f(X_i) \in \mathbb{C}_\infty((X_i))$ the unique Laurent series such that 
        \[
        (\mu^{Y}_i)^\ast(f)=P_f(X_i)(dZ_i)^{\otimes k}.
        \]
         We call $P_f(X_i)$ \textit{the $t$-expansion of $f$ at the cusp $\mu^{Y}_i$}.     
    \end{definition}
Let $\delta_i \cdot b_i=\infty$ for some $\delta_i \in \GL_2(K)$. Recall the notation from \eqref{E:unifdef} and set \[
t_{b_i}:=t_{\delta_i \Gamma_{Y}(I) \delta_i^{-1}}.
\]
Via rigid analytification we have the inclusion $H^0 (M_Y,(\omega_{Y})^{\otimes k}) \hookrightarrow H^0 (\Gamma_Y(I)\backslash \Omega,((\omega_Y)^{\otimes k})^{an})$. Moreover, due to the seminal work of Goss \cite[\S1]{Gos80} (see also \cite[Lem.~10.6]{BBP21}), we know that the set $H^{0}(\Gamma_{Y}(I) \setminus \Omega, (\omega_{Y}^{\otimes k})^{an})$ is the $\mathbb{C}_{\infty}$-vector space of weak Drinfeld modular forms of weight $k$ for $\Gamma_{Y}(I)$. Hence each $f\in H^0 (M_Y,(\omega_{Y})^{\otimes k}) $ admits a unique $t_{b_i}$-expansion.

Assume that $b_i\in \Cusps^Y_{I}$ corresponds to $\mu_i^Y$ as in Proposition \ref{u_expn_0}. In this subsection, we aim to obtain a certain relation between the $t$-expansion of $f$ at $\mu^Y_i$ and the corresponding $t_{b_i}$-expansion.

By Corollary \ref{L:nbd_at_infty}, for some integer $e$, there exists an embedding 
\[
g_i:\mathbb{D}^{\ast}_e(t_{b_i}) \to \Gamma_Y(I) \backslash \Omega
\]
of a punctured disc around the cusp $b_i$. Moreover, as in the construction of $p_i^{(\mathfrak{s})}$, one can also form a unique morphism  
\[
g^{(h)}_i:\Sp\Big(\mathbb{C}_\infty\Big\langle \frac{t_{b_i}}{\theta^{e}},\frac{\theta^h}{t_{b_i}}\Big\rangle\Big)\to M^{an}_Y=\Gamma_Y(I) \backslash \Omega
\]
of rigid analytic spaces.

Since the image of both $p_i$ and $g_i$ contain punctured discs around $b_i$, one can find integers $\ell$ and $h$ such that the image of $\Sp\Big(\mathbb{C}_\infty\Big\langle \frac{t^{\mathfrak{c}}_{\mathfrak{c}_i}}{\theta^{\mathfrak{r}}},\frac{\theta^\mathfrak{\ell}}{t^{\mathfrak{c}}_{\mathfrak{c}_i}}\Big\rangle\Big)\subset \mathbb{D}^{\ast}_{\mathfrak{r}}(t^{\mathfrak{c}}_{\mathfrak{c}_i})$ under $p_i$ and the image of $\Sp\Big(\mathbb{C}_{\infty}\Big\langle \frac{t_{b_i}}{\theta^e},\frac{\theta^{h}}{t_{b_i}}\Big\rangle\Big)\subset \mathbb{D}^{\ast}_e(t_{b_i})$ under $g_i$ intersect in a non-empty admissible open subset. Throughout this subsection, we fix such an $\ell$ and $h$. On the other hand, by Proposition \ref{P:expression}, we see that there exists a proper integral ideal $\mathfrak{n}_i$ of $A$ such that  both $t^{\mathfrak{c}}_{\mathfrak{c}_i}$ and $t_{b_i}$ may be written as a power series of $t_i:=t_{\Gamma_{Y}(\mathfrak{n}_i)}$. Furthermore one can choose non-negative integers $m$ and $k$ so that 
\[\mathbb{C}_{\infty}\Big\langle \frac{t_{b_i}}{\theta^e},\frac{\theta^{h}}{t_{b_i}}\Big\rangle \subset \mathbb{C}_\infty\Big\langle \frac{t_i}{\theta^{m}},\frac{\theta^k}{t_i}\Big\rangle \text{ and } \mathbb{C}_\infty\Big\langle \frac{t_{\mathfrak{c}_i}^{\mathfrak{c}}}{\theta^{\mathfrak{r}}},\frac{\theta^{\ell}}{t_{\mathfrak{c}_i}^{\mathfrak{c}}}\Big\rangle \subset \mathbb{C}_\infty\Big\langle \frac{t_i}{\theta^{m}},\frac{\theta^k}{t_i}\Big\rangle.
\]
We thus obtain the following commutative diagram:
\begin{equation}\label{E:2}
\begin{split}
	\begin{tikzpicture}
    		\node (P0) at (90+45:2.5cm) {$\Sp\Big(\mathbb{C}_{\infty}\Big\langle \frac{t_{b_i}}{\theta^e},\frac{\theta^{h}}{t_{b_i}}\Big\rangle\Big)$};
    		\node (P1) at (90+2*45:2.5cm) {$\Sp\Big(\mathbb{C}_\infty\Big\langle \frac{t_i}{\theta^{m}},\frac{\theta^k}{t_i}\Big\rangle\Big)$};
            \node (P2)  at (90 + 3*45:2.5cm) {$\Sp\Big(\mathbb{C}_\infty\Big\langle \frac{t^{\mathfrak{c}}_{\mathfrak{c}_i}}{\theta^{\mathfrak{r}}},\frac{\theta^{\ell}}{t^{\mathfrak{c}}_{\mathfrak{c}_i}}\Big\rangle\Big)$};
    		\node (P3) at (90 + 6*45:2.5cm) {$\Gamma_{Y}(I)\backslash \Omega$.};
    		\draw
    		(P0) edge[->,>=angle 90] node[above] {$g^{(h)}_i$} (P3)
    		(P0) edge[<-,>=angle 90] node[right] {} (P1)
    		(P2) edge[->,=angle 90] node[above] {$p_i^{(\ell)}$} (P3)
            (P1) edge[->,=angle 90] node[midway,above] {} (P2);
    	\end{tikzpicture}
    \end{split}    
\end{equation}

  We are now ready to prove our next proposition.

    \begin{proposition}  \label{u_expn_1}
    Let $f \in H^{0}(M_Y,\omega^{\otimes k}_Y)$. Then the $t$-expansion of $f$ at the cusp $\mu^{Y}_i$ has no principal part if and only if $f$ has a $t_{b_i}$-expansion at the cusp $b_i$ with no principal part, where, in the latter statement, we regard $f$ as an element of $H^{0}(\Gamma_Y(I)\backslash \Omega,(\omega^{\otimes k}_Y)^{an})$.     	  
    \end{proposition}  
\begin{proof} We prove one direction and the other direction is similar. Assume that the $t$-expansion $P_f(X_i)$ of $f$ at the cusp $\mu^{Y}_i$ has no principal part as a Laurent series in $X_i$. Since $P_f(X_i)\in \mathcal{R}$, by the definition of $P_f(X_i)$, we have
   \[
   (\tilde{p}_i)^\ast f = P_f(X_i)(dZ_i)^{\otimes k} \in \mathbb{C}_\infty((X_i))(dZ_i)^{\otimes k}.
   \]
Let $f^{an}$ be the pullback of $f$ to an element in $H^{0}(\Gamma_Y(I) \backslash \Omega,(\omega_Y^{\otimes k})^{an})$. Then there exists a unique Laurent series $ Q_f(t^\mathfrak{c}_{\mathfrak{c}_i})$ such that
\[
   (p^{(\ell)}_i)^\ast f^{an} = Q_f(t^\mathfrak{c}_{\mathfrak{c}_i})(dZ_i)^{\otimes k} \in \mathbb{C}_\infty\Big\langle \frac{t^\mathfrak{c}_{\mathfrak{c}_i}}{\theta^{\mathfrak{r}}},\frac{\theta^\ell}{t^{\mathfrak{c}}_{\mathfrak{c}_i}}\Big\rangle (dZ_i)^{\otimes k}.
   \]
   Thus, by the commutativity of \eqref{E:1}, we see that upon the substitution $X_i \mapsto t^\mathfrak{c}_{\mathfrak{c}_i}$, $P_f(X_i)$ and $Q_f(t^\mathfrak{c}_{\mathfrak{c}_i})$ coincide. On the other hand, we have  
    \[
    (g^{(h)}_i)^\ast(f^{an}) = \widetilde{Q}_{f^{an}}(t_{b_i})(dZ_i)^{\otimes k} \in   \mathbb{C}_\infty \Big\langle \frac{t_{b_i}}{\theta^e},\frac{\theta^{h}}{t_{b_i}}\Big\rangle (dZ_i)^{\otimes k}
    \]  
    for some unique Laurent series $\widetilde{Q}_{f^{an}}(t_{b_i})$, which is the $t_{b_i}$-expansion of $f^{an}$. Thus, by the commutativity of \eqref{E:2}, we see that $\widetilde{Q}_{f^{an}}(t_{b_i})$ has no principal part as a Laurent series in $t_{b_i}$ if and only if 
    $Q_f(t^\mathfrak{c}_{\mathfrak{c}_i})$ has no principal part as a Laurent series in $t_{b_i}$. Since $Q_f(t^\mathfrak{c}_{\mathfrak{c}_i})$ is identified with $P_f(X_i)$ upon the substitution $X_i \mapsto t^\mathfrak{c}_{\mathfrak{c}_i}$ and $P_f(X_i)$ has no principal part, we conclude that $\widetilde{Q}_{f^{an}}(t_{b_i})$ has no principal part as a Laurent series in $t_{b_i}$.
\end{proof}

\begin{remark} As a consequence of the results described in this subsection, letting $\textit{an}:\Gamma_{Y}(I)\backslash \Omega\to M_Y$ be the analytification map, we obtain the following commutative diagram:
 \begin{equation}\label{D:3}	
 \begin{split}
 \begin{tikzpicture}
    		\node (P0) at (90+45:2.7cm) {$\Sp\Big(\mathbb{C}_{\infty}\Big\langle \frac{t_{b_i}}{\theta^e},\frac{\theta^{h}}{t_{b_i}}\Big\rangle\Big)$};
    		\node (P1) at (90+90:2.5cm) {$\Sp\Big(\mathbb{C}_\infty\Big\langle \frac{t_i}{\theta^{m}},\frac{\theta^k}{t_i}\Big\rangle\Big)$};
    		\node  (P2) at (90+3*45:2.5cm) {$\Spec(\mathcal{R})$} ;
    		\node (P3) at (90+5*45:2.5cm) {$M_Y.$};
    		\node (P4) at (90+7*45:2.5cm) {$\Gamma_Y(I) \backslash \Omega$};
    		\draw
    		(P0) edge[<-,>=angle 90] node[left] {} (P1)
    		(P1) edge[->,>=angle 90] node[right] {} (P2)
    		(P2) edge[->,=angle 90] node[above] {$\tilde{p}_i$} (P3)
    		(P4) edge[->,=angle 90] node[right] {\textit{an}} (P3)
    		(P0) edge[->,=angle 90] node[midway,above] {$g^{(h)}_i$} (P4);
    	\end{tikzpicture}
        \end{split}
    \end{equation}
More precisely, Lemma \ref{L:subs} shows that for each TD module $\mu_i:\Spec(\mathcal{H}((X_i))) \rightarrow M^2_I$, its base change $\times_{\mathcal{H}} \mathbb{C}_\infty$ corresponds to the family of Drinfeld $A$-modules $\phi^{\mathfrak{b}_i z +\mathfrak{c}_i}$ for some fractional $A$-ideals $\mathfrak{b}_i$ and $\mathfrak{c}_i$. In Proposition \ref{u_expn_0}, this allows us to associate to each such TD module a unique analytic cusp $b_i \in \Cusps^Y_I$ and hence it gives rise to  \eqref{E:1}. Then using Proposition \ref{P:expression}, we construct $t_i=t_{\Gamma_{Y}(\mathfrak{n}_i)}$ so that  both  $t_{b_i}$ and $t^{\mathfrak{c}}_{\mathfrak{c}_i}$ may be expressed in terms of a  power series in $t_i$ with a finite radius of convergence. Finally, combining \eqref{E:1} and \eqref{E:2} yield \eqref{D:3}.
\end{remark}

\section{Extension of the de Rham sheaf to $\overline{M^2_I}$ and de Rham sheaf on $\Omega$}

We continue to use the notation from \S6 and again assume that $I$ is an ideal of $A$ so that $| V(I)|\geq 2$. Recall that $\mathbb{E}^{un}_{I}=(\mathcal{L}^{un},\phi^{un})$ is the universal Drinfeld $A$-module over $M^2_I$. Consequently, as described in \S2.3, we can attach to $\mathbb{E}^{un}_{I}$ the de Rham cohomology sheaf $\mathbb{H}_{\DR,un}:=\mathbb{H}_{\DR}(\mathbb{E}^{un}_{I})$. Recall from Remark \ref{H1_hodge} that $\mathbb{H}_{1,un}=\mathbb{H}_1(\mathbb{E}^{un}_{I})$ and $\mathbb{H}_{2,un}=\mathbb{H}_2(\mathbb{E}^{un}_{I})$. 

Our goal in this section is, on the one hand, to study an extension of $\mathbb{H}_{\DR,un}$ to $\overline{M^{2}_I}$, and on the other, to analyze the structure of the pull back of the de Rham sheaf to $\Omega$ via $\pi_Y:\Omega \rightarrow \Gamma_Y(I)\backslash \Omega \cong M^{an}_Y$.

\subsection{Extension of the de Rham sheaf to $\overline{M_{I}^2}$}  We start by analyzing the description of the de Rham sheaf $\mathbb{H}_{\DR}$ at an algebraic  cusp  $\mu_i:\Spec(\mathcal{H}((X_i)))\rightarrow M^{2}_I$ for each $1\leq i \leq |\AlgCusps_{I}|$ which is determined by the $\TD$-module $(\phi_i,\lambda_i)$ described in \S6.1. Let $\frac{d}{dX_i}$ be the $A$-linear derivation that sends any power series in $\mathcal{H}((X_i))$ to its derivative with respect to $X_i$.

\begin{definition} For any non-constant $a\in A$, let 
\[
(\phi_i)_{a}(Z_i):=aZ_i+g_{i,1}(X_i)Z_i^{q}+\cdots+g_{i,2\deg(a)}(X_i)Z_i^{q^{2\deg(a)}}\in \mathcal{H}((X_i))[Z_i].
\]
\textit{The false Eisenstein series $E(\mu_i)$ at the cusp $\mu_i:\Spec(\mathcal{H}((X_i))) \rightarrow M^2_{I}$ } is defined by
\[
E(\mu_i):=-\frac{X_i^{2}}{g_{i,2\deg(a)}(X_i)}\frac{d}{dX_i}g_{i,2\deg(a)}(X_i)\in \mathcal{H}((X_i)).
\]
\end{definition}

Using \cite[Lem.~6.7]{Gek90}, one can immediately see that $E(\mu_i)$ is independent of the choice of $a\in A\setminus\mathbb{F}_q$.

\begin{remark} \begin{itemize}
    \item[(i)] The motivation to consider the above definition of the false Eisenstein series at $\mu_i$ is due to the following relation between differential operators on the $t$-expansions of Drinfeld modular forms (see also \cite[(7.19)]{Gek90}): Let $z$ be an analytic parameter on $\Omega$ and for a fractional ideal $\mathfrak{l}$ and $\xi\in \mathbb{C}_{\infty}\setminus \overline{K}$, set $t:=t(z):=\exp_{\xi\mathfrak{l}}(\xi z)^{-1}$. Then for the operator $\frac{d}{dt}$ which sends $\sum_{i=0}^\infty a_i t(z)^i \mapsto \sum_{i=0}^\infty ia_i t(z)^{i-1}$ (with $a_i\in \mathbb{C}_{\infty}$), we have 
    \begin{equation}\label{E:relationoft}
        \partial_z = -\xi t^2
\frac{d}{dt},   \end{equation}where $\partial_z$ is the operator defined in \S4. Then given a projective $A$-module $Y$ of rank two as before, upon choosing $\xi$ and $\mathfrak{l}$ appropriately, the substitution $X_i \mapsto t(z)$ yields the $t$-expansion of $E$ at the cusp $\mu^Y_i$ up to a non-zero multiple from $\overline{K}$. We also note that setting $q(\zz):=e^{2\pi z \sqrt{-1}}$ for an analytic parameter $\zz$ on the upper half plane, \eqref{E:relationoft} may be seen as a characteristic $p$ analogue of the classical relation \[
\frac{d}{d\zz} = (2\pi \sqrt{-1})q\frac{d}{dq}.
\]
    \item[(ii)] When $i=1$, from Lemma \ref{fEis}, we indeed have that $E(\mu_1)$ agrees with 
    \[
    -\sum_{a\in \mathfrak{b}\setminus \{0\}}a\psi_{a}\Big(\frac{1}{X}\Big)^{-1}\in  \mathcal{H}[[X]],
    \] 
    upto a non-zero constant from $\mathcal{H}$, where as in \S6, $\psi$ is the universal Drinfeld $A$-module of rank one over $\mathcal{H}$ and $\mathfrak{b}$ is the integral ideal defined as in \S6.1. 
\end{itemize} 
\end{remark}

\begin{lemma}[{The Hodge Decomposition}] \label{L:basisderham}
    We have 
    \[
    \h_{\DR}(\phi_i) = \h_1(\phi_i) \oplus \h_2(\phi_i)
    \]
   so that the $\mathcal{H}((X_i))$-vector space $\h_1(\phi_i)$ is spanned by the biderivation $\eta_{i,1}:=[a \mapsto a - (\phi_i)_{a}]$ and the $\mathcal{H}((X_i))$-vector space $\h_2(\phi_i)$ is spanned by the biderivation $\eta_{i,2}:= -X_i^{2}\frac{d}{dX_i}\eta_{i,1} - E(\mu_i)\eta_{i,1}$. 
\end{lemma}
\begin{proof} This simply follows from the observation that $\eta_{i,1}$ ($\eta_{i,2}$ respectively) is a reduced (strictly reduced respectively) biderivation.
\end{proof}

In what follows, we set 
\[
\eta'_{i,2}:=\eta_{i,2} + E(\mu_i)\eta_{i,1}\in \h_{\DR}(\phi_i).
\]
Note that $\{\eta_{i,1},\eta'_{i,2}\}$ is also a $\mathcal{H}((X_i))$-basis of $\h_{\DR}(\phi_i)$. The base change matrix sending $\{\eta_{i,1},\eta_{i,2}\}$ to $\{\eta_{i,1},\eta'_{i,2}\}$ is $\begin{pmatrix}
        1 & E(\mu_i) \\
        0 & 1
    \end{pmatrix}$, which lies in $\GL_2(\mathcal{H}((X_i))$. 

Now we are ready to define the extension $\overline{\mathbb{H}_{\DR,un}}$ of $\mathbb{H}_{\DR,un}$ to $\overline{M_{I}^2}$. 
\begin{definition}We denote by $\overline{\mathbb{H}_{\DR,un}}$ the unique locally free sheaf extension of $\mathbb{H}_{\DR,un}$ to $\overline{M^{2}_I}$ such that its formal completion at $\overline{M^{2}_I} \backslash M^2_I$ is determined by the module $\oplus^{n_I}_{i=1} (\mathcal{H}[[X_i]]\eta_{i,1} \oplus \mathcal{H}[[X_i]]\eta_{i,2})$.
\end{definition} 

Note that by Theorem \ref{deRham}, $\mathbb{H}_{\DR,un}$ is a locally free sheaf of rank $2$ over $M^2_I$ which implies the local freeness of $\overline{\mathbb{H}_{\DR,un}}$.

\subsection{The de Rham sheaf on $\Omega$} Let $B$ be an affinoid algebra so that $\Sp(B)$ is an admissible open subset of $\Omega$. Recall the Drinfeld $A$-module $(\mathbb{G}_{a,\Omega},\bold\Psi^{Y})$ from \S2.8 and let $\bold\Psi^Y_{|\Sp(B)}$ denote its restriction to $\Sp(B)$. Let 
\[
\bold\Psi^{Y}_{a|\Sp(B)} = a + \sum^{2\deg(a)}_{i=1} (g_{i,a})_{|\Sp(B)}  \ \tau^i\in B[\tau].
\]

Recall the quotient map $\pi$ from \S2.6. We denote by $\pi_Y:\Omega \rightarrow M_Y$ the composite map \[\Omega \xrightarrow{\pi} \Gamma_Y(I)\backslash \Omega \xrightarrow{i_Y} M_Y.\]
Let $\mathbb{H}_{\DR,Y}$ be the restriction of $\mathbb{H}_{\DR,un}$, after base change with $\mathbb{C}_\infty$, to $M_Y$ and $\mathbb{H}^{an}_{\DR,Y} $ be its analytification. In what follows, we explicitly describe the pull backs $i^\ast_Y(\mathbb{H}_{\DR,Y})=\mathbb{H}^{an}_{\DR,Y}$ and $\pi^\ast_Y(\mathbb{H}_{\DR,Y})$ which are sheaves on $\Gamma_Y(I) \backslash \Omega$ and $\Omega$ respectively. Note that, by Corollary \ref{a12},  $\mathbb{H}^{an}_{\DR,Y} = (\pi^\ast_Y(\mathbb{H}_{\DR,Y}))^{\Gamma_Y(I)}$. Therefore it suffices to analyze the sheaf $\pi^\ast_Y(\mathbb{H}_{\DR,Y})$ on $\Omega$. 

\begin{proposition} The sheaf 
    $\pi^\ast_Y(\mathbb{H}_{\DR,Y})$ is the unique sheaf on $\Omega$ such that for any affinoid subdomain $j:\Sp(B) \rightarrow \Omega$, its section over $\Sp(B)$ is given by the $B$-module $\h_{\DR} (\bold\Psi^{Y}_{|\Sp(B)})$.
\end{proposition}

\begin{proof}
    The morphism $\pi_Y \circ j:\Sp(B) \rightarrow M_Y$ factors through a canonical map $\tilde{j}:\Spec(B) \rightarrow M_Y$. Note that the coherent sheaves $j^\ast(\pi^\ast_Y(\mathbb{H}_{\DR,Y}))$ on $\Sp(B)$ and $(\tilde{j})^{\ast} \mathbb{H}_{\DR,Y} $ on $\Spec(B)$ correspond to the same finitely generated $B$-module, say $N$. In other words, they are obtained by the $\tilde{()}$ operation on $N$, with respect to the corresponding topologies.  But by the affine base change property of de Rham cohomology \cite[Thm.~4.5]{Gek89}, we have  $(\tilde{j})^{\ast}\mathbb{H}_{\DR,Y} = \mathbb{H}_{\DR} (\bold\Psi^{Y}_{|\Sp(B)})$ which implies the desired statement. 
\end{proof}
 
In what follows, we denote the sheaf $\pi^\ast_Y(\mathbb{H}_{\DR,Y})$ by $\mathbb{H}_{\DR}(\bold\Psi^{Y})$. Since $\mathbb{H}_{\DR,Y}$ is a locally free sheaf of rank $2$ on $M_Y$, so is $\mathbb{H}_{\DR}(\bold\Psi^Y)$  on $\Omega$. Furthermore, since $\Omega$ is a Stein space \cite[Thm.~4]{SS91}, $\mathbb{H}_{\DR}(\bold\Psi^Y)$ is generated by its global sections. In our next proposition, we explicitly describe these sections to show that $\mathbb{H}_{\DR}(\bold\Psi^Y)$ is indeed a free sheaf of rank two over $\mathcal{O}_\Omega$.

Let $\eta_{1,B}$ be the local section of $\mathbb{H}_{\DR} (\bold\Psi^{Y}_{|\Sp(B)})$ given by
\[
\eta_{1,B}:=\eta^{(1)}_{B}:=[a \mapsto a - \bold\Psi^{Y}_{a|\Sp(B)}] \in H^{0}(\Sp(B),\mathbb{H}_{\DR} (\bold\Psi^{Y}_{|\Sp(B)})).
\]
Recall the false Eisenstein series $E$ defined in \S4. We further define
the local section $\eta_{2,B}$ of $\mathbb{H}_{\DR} (\bold\Psi^{Y}_{|\Sp(B)})$ given by 
\[
\eta_{2,B}:=\frac{1}{\xi(\mathfrak{g}^{-1}\mathfrak{h})}\partial_z(\eta_{1,B})-E\eta_{1,B} \in H^{0}(\Sp(B),\mathbb{H}_{\DR} (\bold\Psi^{Y}_{|\Sp(B)})).
\]
Note that for affinoid subdomains $j:\Sp(B) \rightarrow \Omega$, the local sections $\eta_{1,B}$ ($\eta_{2,B}$ respectively), as $B$ varies, glue together to form a global section of $\mathbb{H}_{\DR}(\bold\Psi^Y)$, denoted by $\eta_1$ ($\eta_2$ respectively).

\begin{lemma}[{Hodge decomposition}] \label{hodge_2} 
     Let $\eta_1,\eta_2 \in \mathbb{H}_{\DR}(\bold\Psi^Y)(\Omega)$ be as above. Then the natural map \[\mathcal{O}_{\Omega}  \oplus \mathcal{O}_{\Omega}  \rightarrow \mathbb{H}_{\DR}(\bold\Psi^Y)\] sending $(1,0) \mapsto \eta_1$ and $(0,1) \mapsto \eta_2$ is an isomorphism of $\mathcal{O}_\Omega$-sheaves.

\end{lemma}
\begin{proof} The proof is essentially the content of \cite[Prop.~7.7]{Gek90} but we give it for the sake of completeness. Observe that it is enough to show the isomorphism when restricted to any affinoid subdomain $\Sp(T) \rightarrow \Omega$. For all $z \in \Sp(T)$, let $T_z$ denote the localization of $T$ at the maximal ideal corresponding to $z$. This is a local ring and let $\widehat{T_z}$ denote the corresponding completion. We have
\begin{equation}\label{E:maps}
\widehat{T_z} \cong \widehat{\mathcal{O}_{\Omega,z}} \cong \widehat{\mathcal{O}_{\Gamma_Y(I) \backslash \Omega,\pi(z)}} \cong \widehat{\mathcal{O}_{M_Y,\pi_Y(z)}}
\end{equation}
where the first isomorphism follows from \cite[Prop.~4.6.1]{FvdP04}, the second isomorphism from Proposition \ref{a6} and the last one follows since $\Gamma_Y(I) \backslash \Omega = M^{an}_Y$. The completed localization at $z$ of the map $\mathcal{O}_{\Omega}  \oplus \mathcal{O}_{\Omega}  \rightarrow \mathbb{H}_{\DR}(\bold\Psi^Y)$ via \eqref{E:maps} is given by the injective map
\[ 
\widehat{\mathfrak{i}}: \widehat{\mathcal{O}_{M_Y,\pi_Y(z)}} \oplus \widehat{\mathcal{O}_{M_Y,\pi_Y(z)}} \rightarrow \h_{\DR}(\bold\Phi_{\widehat{\mathcal{O}_{M_Y,\pi_Y(z)}}},\widehat{\mathcal{O}_{M_Y,\pi_Y(z)}}) 
\] 
where $\bold\Phi_{\widehat{\mathcal{O}_{M_Y,\pi_Y(z)}}}$ is the pull back of the universal Drinfeld $A$-module via $\Spec(\widehat{\mathcal{O}_{M_Y,\pi_Y(z)}}) \rightarrow M_Y$. As in the proof of \cite[Prop.~7.7]{Gek90}, $\eta_2$ generates  $\h_2(\bold\Phi_{\widehat{\mathcal{O}_{M_Y,\pi_Y(z)}}},\widehat{\mathcal{O}_{M_Y,\pi_Y(z)}})$. Moreover, since, by construction, $\eta_1$ generates  $\h_1(\bold\Phi_{\widehat{\mathcal{O}_{M_Y,\pi_Y(z)}}},\widehat{\mathcal{O}_{M_Y,\pi_Y(z)}})$, we obtain \[
\im(\widehat{\mathfrak{i}})=\h_{1}(\bold\Phi_{\widehat{\mathcal{O}_{M_Y,\pi_Y(z)}}},\widehat{\mathcal{O}_{M_Y,\pi_Y(z)}}) \oplus \h_2(\bold\Phi_{\widehat{\mathcal{O}_{M_Y,\pi_Y(z)}}},\widehat{\mathcal{O}_{M_Y,\pi_Y(z)}}).
\]
 Thus, by \eqref{E:Hodge0}, $\widehat{\mathfrak{i}}$  is surjective and hence it is an isomorphism. We obtain that the cokernel of the map $\mathcal{O}_{\Omega|\Sp(T)}  \oplus \mathcal{O}_{\Omega|\Sp(T)}  \rightarrow \mathbb{H}_{\DR}(\bold\Psi^Y)_{|\Sp(T)}$, being represented by a finitely generated $T$-module, vanishes at the localization by maximal ideals of $T$, hence must be zero. Consequently the map is surjective.
\end{proof}  

Note that, by definition, $\mathbb{H}_{\DR}(\bold\Psi^Y)$ is the pull back of a coherent sheaf on $\Gamma_Y(I)\backslash \Omega$ by $\pi$ and hence is a $\Gamma_Y(I)$-sheaf in the sense of Definition \ref{D: G_sheaf}. In what follows, we explicitly describe this action, which is induced from the action of $\Gamma_Y(I)$ on $(\mathbb{G}_{a,\Omega},\bold\Psi^{Y})$ described at the end of \S2.8.

\begin{proposition} \label{trans} For any $\gamma  = \begin{pmatrix}
	a&b\\
	c&d
	\end{pmatrix}
 \in \Gamma_Y(I)$, let \[\alpha_\gamma:\mathbb{H}_{\DR}(\bold\Psi^Y) = \mathcal{O}_\Omega \eta_1 \oplus \mathcal{O}_\Omega \eta_2 \rightarrow \gamma_\ast \mathbb{H}_{\DR}(\bold\Psi^Y) := (\gamma_\ast\mathcal{O}_\Omega) \gamma_\ast\eta_1 \oplus (\gamma_\ast\mathcal{O}_\Omega )\gamma_\ast\eta_2 \] be the $\Gamma_{Y}(I)$-sheaf structure. The following statements hold.  
 \begin{enumerate}
		\item[(i)] The map $\alpha_\gamma$ can be represented by
		\[\begin{bmatrix}
			\eta_1 \\
			\eta_2 
		\end{bmatrix} \mapsto \begin{bmatrix}
			j(\gamma;-)^{-1} & 0 \\
			0   & j(\gamma;-) 
\end{bmatrix}
		\begin{bmatrix}
			\gamma_\ast\eta_1 \\
			\gamma_\ast\eta_2
		\end{bmatrix}.
		\] and the maps $\mathcal{O}_\Omega \rightarrow \gamma_\ast \mathcal{O}_\Omega$ in the first and second coordinates, sends a function $f \in \mathcal{O}_\Omega(U)$ to $f^\gamma :=[z \mapsto f(\gamma z)] \in \mathcal{O}_\Omega(\gamma^{-1}U)$.
		\item[(ii)] Let $\eta'_2:=\eta_2+E\eta_1$. Then $\eta_2'$ is a global section of $\mathbb{H}_{\DR}(\bold\Psi^Y)$. Moreover,  $\alpha_\gamma$ can be also represented by
		\[ \begin{bmatrix}
			\eta_1 \\
			\eta'_2 
		\end{bmatrix} \mapsto \begin{bmatrix}
			j(\gamma;-)^{-1} & 0 \\
			-c   & j(\gamma;-) 
			\end{bmatrix}
		\begin{bmatrix}
			\gamma_\ast\eta_1 \\
			\gamma_\ast\eta'_2
		\end{bmatrix}.
		\]
	\end{enumerate}
	\label{7.1}
\end{proposition}
\begin{proof} Part (i) is a consequence of \cite[(7.16)]{Gek90}. Part (ii) follows from (i) and the functional equation \eqref{4.5} of the false Eisenstein series $E$.
\end{proof}

Observe that $\mathcal{O}_\Omega \eta_1 \subset \mathbb{H}_{\DR}(\bold\Psi^Y)$ is a $\Gamma_Y(I)$-subsheaf. Before we state our next lemma, recall that $\omega_{un}=\Lie(\mathbb{E}^{un}_I)^{\vee}$ and $\omega_Y$ is its restriction to $M_Y$ after base change with $\mathbb{C}_\infty$.

\begin{lemma}\label{ana_H1_hodge} We have
    $\pi^{\ast}_Y(\omega_{Y}) \cong \mathcal{O}_\Omega \eta_1$ as $\Gamma_Y(I)$-sheaves. 
\end{lemma}

\begin{proof}
     We aim to show that $\pi^{\ast}_Y(\omega_{Y})$ is isomorphic to $\mathcal{O}_\Omega$ with a $\Gamma_Y(I)$-structure given by multiplication by $j(\gamma,-)^{-1}$. Recall the notation in \S2.8 and let $g \in \GL_2(\widehat{A})$ be such that $\pi_g=\pi_Y$. By Proposition \ref{2.4}, we see that $\pi^{\ast}_Y(\omega_{Y}) \cong \mathcal{O}_\Omega$. Using the commutative diagram in the proof of \cite[Lem.~10.5]{BBP21}, we obtain
\begin{equation*}
\begin{tikzcd}
         \pi^{\ast}_{Y}(\omega_{Y}) \arrow[r,"\sim"] \arrow[d,"="] & 
         \mathcal{O}_{\Omega} \arrow[d,"\text{multiplication by }j(\gamma\text{,}-)"] \\
         \gamma^{\ast} \pi^{\ast}_{Y}(\omega_{Y}) \arrow[r,"\sim"] &
         \gamma^{\ast} \mathcal{O}_{\Omega}=\mathcal{O}_{\Omega}.
\end{tikzcd}
\end{equation*}
Hence under the trivialization $\pi^{\ast}_Y(\omega_{Y}) \cong \mathcal{O}_\Omega$, we see that $\gamma^{\ast} \mathcal{O}_{\Omega}=\mathcal{O}_{\Omega} \rightarrow \mathcal{O}_{\Omega}$ is the map given by the multiplication by $j(\gamma,-)^{-1}$. Equivalently, the map $ \mathcal{O}_{\Omega} \rightarrow \gamma_{\ast} \mathcal{O}_{\Omega}=\mathcal{O}_{\Omega} $ obtained by adjunction, is also the multiplication by $j(\gamma,-)^{-1}$, finishing the proof of the lemma. 
\end{proof}

\begin{definition} We define $\mathbb{H}_{\DR}(\overline{\bold\Psi}^Y):=\mathbb{H}_{\DR}(\bold\Psi^Y)^{\Gamma_Y(I)}$.  Furthermore, we set  $\omega(\bold\Psi^Y):=\mathcal{O}_\Omega \eta_1$ and $\omega(\overline{\bold\Psi}^Y):= \omega(\bold\Psi^Y)^{\Gamma_Y(I)}$. 
\end{definition}

\begin{remark}
    By Corollary \ref{a12}, $\mathbb{H}_{\DR}(\overline{\bold\Psi}^Y) = i^{\ast}_Y(\mathbb{H}_{\DR,Y}) = \mathbb{H}^{an}_{\DR,Y}$.  Similarly $\omega(\overline{\bold\Psi}^Y) = \omega^{an}_Y$. 
\end{remark}

We conclude this section with the following useful lemma.
 
\begin{lemma}\label{derham_2}
    Let $r$ and $k$ be non-negative integers so that $k\geq r$. Then there exists a canonical isomorphism \[\Sym^r(\mathbb{H}_{\DR}(\overline{\bold\Psi}^Y)) \otimes \omega(\overline{\bold\Psi}^Y)^{\otimes (k-r)} \cong (\Sym^r(\mathbb{H}_{\DR}(\bold\Psi^Y)) \otimes \omega(\bold\Psi^Y)^{\otimes (k-r)})^{\Gamma_Y(I)}. \]
\end{lemma}
\begin{proof}
    By definition, we have \[\Sym^r(\mathbb{H}_{\DR}(\overline{\bold\Psi}^Y)) \otimes \omega(\overline{\bold\Psi}^Y)^{\otimes (k-r)} \cong \Sym^r(\mathbb{H}_{\DR}(\bold\Psi^Y)^{\Gamma_Y(I)}) \otimes (\omega(\bold\Psi^Y)^{\Gamma_Y(I)})^{\otimes (k-r)}.\] By Corollary \ref{a12}, applying $()^{\Gamma_Y(I)}$ commutes with $\Sym^r$ and $\otimes$, and hence we obtain the result.  
\end{proof}
\section{Algebraic nearly holomorphic Drinfeld modular forms} For the convenience of the reader, we recall our notation from \S6. Recall the  projective $A$-module $Y$ given as in \eqref{E:defofY} and note from the beginning of  \S3, without loss of generality, we assume that $\mathfrak{g}$ and $\mathfrak{h}$ are integral ideals of $A$. Let $I$ be an ideal of $A$ such that $|V(I)|\geq 2$.  Let $M^2_{I,\mathbb{C}_{\infty}}=\Spec(\mathbb{C}_{\infty})\times_{\Spec(A)} M_I^2$ and set $(M^{2}_{I,\mathbb{C}_\infty})^{an}$ to be the analytification of $M^2_{I,\mathbb{C}_{\infty}}$. Let $M_Y \subset (M^{2}_{I,\mathbb{C}_\infty})^{an}$ be the connected component of $M^2_{I,\mathbb{C}_{\infty}}$ so that $M_{Y}(\mathbb{C}_\infty) = \Gamma_{Y}(I) \setminus \Omega$. We also fix an embedding $\mathcal{H} \rightarrow \mathbb{C}_\infty$ so that $M_Y =M^{2}_I \times_{\mathcal{H}} \mathbb{C}_\infty$ and $\overline{M_Y}:=\overline{M^2_I} \times_\mathcal{H} \mathbb{C}_\infty$. We further denote by $\omega_Y$ the restriction of $\omega_{un}$, after base change with $\mathbb{C}_\infty$, to $M_Y$.

 Let us denote by $\overline{\omega_{un}}$ the unique line bundle over $\overline{M^{2}_I}$ such that the restriction of $\overline{\omega_{un}}$ to $M^2_I$ is $\omega_{un}$ and at each algebraic cusp $\mu_i:\Spec(\mathcal{H}((X_i))) \rightarrow M^{2}_I$, its formal completion at the corresponding point in $\overline{M^2_I} \backslash M^2_I$ (Remark \ref{R:uniTD}) is $\mathcal{H}[[X_{i}]]dZ_{i}$. Furthermore, we denote the restriction of $\overline{\omega_{un}}$ to $M_Y$ by $\overline{\omega_Y}$ so that for any cusp $\mu^{Y}_i \in \AlgCusps^{Y}_I$, its formal completion at the cusp in $\overline{M_Y} \backslash M_Y$ corresponding to $\mu^{Y}_i$, is given by $\mathbb{C}_\infty[[X_i]]dZ_i$. In this section, we describe the nearly holomorphic Drinfeld modular  forms as the global sections of the sheaf $\overline{\mathcal{H}^{r}_k}:= \Sym^{r}(\overline{\mathbb{H}_{\DR,un}} )\otimes \overline{\omega_{un}}^{\otimes (k-r)}$ pulled back to the appropriate component of $M^{2}_{I,\mathbb{C}_\infty}$. To achieve our goal, in what follows, we recall the sheaf $\mathbb{H}_{\DR,Y}$ ($\overline{\mathbb{H}_{\DR,Y}}$ respectively) which is the pull back of $\mathbb{H}_{\DR,un}$ ($\overline{\mathbb{H}_{\DR,un}}$ respectively), after base change with $\mathbb{C}_\infty$, to $M_Y$ ($\overline{M_Y}$ respectively) Moreover, we let
\[
\mathcal{H}^r_{k,Y}:= \Sym^r(\mathbb{H}_{\DR,Y}) \otimes \omega^{\otimes (k-r)}_Y \ \ \text{ and } \ \ 
\overline{\mathcal{H}^r_{k,Y}}:= \Sym^r(\overline{\mathbb{H}_{\DR,Y}}) \otimes \overline{\omega}^{\otimes (k-r)}_Y.
\]

Let $\mathcal{WN}_k^{\leq r}(\Gamma_Y(I))$ be the $\mathbb{C}_{\infty}$-vector space of weak nearly holomorphic Drinfeld modular forms of weight $k$ and depth less than or equal to $r$ for $\Gamma_Y(I)$.

\begin{remark}
     The construction of $\overline{\omega_{un}}$ above and $\overline{\mathbb{H}_{1,un}}$ from \S7.1, imply that the isomorphism $\mathbb{H}_{1,un} \cong \omega_{un}$ from Remark \ref{H1_hodge} extends to
     a natural isomorphism $\overline{\mathbb{H}_{1,un}} \cong \overline{\omega_{un}}$.  
\end{remark}

\begin{theorem}\label{T:weak}
    There is a natural isomorphism of $\mathbb{C}_\infty$-vector spaces \[H^{0}( \Gamma_{Y}(I) \setminus \Omega,\mathcal{H}^{r,an}_{k,Y}) \cong \mathcal{WN}^{\leq r}_{k}(\Gamma_Y(I)).\]
	In particular, the analytification morphism induces a canonical injective map \[H^{0}(M_Y,\mathcal{H}^{r}_{k,Y}) \hookrightarrow \mathcal{WN}^{\leq r}_{k}(\Gamma_Y(I)).\]
\end{theorem}
\begin{proof} By Lemma \ref{ana_H1_hodge} and Lemma \ref{derham_2}, we have
 \[
 H^{0}( \Gamma_{Y}(I) \setminus \Omega,\mathcal{H}^{r,an}_{k,Y}) = H^{0}(\Omega,\Sym^{r}(\mathbb{H}_{\DR}(\bold\Psi^{Y})) \otimes \omega(\bold\Psi^{Y})^{\otimes (k-r)})^{\Gamma_Y(I)}.
 \]
  Note that any section $s \in H^{0}(\Omega,\Sym^{r}(\mathbb{H}_{\DR}(\bold\Psi^{Y}) \otimes \omega(\bold\Psi^{Y})^{\otimes (k-r)})^{\Gamma_Y(I)}$ may be written in the form $s=\sum^{r}_{l=0} f_l \eta_1^{\otimes (k-l)} (\eta_2')^{\otimes l}$  for some rigid analytic functions $f_l$ and by using the $\mathcal{O}_{\Omega}$-basis $\{\eta_1,\eta_2'\}$ of $\mathbb{H}_{\DR}(\bold\Psi^{Y})$  from Proposition \ref{7.1}(ii). Moreover, it satisfies the transformation property 
	 \begin{equation}\label{E:weakforms}
	 \sum^{r}_{l=0} f_l(\gamma\cdot z) (j(\gamma,z)^{-1}\eta_1)^{\otimes (k-l)} (-c\eta_1 + j(\gamma;z)\eta_2')^{\otimes l} = \sum^{r}_{l=0} f_l(z) \eta_1^{\otimes (k-l)}(\eta_2')^{\otimes l}
	 \end{equation}
	  for all $\gamma \in \Gamma_Y(I)$ and $z\in \Omega$. For each $\ell\geq 0$, comparing the coefficients of $\eta_1^{\otimes (k-\ell)}{(\eta_2')}^{\otimes \ell}$ on both sides of \eqref{E:weakforms}, one can see that 
	  \[
	  j(\gamma,z)^{-k}\sum_{u=0}^{r-\ell}\binom{\ell+u}{u}\Big(\frac{-c}{j(\gamma;z)}\Big)^{u}j(\gamma;z)^{2(\ell+u)}f_{\ell+u}(\gamma\cdot z)=f_{\ell}(z)
	  \]
	 which, combining with Lemma \ref{L:1}, implies that the function $\sum^{r}_{l=0} \frac{f_l}{(\Id-\varphi)^{l}}$ on $\Omega^{\varphi}(M)$ is a weak nearly holomorphic Drinfeld modular form of weight $k$ and depth less than or equal to $r$ for $\Gamma_Y(I)$. Consequently there exists a well-defined map 
	 \[
	 H^{0}(\Omega,\Sym^{r}(\mathbb{H}_{\DR}(\bold\Psi^Y)) \otimes \omega(\bold\Psi^{Y})^{\otimes (k-r)})^{\Gamma_Y(I)} \rightarrow \mathcal{WN}^{\leq r}_{k}(\Gamma_Y(I))
	 \]
	 sending
	 $
	 s=\sum^{r}_{l=0} f_l \eta_1^{\otimes (k-l)} (\eta_2')^{\otimes l} \mapsto \sum^{r}_{l=0} \frac{f_l}{(\Id-\varphi)^{l}}. $ 
	 Note that it is injective by virtue of Theorem \ref{T:1}(ii). Since, for a given element $\sum_{l=0}^{r} \frac{h_l}{(\Id - \varphi)^l}\in \mathcal{W}\mathcal{N}_{k}^{\leq r}(\Gamma_Y(I))$, from the discussion above, we obtain a $\Gamma_Y(I)$-invariant section $\sum^{r}_{l=0} h_l \eta_1^{\otimes (k-l)} (\eta_2')^{\otimes l}$ of $H^{0}(\Omega,\Sym^{r}(\mathbb{H}_{\DR}(\bold\Psi^Y)) \otimes \omega(\bold\Psi^{Y})^{\otimes (k-r)})$, the surjectivity also follows and it finishes the proof of the theorem. 
\end{proof}

Before stating the main theorem of this section, analogous to Definition \ref{6.5}, we define the $t_{\mathfrak{c}_i}^{\mathfrak{c}}$-expansion and $t$-expansion of a nearly holomorphic Drinfeld modular form at an algebraic cusp. For each $1\leq i \leq n_{I}$, recall the Drinfeld $A$-module $\phi_{i,Y}$ corresponding to an algebraic cusp $\mu_{i,Y}:\Spec(\mathbb{C}_\infty((X_i))) \rightarrow M_Y$. In what follows, by abuse of notation in \S7.1, we denote by $\{\eta_{i,1},\eta'_{i,2}\}$ the $\mathbb{C}_{\infty}((X_i))$-basis for  $\mathbb{H}_{\DR}(\phi_{i,Y})$.
\begin{definition}
    Let $k$ and $r$ be non-negative integers such that $k\geq r$.  Let $F \in H^{0}(M_Y,\mathcal{H}^r_{k,Y})$. Then by Remark \ref{H1_hodge}, we have $ \mathbb{H}_1(\phi_{i,Y}) \cong \omega(\phi_{i,Y}) $. Thus one can identify $dZ_i$ with $\eta_{i,1}$ and hence
  \begin{equation}\label{E:decomNHMF}
(\mu^{Y}_i)^\ast(F)\in \bigoplus^r_{j=0} \mathbb{C}_\infty((X_i))(\eta_{i,1})^{\otimes (k-r+j)} \otimes (\eta'_{i,2})^{\otimes (r-j)}.
\end{equation}
 Then there exists an $(r+1)$-tuple of unique Laurent series $\{P^{(j)}_F(X_i)\}_{0 \leq j \leq r}$ such that 
    \[
 (\mu^{Y}_i)^\ast(F)=\sum^r_{j=0}P^{(j)}_F(X_i)(\eta_{i,1})^{\otimes (k-r+j)} \otimes (\eta'_{i,2})^{\otimes (r-j)}.
   \] 
We call the tuple $\{P^{(j)}_F(X_i)\}_{0 \leq j \leq r}$   \textit{the $t$-expansion of $F$ at the cusp $\mu^{Y}_i$}. 
\end{definition}

Applying Proposition \ref{u_expn_1} to each component of the right hand side of \eqref{E:decomNHMF}, we immediately obtain our next lemma.

\begin{lemma}\label{L:nhf_t_expn} Let $b_i\in \Cusps^Y_{I}$ correspond to $\mu^Y_i$ as in Proposition \ref{u_expn_0}. Consider an element $F \in H^{0}(M_Y,\mathcal{H}^r_{k,Y})$. Then  each Laurent series in the $t_{b_i}$-expansion of $F$ has no principal part if and only if each Laurent series in the $t$-expansion of $F$ at $\mu^Y_i$ has no principal part.     
\end{lemma}
We are now ready to state the main result of this section.
\begin{theorem}\label{T:main1}
	The map constructed in Theorem \ref{T:weak} induces an isomorphism of $\mathbb{C}_{\infty}$-vector spaces \[H^{0}(\overline{M_Y},\overline{\mathcal{H}^{r}_{k,Y}}) \cong \mathcal{N}^{\leq r}_{k}(\Gamma_{Y}(I)).\]
\end{theorem}

\begin{proof} We first claim that the injection 
\[
H^{0}(\overline{M_Y},\overline{\mathcal{H}^{r}_{k,Y}})\hookrightarrow H^{0}(M_Y,\mathcal{H}^{r}_{k,Y}) \xhookrightarrow{\text{Theorem  }\ref{T:weak}}\mathcal{WN}^{\leq r}_{k}(\Gamma_{Y}(I))
\]
has image in $\mathcal{N}^{\leq r}_k(\Gamma_Y(I))$. This is equivalent to showing that for any $F \in H^{0}(\overline{M_Y},\overline{\mathcal{H}^{r}_{k,Y}})$ which is regarded as an element of $\mathcal{WN}^{\leq r}_k(\Gamma_Y(I))$,  its $t_{b_i}$-expansion has no principal part for all $1 \leq i \leq n_I$. By Lemma \ref{L:nhf_t_expn}, each Laurent series in the $t_{b_i}$-expansion of $F$ has no principal part if and only if each Laurent series in its $t$-expansion at $\mu^Y_i$ has no principal part. But the latter is a consequence of the construction of $\overline{\mathcal{H}_{k,Y}^r}$, finishing the proof of the claim. Thus we establish an injective map $H^{0}(\overline{M_Y},\overline{\mathcal{H}^{r}_{k,Y}}) \hookrightarrow \mathcal{N}^{\leq r}_{k}(\Gamma_{Y}(I))$. To show the surjectivity, consider $G \in \mathcal{N}^{\leq r}_{k}(\Gamma_{Y}(I))$. By definition, at each cusp $b_i$, each Laurent series in the  $t_{b_i}$-expansion of $G$ has no principal part. Hence its $t_{b_i}$-expansion defines a section in $H^{0}\Big(\mathbb{C}_\infty\Big\langle \frac{t_{b_i}}{\theta^e}\Big\rangle, g^\ast\overline{\mathcal{H}^r_{k,Y}}\Big)$ where $e$ is some non-negative integer and $g:\Sp\Big(\mathbb{C}_\infty\Big\langle \frac{t_{b_i}}{\theta^e}\Big\rangle\Big) \rightarrow \overline{\Gamma_Y(I) \backslash \Omega}$ is the neighborhood at the cusp $b_i$ as described in \S2.6. Letting $\overline{(\mathcal{H}^r_{k,Y})^{an}}$  be the analytification of $\overline{\mathcal{H}^r_{k,Y}}$, we see that these $n_I$-many sections can be glued together with $G$ (regarded as an element in $\h^0(\Gamma_Y(I)\backslash \Omega,\mathcal{H}^{r,an}_{k,Y})$ via Theorem \ref{T:weak}) to a section in $H^{0}(\overline{\Gamma_Y(I)\backslash \Omega},\overline{(\mathcal{H}^r_{k,Y})^{an}}) = H^{0}(\overline{(M_Y)^{an}},\overline{(\mathcal{H}^r_{k,Y})^{an}})$. By rigid analytic GAGA, the latter is isomorphic to $H^{0}(\overline{M_Y},\overline{\mathcal{H}^r_{k,Y}})$ and hence it implies the surjectivity of the map $H^{0}(\overline{M_Y},\overline{\mathcal{H}^{r}_{k,Y}}) \hookrightarrow \mathcal{N}^{\leq r}_{k}(\Gamma_{Y}(I))$, finishing the proof of the theorem.
\end{proof}

Letting $r=0$ in Theorem \ref{T:main1} (see Remark \ref{R:mf_nhdf}), we have the following algebraic description of Drinfeld modular forms, which was originally obtained by Goss.
	
	\begin{corollary}[{\cite[Prop.~1.79]{Gos80}}]\label{alg_forms} We have an isomorphism of $\mathbb{C}_{\infty}$-vector spaces $$H^{0}(\overline{M_Y},\overline{\omega_{Y}}^{\otimes k}) \cong \mathcal{M}_k(\Gamma_Y(I)).$$
	\end{corollary}

\end{document}